\documentclass[11pt]{amsart}
\usepackage{fullpage,verbatim,amssymb}
\usepackage{hyperref}
\usepackage[active]{srcltx}
\usepackage[cmtip,arrow,matrix,curve,tips,frame]{xy}
\usepackage[usenames,dvipsnames]{color}

\makeatletter
\let\@@pmod\mod
\DeclareRobustCommand{\mod}{\@ifstar\@pmods\@@pmod}
\def\@pmods#1{\mkern4mu({\operator@font mod}\mkern 6mu#1)}
\makeatother

\definecolor{blue}{rgb}{0,0,1}
\definecolor{red}{rgb}{1,0,0}
\definecolor{green}{rgb}{0,.6,.2}
\definecolor{purple}{rgb}{1,0,1}

\long\def\red#1\endred{\textcolor{red}{#1}}
\long\def\blue#1\endblue{\textcolor{blue}{#1}}
\long\def\purple#1\endpurple{\textcolor{purple}{ #1}}
\long\def\green#1\endgreen{\textcolor{green}{#1}}

\newcommand{\bsl}{\backslash}

\newcommand{\Z}{\mathbb{Z}}
\newcommand{\Q}{\mathbb{Q}}
\newcommand{\R}{\mathbb{R}}
\newcommand{\C}{\mathbb{C}}

\newcommand{\HH}{\mathbb{H}}

\newcommand{\scrA}{\mathcal{A}}
\newcommand{\scrM}{\mathcal{M}}
\newcommand{\scrE}{\mathcal{E}}

\newcommand{\scrL}{\mathcal{L}}
\newcommand{\scrP}{\mathcal{P}}
\newcommand{\scrO}{\mathcal{O}}

\newcommand{\cuspa}{\mathfrak{a}}

\DeclareMathOperator{\SL}{SL}

\DeclareMathOperator{\GL}{GL}

\DeclareMathOperator{\OD}{OD}

\DeclareMathOperator{\ord}{{ord}}

\DeclareMathOperator{\vol}{{vol}}

\newcommand{\cusp}{{\rm cusp}}
\newcommand{\cont}{{\rm cont}}
\newcommand{\res}{{\rm res}}
\newcommand{\Res}{{\rm Res}}
\newcommand{\Span}{{\rm Span}}

\newcommand{\intg}{{\rm int}}

\newcommand{\new}{{\rm new}}

\newcommand{\sm}{\left(\begin{smallmatrix}}
\newcommand{\esm}{\end{smallmatrix}\right)}
\newcommand{\bpm}{\begin{pmatrix}}
\newcommand{\ebpm}{\end{pmatrix}}

\newtheorem{theorem}{Theorem}
\newtheorem{lemma}[theorem]{Lemma}
\newtheorem{proposition}[theorem]{Proposition}
\newtheorem{corollary}[theorem]{Corollary}

\theoremstyle{remark}

\newtheorem{remark}[theorem]{Remark}

\numberwithin{theorem}{section}
\numberwithin{equation}{section}

\title{Second moments of Rankin-Selberg convolutions and Shifted Dirichlet Series}
\author{Jeff Hoffstein and Min Lee}
\thanks{M.~L.\ was supported by Royal Society University Research Fellowship ``Automorphic forms,
L-functions and trace formulas''.}
\address{Mathematics Department, Brown University, Providence RI 02912, USA }
\email{\tt jhoff@math.brown.edu}
\address{School of Mathematics, University of Bristol, Bristol, BS8 1TW, UK}
\email{\tt min.lee@bristol.ac.uk}

\begin{document}
\maketitle

\begin{abstract}
In this paper we work over $\Gamma_0(N)$, for any $N$, 
and write the spectral moment of a product of two distinct Rankin-Selberg convolutions at a general point on the critical line $\frac{1}{2} +it$ as a main term plus a sharp error term in the $t$ aspect and the spectral aspect.   
As a result we obtain hybrid Weyl type subconvexity results in the $t$ and spectral aspects.  
Also, for fixed modular forms $f$, $g$ of even weight  $k \ge 4$ we show there exists 
a Maass cusp form $u_j$ such that $L(1/2, f\times u_j)$, $L(s, \bar{g}\times \overline{u_j})$ are simultaneously non-zero.
\end{abstract}

\tableofcontents

\section{Introduction}\label{s:intro}

The object of this paper is to study the spectral moment of a product of two Rankin-Selberg convolutions 
%of $f$ and $g$ 
of holomorphic newforms with Maass cusp forms in $L^2(\Gamma_0(N) \bsl \HH)$. 
We will begin, in the following section, by introducing notations for the spectral decomposition of the space . 

\subsection{The basis for $L^2(\Gamma_0(N)\bsl \HH)$}\label{ss:basis_L2}
Let $\HH=\{x+iy\;:\; x\in \R, \; y>0\}$ be the Poincar\'e upper half plane. 
For any $\varphi_1, \varphi_2\in L^2(\Gamma_0(N) \bsl \HH)$, we define the Petersson inner product
\begin{equation}
\left<\varphi_1, \varphi_2\right> = \iint_{\Gamma_0(N)\bsl \HH} \varphi_1(z)\overline{\varphi_2(z)} \; \frac{dx\; dy}{y^2}.
\end{equation}

Let $\{u_j\}_{j\geq 0}$ be an orthonormal basis (with respect to the Petersson inner product) 
for the discrete part of the spectrum of the Laplace operator on $L^2\left(\Gamma_0(N)\bsl \HH\right)$.   
We will assume the basis is simultaneously diagonalized with respect to Hecke operators corresponding to any integer relatively prime to $N$. 
Let $s_j(1-s_j)$, for $s_j\in \C$, be an eigenvalue of the Laplace operator for $u_j$ for each $j\geq 0$.   
Then we have the Fourier expansion
\begin{equation}\label{e:normalized_maassform}
u_j(z) = \sum_{n\neq 0} \rho_j(n)\sqrt{y}K_{ir_j}(2\pi|n|y)e^{2\pi inx}\,.
\end{equation}
Here $s_j=\frac{1}{2}+ir_j$ for $r_j\in \R$ or $\frac{1}{2}<s_j<  \frac34$ (if exists, for finitely many $j$'s), 
and  
\begin{equation}
K_{\nu}(y)
= \frac{1}{2} \int_0^\infty e^{-\frac{1}{2} y \left(t+t^{-1}\right)} t^{\nu} \; \frac{dt}{t},  
\end{equation}
for $\nu\in \C$ and $y>0$, is the $K$-Bessel function. 
We further diagonalize the basis under the action $z\mapsto -\bar{z}$; 
there exist $\epsilon_j\in \{0, 1\}$ such that $u_j(-\bar{z}) = (-1)^{\epsilon_j} u_j(z)$.
We then have $\rho_j(-m)=(-1)^{\epsilon_j} \rho_j(m)$ for any $m\neq 0$. 

The Eisenstein series for $\Gamma=\Gamma_0(N)$ are indexed by the cusps $\cuspa\in \Q\cup\{\infty\}$.
For each cusp $\cuspa\in \Q\cup\{\infty\}$, let $\sigma_\cuspa\in \SL_2(\R)$, 
with $\sigma_\cuspa \infty = \cuspa$, 
be a scaling matrix for the cusp $\cuspa$, i.e., 
$\sigma_\cuspa$ is the unique matrix (up to right translations)
such that 
$\sigma_\cuspa \infty=\cuspa$
and 
\begin{equation}
\sigma_\cuspa^{-1}\Gamma_\cuspa \sigma_\cuspa 
= \Gamma_\infty = \left\{\left. \pm \bpm 1 & b\\ 0 & 1\ebpm \;\right|\; b\in \Z\right\},
\end{equation}
where
\begin{equation}
\Gamma_\cuspa 
= \left\{\gamma\in \Gamma\;|\; \gamma \cuspa = \cuspa\right\}.
\end{equation}
For a cusp $\cuspa$, define the Eisenstein series at the cusp $\cuspa$ to be
\begin{equation}\label{e:Ea}
E_\cuspa(z, s) 
= \sum_{\gamma\in \Gamma_\cuspa \bsl \Gamma} \Im(\sigma_\cuspa^{-1}\gamma z)^s, 
\end{equation}
with the following Fourier expansion:
\begin{equation}\label{m:EisensteinSeries_fex}
E_\cuspa \left(z, s\right) 
= \delta_{\cuspa, \infty} y^{s}
+ \tau_\cuspa \left(s, 0\right) y^{1-s}
+ \sum_{n\neq 0} \tau_\cuspa \left(s, n\right) \sqrt y K_{s-\frac{1}{2}}(2\pi |n|y)e^{2\pi inx}.
\end{equation}
We follow \cite{Y19} and give an explicit description of the Eisenstein series $E_{\cuspa}(z, s)$ in \S\ref{ss:factorization_scrL}. 

\subsection{The object of study: Second moment of Rankin-Selberg convolutions}\label{ss:notation2}

In this paper we study an average over the spectrum of $L$-functions of the degree $8$ Euler product
which arises from the product of the two Rankin-Selberg convolutions, 
with arguments on the critical line, of $f$ with $u_j$ and $g$ with $u_j$.   
Our main objective is to write this as a main term plus a sharp error term.   
One consequence of this is that we derive a simultaneous non-vanishing result given in Corollary~\ref{cor:nonvanishing_1}. 
We also provide a proof for general $N$ of some subconvexity estimates for Rankin-Selberg $L$-series 
that are uniform in the spectral aspect, and $t$. 
These were proved in the $N=1$ case in \cite{Ivic-Jutila}, and without the $t$ dependence by \cite{LLY06}, 
and will be discussed below.

The method does not allow us to keep track of the  explicit  dependence on the weight $k$ and the level $N$. This is due to a gap in the literature.   
A crucial ingredient in our work is an estimate for sums of triple products, 
due to Bernstein, and Reznikov (see \cite[Proposition~4.1]{HHR} for a statement): 
\begin{equation}\label{e:innerbound}
\sum_{|r_j|< T} |\left<u_j, V_{f, g}\right>|^2 e^{\pi|r_j|}
+\sum_\cuspa \int_{-T}^T |\left<E_{\cuspa}(*,1/2 + ir), V_{f, g}\right>|^2 e^{\pi |r|}\; dr
\ll_{N,k}T^{2k}\log(T).
\end{equation}
%Here the sum is over the Eisenstein series at the cusps of $\Gamma_0(N)$ 
%(these will be defined in \eqref{e:Ea} below), 
Here $V_{f,g} (z)= \overline{f(z)}g(z) y^k$.  
Unfortunately the implied constant has an unknown dependence on $N$ and $k$.   
Consequently, all implied constants in this paper will be assumed to depend on $\epsilon >0$ if appropriate, and on $N$ and $k$.

We deal with the case of general $N$, not necessarily square free.  
However, the notation becomes sufficiently more complicated,in the case of $N>1$ 
(mostly because of the Mass cusp oldforms and cusps of $\Gamma_0(N)$),
that we feel it would be clearer to first present the results for $N=1$. 
%followed by the precise result for $N$ square free. 
The general case, given in \S\ref{ss:second_N}, has exactly the same form as the level $N=1$ case \S\ref{ss:main_1},
with dependence on the primes dividing $N$.
% however, the formulas are simply more complicated. 
   
We again remark that, to make notation less cumbersome, 
throughout this paper every implied constant in a $\ll$ or $\scrO$ expression 
will be assumed to depend on $N$, $k$ and, if relevant, $\epsilon >0$.

Let $f$ and $g$ be weight $k$ holomorphic newforms for $\Gamma_0(N)$ with the following Fourier expansions: 
\begin{align}
& f(z) =\sum_{n=1}^\infty a(n) e(nz) \text{ and } \qquad A(n) = a(n) n^{-\frac{k-1}{2}}, \label{e:f_Fourier}
\\ & g(z) = \sum_{n=1}^\infty b(n) e(nz) \text{ and } \qquad B(n) = b(n) n^{-\frac{k-1}{2}}. \label{e:g_Fourier}
\end{align}
Here we assume that $k$ is an even positive integer and $k\geq 4$. 

With the spectral decomposition in \S\ref{ss:basis_L2}, the Rankin-Selberg convolution of the cusp form $f$ 
with an $L^2$-normalized Maass cuspform $u_j$, 
and $f$ with the Eisenstein series $E_{\cuspa}$ at each cusp $\cuspa$, is written as
\begin{equation}\label{RSu}
\scrL(s, f\times u_j) = \zeta^{(N)}(2s) \sum_{m=1}^\infty \frac{A(m) \rho_j(m)}{m^s}
\end{equation}
and
\begin{equation}\label{RSE}
\scrL_{\cuspa}(s, ir; f) = \zeta^{(N)}(2s)\sum_{m=1}^\infty \frac{A(m) \tau_{\cuspa}(1/2+ir, m)}{m^s}.
\end{equation}
Here $\zeta^{(N)}(s) = \prod_{p\mid N}(1-p^{-s}) \zeta(s)$. 

Let us remind that we take the orthonormal baiss $\{u_j\}_{j\geq 0}$ simultaneously diagonalized with respect to Hecke operators $T_n$, $\gcd(n, N)=1$. 
When $u_j$ is a newform of level $N$, we choose $u_j$ to be an eigenfunction of all Hecke operators $T_n$ 
(when $n\mid N$, the traditional notation is $U_n$, but we use $T_n$ for any positive integer $n$)
with eigenvalues $\lambda_j(n)$ for $n\in \Z_{\geq 1}$.
Then $\rho_j(n)=\rho_j(1)\lambda_j(n)$ and we have 
\begin{equation}\label{e:scriptL}
\scrL(s, f\times u_j) 
%= \zeta^{(N)}(2s) \sum_{m=1}^\infty \frac{A(m) \rho_j(m)}{m^s}
= \rho_j(1)\zeta^{(N)}(2s) \sum_{m=1}^\infty \frac{A(m)\lambda_j(m)}{m^s}= \rho_j(1)L(s, f\times u_j).
\end{equation}
%where 
%\begin{equation}
%L(s, f\times u_j) = \zeta^{(N)}(2s) \sum_{m=1}^\infty \frac{A(m)\lambda_j(m)}{m^s}.
%\end{equation}

Recall that by \cite[Appendix]{H-L}, when the $u_j$ are normalized to have $L^2$-norm 1 then
\begin{equation}\label{weight}
(1+| r_j|)^{-\epsilon} \ll |\rho_j(1)|^2 e^{-{\pi |r_j|}} \ll\log(1+| r_j|),
\end{equation}
so $ |\rho_j(1)|^2(\cosh(\pi r_j))^{-1}$ is close to a constant. 
The relationship between the two definitions of the Rankin-Selberg convolution is not quite as straightforward 
when $u_j$ is not a newform
%for general $N>1$ 
as it is in \eqref{e:scriptL}.  This is why this notation $\scrL(s, f\times u_j)$ is introduced.  
We will only use $\scrL(s, f\times u_j)$ for general level $N>1$.

We now introduce the test function we will use for our spectral expansion.
Following \cite[\S9]{Iwa02}, 
let $h(r)$ be a function satisfying the following conditions:
\begin{enumerate}\label{hcond}
\item $h(r)$ is even;
\item $h(r)$ is holomorphic in the strip;
$|\Im(t)|\leq 1/2 +\epsilon$;
\item $h(r)\ll (|t|+1)^{-2-\epsilon}$ in the strip;
\item $h(\pm i/2)=0$.
\end{enumerate}
For our applications, we would like $h(r)$ to have another property: 
Fix $T\gg 1$ and $\frac{1}{3}< \alpha < \frac{2}{3}$. 
We will require that $h(r)$ decays exponentially when $|T-|r|| \gg T^\alpha$
and remains close to constant when $|T-|r|| \ll T^\alpha$. 
The specific example that we will use for the present application is 
\begin{equation}\label{e:hdef}
h(r) = h_{T, \alpha}(r) 
= \left(e^{-\left(\frac{r-T}{T^\alpha}\right)^2} + e^{-\left(\frac{r+T}{T^\alpha}\right)^2}\right) 
\left(\frac{r^2+\frac{1}{4}}{r^2+R}\right)
\end{equation}
for $1\ll R< T^2$. 

We consider the following second moment for Rankin-Selberg convolution for level $N\geq 1$: 
\begin{multline}\label{e:secondmoment_intro}
S(s, t; f, g; h)
:= \sum_{j} \frac{h(r_j)}{\cosh(\pi r_j)} \scrL(1/2+it, f\times u_j) \overline{\scrL(\bar{s}, g\times u_j)}
\\ + \sum_{\cuspa} \frac{1}{4\pi} \int_{-\infty}^\infty \frac{h(r)}{\cosh(\pi r)}
\scrL_{\cuspa}(1/2+it, ir; f) \overline{\scrL_{\cuspa}(\bar{s}, ir; g)} \; dr. 
\end{multline}

\subsection{Statement of results in the case $N=1$}\label{ss:main_1}
When $N=1$, the second moment \eqref{e:secondmoment_intro} can be written as 
%We can finally define the second moment we intend to investigate:
\begin{multline}\label{e:secondmoment_intro_N=1}
S(s, t; f, g; h)
= \sum_{j} \frac{h(r_j)}{\cosh(\pi r_j)} |\rho_j(1)|^2 L(1/2+it, f\times u_j) L(s, \bar{g}\times \overline{u_j}) 
\\ + \int_{-\infty}^\infty \frac{h(r)}{\cosh(\pi r)}
\frac{L(1/2+it+ ir; f)L(1/2+it- ir; f) L(s +ir; \overline{g}) L(s -ir; \overline{g})}
{\Gamma\left(\frac{1}{2}+ir\right)\Gamma\left(\frac{1}{2}-ir\right) \zeta(1+2ir)\zeta(1-2ir)}\; dr. 
\end{multline}
Here 
\begin{equation}
L(s, f) = \sum_{n=1}^\infty \frac{A(n)}{n^s}
\end{equation}
is the $L$-function for $f$. 
%For general level $N>1$, the continuous spectrum, corresponding to the integral piece in $S(s, t; f, g; h)$, 
%is associated with Eisenstein series at cusps for $\Gamma_0(N)$. 

We will give in Theorem~\ref{thm:second_main} an expression of $S(s, t; f, g; h)$ for all $N\geq 1$
and all $h$ satisfying the conditions listed in \S\ref{ss:notation2} as a main term plus some explicit presumed error terms.   
This is for potential future applications with different versions of $h$.  
However our primary objective is to write the spectral piece of $S(s, t; f, g; h)$ in \eqref{e:secondmoment_intro}
(and in \eqref{e:secondmoment_intro_N=1} when $N=1$), that is, the sum over $j$, 
as a main term plus an error term for the specific $h$, defined in \eqref{e:hdef},
for sufficiently large $T$. 
In fact, note that the definition of $h=h_{T, \alpha}$ in \eqref{e:hdef} has the effect of making the sum over $j$ 
a smoothed sum over the interval 
\begin{equation}
|T-|r_j|| \ll T^\alpha 
\end{equation} 
of $L(1/2+it, f\times u_j) L(s, \bar{g}\times \overline{u_j})$, weighted by $\frac{|\rho_j(1)|^2}{\cosh(\pi r_j)}$, 
which, by \eqref{weight}, lies in the interval 
\begin{equation}
(1+| r_j|)^{-\epsilon} \ll\frac{ |\rho_j(1)|^2}{\cosh(\pi r_j)} \ll\log(1+| r_j|).
\end{equation}

Let 
\begin{equation}\label{e:H0}
H_0(ix; h) := \frac{1}{\pi^2} \int_{-\infty}^\infty h(r) r \tanh(\pi r) 
\frac{\Gamma\left(ir+ix+it+\frac{k}{2}\right)\Gamma\left(-ir+ix+it+\frac{k}{2}\right)}
{\Gamma\left(ir+it+\frac{k}{2}\right) \Gamma\left(-ir+it+\frac{k}{2}\right)} \; dr.
\end{equation}
We note that it is easily checked that, taking $h=h_{T, \alpha}$, given in \eqref{e:hdef},  as $T \rightarrow \infty$
\begin{equation}\label{e:H0_x=0}
H_0(0; h_{T, \alpha}) ~\sim 2 \pi^{-\frac{3}{2}}T^{1+\alpha}.
\end{equation}
For  $|x| \ll T^\beta$, with $\beta <1-\alpha$,  
\begin{equation}\label{e:H0_asymp}
H_0(ix; h_{T, \alpha}) \ll T^{1+\alpha}
\end{equation}
and finally, for $1-\alpha < \beta < 1$ and  $|x| = T^\beta$,
$H_0(ix; h_{T, \alpha})$ decays exponentially in $T$.

For $\Re(s)=\frac{1}{2}$ and $t\in \R$ let
\begin{multline}\label{e:scrM_N=1}
\scrM (s, t)
= \zeta(2s) \zeta(1+2it)
H_0(0; h)
\frac{L(s+1/2+it, f\times \bar{g})}{\zeta(2s+1+2it)}
\\ + (2\pi)^{4it} \zeta(2s)\zeta(1-2it)
H_0(-2it; h)
\frac{L(s+1/2-it, f\times \bar{g})}{\zeta(2s+1-2it)} 
\\ + (2\pi)^{4s-2+4it}\zeta(2-2s)\zeta(1-2it)
H_0(-2s+1-2it; h)
\frac{L(3/2-s-it, f\times \bar{g})}{\zeta(3-2s-2it)}
\\ + 
(2\pi)^{4s-2}\zeta(2-2s)\zeta(1+2it)
H_0(-2s+1; h)
\frac{L(3/2-s+it, f\times \bar{g})}{\zeta(3-2s+2it)}.
\end{multline}
Take $s = \frac{1}{2} -it'$, with $t' = \pm t$. 
If $t' = t$, and $1-\alpha< \beta< 1$, then, because of the exponential decay of $H_0(\pm 2it; h)$, the above simplifies considerably to (when $f\neq g$) 
\begin{equation}
\scrM(s, t)
= 2H_0(0; h) \frac{| \zeta(1-2it)|^2}{\zeta(2)} 
L(1, f\times \bar{g}).
\end{equation}
See the proof of Corollary~\ref{cor:nonvanishing_1} below for details.
Our main theorem, in the case $N=1$ is:
\begin{theorem}\label{thm:N=1} 
Fix $T \gg 1$, $\epsilon >0$, and $t$ with $t=0$ or $|t| =T^\beta$, with $\beta <1$. 
Let $\frac{1}{3} < \alpha < \frac{2}{3}$ and $h=h_{T, \alpha}$ as in \eqref{e:hdef}. 
Set $s= \frac{1}{2} - it'$, with $t' = \pm t$.  
%If $2\alpha > \beta$  and $\beta \le 1-\alpha$, 
%Then for even $k \ge 4$, 
The second moments of Rankin-Selberg convolutions $S(s, t; f, g; h)$, 
defined in \eqref{e:secondmoment_intro_N=1} are given by 
\begin{equation}
S(1/2-it', t; f, g; h) = \scrM(s, t) + \scrO\left(T^{\alpha+ \max(\alpha,1-s(\alpha, \beta; t')+\epsilon})\right),
\end{equation}
where 
\begin{equation}\label{e:s_degree_error}
s(\alpha, \beta; t')
= \begin{cases} 
\frac{3\alpha-1}{2} & \text{ if } \beta < \min\{2\alpha, 1-\alpha\} \text{ or } \beta=1-\alpha, \\
(2\alpha-\beta)\frac{k+1}{2} & \text{ if } 1-\alpha < \beta < 2\alpha \text{ and } t'=t, \\
\left(1-\frac{3\beta}{2} + \frac{\delta k}{2}\right) & \text{ if } 1-\alpha < \beta< \frac{\alpha+1}{2}, 
\; 2\beta-1+\delta=\alpha < \frac{2}{3}, \; \delta>0 \text{ and } t'=-t. 
\end{cases} 
\end{equation}
%If $2\alpha > \beta >1-\alpha$ and $t'=t$
%then
%$$
%S(s,t; f, g; h) = \scrM(s, t) + \scrO\left(T^{1+\alpha - (2\alpha-\beta)\frac{k+1}{2}+\epsilon}\right)
%$$
%If $t = -t'$ and $1-\alpha < \beta < \frac{\alpha+1}{2}$, and $2/3 >\alpha = 2\beta-1+ \delta$, $\delta>0$, then
%$$
%S(s,t; f, g; h) = \scrM(s, t) + \scrO\left(T^{1+\alpha -(1-3\beta/2 +\delta k/2)+\epsilon}\right).
%$$
For general $t$ in the given range, and any $\epsilon>0$,
$$
\scrM(s, t) \ll T^{1+\alpha + \epsilon}
$$
and for any $\alpha, \beta$ in the given ranges, for appropriate choices of $\epsilon$, $T^{1+\alpha + \epsilon}$ will always dominate the error term. 
When $t=0$, $\scrM(1/2, 0)$ is asymptotic to $T^{1+\alpha}P_d(\log T)$ 
where $P_d(x)$ is a degree $d$ polynomial of $x$ 
and the degree $d=3$ when $f=g$ and $d=2$ when $f\neq g$. 
The leading coefficient $c_{f, g}$ of $P_{d}(x)$ is given by 
%Moreover $\scrM(1/2-it', t)$ is asymptotic to $T^{1+\alpha} P_d(\log T)$, 
%where $P_d(x)$ is a degree $d$ polynomial of $x$ 
%and the degree $d$ is given by  
%\begin{equation}
%d=\begin{cases}
%3 & \text{ for $f=g$ and $t=0$, }\\
%1 & \text{ for $f=g$ and $t\neq 0$, }\\
%2 & \text{ for $f\neq g$ and $t=0$, }\\
%0 & \text{ for $f\neq g$ and $t\neq 0$.}
%\end{cases}
%\end{equation}
%When $t=0$, the leading coefficient of $P_d(x)$ is given as 
\begin{equation}\label{e:cfg_N=1}
c_{f, g} = \pi^{-\frac{3}{2}} \frac{2}{\zeta(2)}
\begin{cases} 
4L(1, f\times \bar{g}) & \text{ when } f\neq g, \\
8\Res_{s=1}L(s, f\times\bar{f}) & \text{ when } f=g. 
\end{cases} 
\end{equation}
%When $t\neq 0$, the leading coefficient $c(f, g; t,t')$ can be written in terms of $\zeta(1\pm 2it)$ and $L(1\pm 2it, f\times \bar{g}$ or $\Res_{s=1}L(s, f\times\bar{f})$, 
%and satisfying 
%\begin{equation}
%(k(|t|+1))^{-\epsilon} \ll c(f,g;t, t') \ll (k(|t|+1))^\epsilon
%\end{equation}

%Here $c(t, N, k)$ and $c_{2}(t, N, k)$ are explicit constants (and positive when $t=0$) which are independent of $f, g$ and $T$ 
%and satisfy, for any $\epsilon>0$, 
%\begin{equation}
%(kN(|t|+1))^{-\epsilon} \ll c(t, N, k), \; c_2(t, N, k) \ll (kN(|t|+1))^\epsilon.
%\end{equation}
\end{theorem}

\begin{remark}
Explicit descriptions of the error terms, without estimations, are given in Theorem~\ref{thm:second_main}.
This theorem is a special case of Theorem~\ref{thm:second_asymp}.  For general $N\geq 1$, this theorem remains true, as given in Theorem~\ref{thm:second_main} 
and Theorem~\ref{thm:second_asymp},
with the differences that the formulas for $\scrM(s, t)$ and $c_{f, g}$ include dependence on $N,k$ 
(given explicitly in \eqref{e:M2} and \eqref{e:cfg_N} respectively) 
and the error term is identical, except for an implied constant that depends on $N$ and $k$.  

%with the difference that $S(s,t; f, g; h)$ is defined by \eqref{e:secondmoment_intro} and it's following remarks, and $\scrM(s,t)$ is defined in full generality by  \eqref{e:M2}.   
%A simpler form for square free $N$ is given in \eqref{e:sqfreeM}.   A primitive form of Theorem~\ref{thm:N=1} for general $N$, demonstrating it's relation to the first moment it is built from is given in  Theorem~\ref{thm:second_main}.   The full version of the theorem, for general $N$, is given in Theorem~\ref{thm:second_asymp}.   In all expressions, the implied constants depend upon $N, k$ and $\epsilon$.

As is done in \cite{Ivic-Jutila}, 
for $\beta \ge 1$, it is possible to obtain useful, but weaker, upper bounds for the error terms from the explicit formula in Theorem~\ref{thm:second_asymp}. 
We leave the computations to a time when specific applications of these extremal cases arise. 

%Based on the explicit formulas given in Theorem~\ref{thm:second_main}, one can 
%However our paper is already sufficiently complicated that it seems better to leave these computations to a time when specific applications of these extremal cases arise. 
\end{remark}

\subsection{Separating the discrete contribution from the continuous contribution}

The expression $S(s,t; f, g; h)$, defined in \eqref{e:secondmoment_intro_N=1} combines 
both a discrete spectral sum over $j$, and the continuous piece which is given by the integral:
\begin{equation}\label{intpiece}
S_{\infty}(s, t; f, g; h)
= 
 \int_{-\infty}^\infty \frac{h(r)}{\cosh(\pi r)}
\frac{L(1/2+it+ ir; f)L(1/2+it- ir; f) L(s +ir; \overline{g}) L(s -ir; \overline{g})}
{\Gamma\left(\frac{1}{2}+ir\right)\Gamma\left(\frac{1}{2}-ir\right) \zeta(1+2ir)\zeta(1-2ir)}\; dr.
\end{equation}
As is seen in Theorem~\ref{thm:N=1}, when $s=\frac{1}{2}-it'$, 
the second moment $S(s,t; f, g; h)$ is on the order of at most a power of $T^{1+\alpha+ \epsilon}$.  
The integral piece in \eqref{intpiece}, which corresponds to the continuous spectrum, 
%continuous or integral piece in \eqref{intpiece} 
can be easily bounded from above after first using the approximate functional equation 
to replace the product of the four $L$-series in the integral by essentially finite polynomials.  
Then a mean value of a Dirichlet polynomial argument,
in particular, see \cite[Theorem 9.1]{IK04}, shows that  
\begin{equation}\label{e:intpiece_asymp}
S_{\infty}(1/2-it', t; f, g; h)
= \scrO\left( T^{2\max(\alpha, \beta) + \epsilon}  \right). 
\end{equation}
%the expression in \eqref{intpiece} is $\scrO\left( T^{2\max(\alpha, \beta) + \epsilon}  \right)$. 
This means that in any range where $1 + \alpha > 2\max(\alpha, \beta)$, and $s(\alpha, \beta; t')>0$
(given in \eqref{e:s_degree_error}),
the upper bound for the main term will dominate the error and the discrete cuspidal spectral piece of  $S(s,t; f, g; h)$ in \eqref{intpiece}, namely
\begin{equation}
\sum_{j} \frac{h(r_j)}{\cosh(\pi r_j)} |\rho_j(1)|^2 L(1/2+it, f\times u_j) L(s, \bar{g}\times \overline{u_j}),
\end{equation}
has the potential (if it is actually on the order of its upper bound) to dominate the continuous piece.  
Interestingly, there are ranges where the trivial bound for the continuous part exceeds the main term, and yet the error term is less than the main term.   For example, $\beta = 1-\epsilon$, $\alpha = \frac{1}{2}-\frac{\epsilon}{4}$, 
and so it makes sense to include the continuous part on the left hand side, and not just absorb it into the error term.

As the main term is a linear combination of terms it is hard to say in general that significant cancellation does not occur and that it really obtains its estimated maximum value.  However,  
in the case $t = 0$, so $s = \frac{1}{2}$, the main term can be calculated precisely, as we saw in Theorem~\ref{thm:N=1},  
and is asymptotic to $c_{f, g}T^{1 +\alpha}(\log T)^d$, with $d=2$ when $f\ne g$, 
and $d=3$ when $f= g$.   Note the definition of $c_{f,g}$ for general $N$ is given in Theorem~\ref{thm:second_asymp}.   
The same error estimates relating the sizes of the main terms and error terms are true in the case of general $N$ as in the case $N=1$.  
Unfortunately, however, although the dependence on $N$ in the continuous spectrum $S_\infty(s, t; f, g; h)$ for arbitrary level $N$ can be made explicit,  
the dependence on $N$ in other error terms is unknown.  
As was remarked earlier, this is because of the unknown dependence on $N$ of the triple product estimate \eqref{e:innerbound}.
 
\subsection{Statement of the main result for general $N$}\label{ss:second_N}
To state our result for arbitrary level $N\geq 1$, we must introduce one more notation. 
For each cusp $\cuspa\in \Q\cup\{\infty\}$ for $\Gamma_0(N)\bsl \HH$, let 
\begin{align}
& f_\cuspa(z) = (f|\sigma_\cuspa)(z) = \sum_{n=1}^\infty a_\cuspa(n)e^{2\pi inz}; \label{e:fourier_fa} 
\\ & g_\cuspa(z) = (g|\sigma_\cuspa)(z) = \sum_{n=1}^\infty b_\cuspa(n)e^{2\pi inz} \label{e:fourier_ga}
\end{align}
be the Fourier expansions for $f$ and $g$ at the cusp $\cuspa$. 
Let 
\begin{equation}\label{e:Rankin-Selberg_fg_a}
L_\cuspa(s, f\times \bar{g}) = \zeta^{(N)}(2s) \sum_{n=1}^\infty \frac{a_\cuspa(n)\overline{b_{\cuspa}(n)} n^{-k+1}}{n^s}
\end{equation}
be the Rankin-Selberg convolution of $f_\cuspa$ and $g_\cuspa$. 
When the cusp $\cuspa$ is equivalent to $\infty$, we also write $L(s, f\times \bar{g}) = L_\infty(s, f\times \bar{g})$. 
To describe the main term $\scrM(s, t)$ explicitly (see \eqref{e:M2}) we use the parameterization of cusps $\cuspa$ given in \cite{Y19}, explained in \S\ref{ss:factorization_scrL}. 
%$\cuspa = \frac{1}{ca}$ for $a\mid N$, $c\bmod{\gcd(a, N/a)}$, with $c$ chosen relatively prime to $\gcd(a, N/a)$ with 
%$\gcd(c, N)=1$. 
The primordial form of our starting place (before error terms are estimated) is given in the next theorem.  
\begin{theorem}\label{thm:second_main}
%Recalling the definition of the general $S(s,t; f, g; h)$ given in \eqref{e:secondmoment_intro},
For $\Re(s)=1/2$, we have 
\begin{equation}\label{e:second_intro}
S(s,t; f, g; h) = \scrM(s, t) + \scrE^+(s, t) + \scrE^-(s, t),  
\end{equation}
where 
%\begin{multline}\label{e:M2}
%\scrM (s, t)
%= \zeta(2s) \zeta(1+2it)
%H_0(0; h)
%\prod_{p\mid N} \frac{(1-p^{-2s})(1-p^{-1-2it})}{1-p^{-2s-1-2it}} 
%\frac{L(s+1/2+it, f\times \bar{g})}{\zeta(2s+1+2it)}
%\\ + (2\pi)^{4it} \zeta(2s)\zeta(1-2it)
%H_0(-2it; h)
%N^{-2it} \prod_{p\mid N} \frac{(1-p^{-1})(1-p^{-2s})}{1-p^{-2s-1+2it}}
%\frac{L(s+1/2-it, f\times \bar{g})}{\zeta(2s+1-2it)} 
%%
%\\ + (2\pi)^{4s-2+4it}\zeta(2-2s)\zeta(1-2it)
%H_0(-2s+1-2it; h)
%\frac{1}{2}\sum_{\cuspa} \frac{\scrP_{\cuspa}(s, it; 1-s-it)}{\prod_{p\mid N}(1-p^{1-2s-2it})}
%\frac{L_\cuspa(3/2-s-it, f\times \bar{g})}{\zeta^{(N)}(3-2s-2it)}
%\\ + 
%(2\pi)^{4s-2}\zeta(2-2s)\zeta(1+2it)
%H_0(-2s+1; h)
%N^{1-2s} \prod_{p\mid N}\frac{(1-p^{-1-2it})(1-p^{-1})}{(1-p^{-3+2s-2it})}
%\frac{L(3/2-s+it, f\times \bar{g})}{\zeta(3-2s+2it)}.
%\end{multline}
\begin{multline}\label{e:M2}
\scrM (s, t)
= \zeta(2s) \zeta(1+2it)
H_0(0; h)
\prod_{p\mid N} \frac{(1-p^{-2s})(1-p^{-1-2it})}{1-p^{-2s-1-2it}} 
\frac{L(s+1/2+it, f\times \bar{g})}{\zeta(2s+1+2it)}
\\ + (2\pi)^{4it} \zeta(2s)\zeta(1-2it)
H_0(-2it; h)
N^{-2it} \prod_{p\mid N} \frac{(1-p^{-1})(1-p^{-2s})}{1-p^{-2s-1+2it}}
\frac{L(s+1/2-it, f\times \bar{g})}{\zeta(2s+1-2it)} 
\\ + (2\pi)^{4s-2+4it}\zeta(2-2s)\zeta(1-2it)
H_0(-2s+1-2it; h)
N^{-\frac{1}{2}-s-it}
\\ \times \sum_{\cuspa=\frac{1}{ca}} 
\left(\frac{a}{\gcd(a, N/a)}\right)^{-s+\frac{3}{2}-it}
\prod_{p\mid \frac{N}{a}} \frac{(1-p^{1-2s})(1-p^{-2it})}{1-p^{-3+2s+2it}}
\prod_{\substack{p\mid a\\ p\nmid\frac{N}{a}}} \frac{(1-p^{-1})^2}{1-p^{-3+2s+2it}}
\frac{L_\cuspa(3/2-s-it, f\times \bar{g})}{\zeta(3-2s-2it)}
\\ + 
(2\pi)^{4s-2}\zeta(2-2s)\zeta(1+2it)
H_0(-2s+1; h)
N^{1-2s} \prod_{p\mid N}\frac{(1-p^{-1-2it})(1-p^{-1})}{(1-p^{-3+2s-2it})}
\frac{L(3/2-s+it, f\times \bar{g})}{\zeta(3-2s+2it)}
\end{multline}
and the potential error terms $\scrE^+(s, t)$ and $\scrE^-(s, t)$ are described in \eqref{e:E_OD+} and \eqref{e:scrE-}.   
Here $H_0(ix; h)$ is given in \eqref{e:H0}. 
\end{theorem}

%\begin{remark}
%There is a cancelation of poles in $\scrM(s, t)$ when $s=1/2\pm it$.  
%Because of this cancelation, at $s=1/2\pm it$, we see that
%\begin{equation}\label{e:Pcuspa_special}
%\sum_{\cuspa} P_\cuspa(1/2-it, -1/2, it)
%= 2\prod_{p\mid N}(1-p^{-1+2it})(1-p^{-1-2it}). 
%\end{equation}
%\end{remark}

%\begin{remark}
%When $N$ is square-free, the Euler polynomials $P_{\cuspa}(s, -1+s-it, it)$ 
%and the Rankin-Selberg convolutions $L_\cuspa(s, f\times\bar{g})$ can be written explicitly 
%As we noted above each cusp $\cuspa$ can be given as $\cuspa=1/a$ for $a\mid N$. 
%By \eqref{e:P1/a}, 
%\begin{equation}\label{e:P1/a_+}
%P_{1/a}(s, -1+s+it, it)
%= 2(Na)^{\frac{1}{2}-s-it}
%\prod_{p\mid a} (1-p^{-1})^2 
%\prod_{p\mid \frac{N}{a}} p^{-2it}(p^{-1+2it}-p^{-1})(1-p^{1-2s}).
%\end{equation}
%Note that $P_{1/a}(s, -1+s-it, it)=0$ unless $a=N$. 
%When $a=N$
%\begin{equation}\label{e:P1/a_-}
%P_{1/N}(s, -1+s-it, it)
%= 2 N^{1-2s} \prod_{p\mid N} (1-p^{-1})(1-p^{-1-2it}).
%\end{equation}
%Moreover we get
%\begin{equation}\label{e:L1/afg}
%L_{1/a}(s, f\times\bar{g}) = \frac{N}{a} A(N/a)\overline{B(N/a)} L(s, f\times\bar{g}). 
%\end{equation}
%By \cite[Lemma~3]{Asa76}, $|A(N/a)\overline{B(N/a)}| = \frac{a}{N}$. 
%\end{remark}

For our applications, we will use $h=h_{T, \alpha}$ as defined in \eqref{e:hdef}.  
Recall that for fixed $T\gg 1$ and $\frac{1}{3}< \alpha < \frac{2}{3}$,  
$h(r)$ decays exponentially when $|T-|r|| \gg T^\alpha$.

\begin{theorem}\label{thm:second_asymp}
Fix $T \gg 1$ and $\epsilon >0$. 
%and $h=h_{T, \alpha}$ as in \eqref{e:hdef} for $\frac{1}{2}< \alpha< \frac{2}{3}$. 
Take $s=\frac{1}{2}-it'$ in $S(s,t; f, g; h_{T, \alpha})$, defined in \eqref{e:second_intro}, 
for $t=0$ or $|t'|=|t|=T^\beta$ for $0<\beta< 1$. 
Then the second moment  $S(s, t; f, g; h)$  is given by 
\begin{equation}
S(1/2-it', t; f, g; h) = \scrM(1/2-it', t) + \scrO\left(T^{\alpha+ \max(\alpha,1-s(\alpha, \beta; t')+\epsilon})\right),
\end{equation}
where $s(\alpha, \beta; t')$ as given in \eqref{e:s_degree_error}. 
%\begin{equation}
%s(\alpha, \beta; t') = \begin{cases} 
%\frac{3\alpha-1}{2} & \text{ if } 2\alpha> \beta \text{ and } \beta\leq 1-\alpha, \\ 
%(2\alpha-\beta)\frac{k+1}{2} & \text{ if } 2\alpha> \beta> 1-\alpha \text{ and } t'=t, \\
%\left(1-\frac{3\beta}{2} + \frac{\delta k}{2}\right) & \text{ if } 1-\alpha < \beta< \frac{\alpha+1}{2},\,\,\text{and} \, \; \frac{2}{3} > \alpha= 2\beta-1+\delta, \; \delta>0 \text{ and } t'=-t. 
%\end{cases} 
%\end{equation}
Here $\scrM(s, t)$ is given in \eqref{e:M2} and satisfies the upper bound
\begin{equation}
\scrM(s, t) \ll T^{1+ \alpha + \epsilon}.
\end{equation}

When $t=0$, as $T\to \infty$, 
\begin{equation}\label{e:scrM_1/2_t=0_asymp}
\scrM(1/2, 0) \sim T^{1+\alpha}P_d(\log T), 
\end{equation}
where $P_d(x)$ is a degree $d$ polynomial of x
and the degree $d=3$ when $f=g$ and $d=2$ when $f\neq g$. 
The leading coefficient $c_{f, g}$ is given by 
\begin{equation}\label{e:cfg_N}
c_{f, g} = 2^d\pi^{-\frac{3}{2}} \frac{2}{\zeta(2)} \prod_{p\mid N} \frac{1-p^{-1}}{1+p^{-1}} 
\begin{cases} 
L(1, f\times \bar{g}) & \text{ when } f\neq g, \\
\Res_{s=1}(s, f\times\bar{f}) & \text{ when }f=g. 
\end{cases}
\end{equation}
%
%This is as $T \rightarrow \infty$ and $t,t',k,N$ remain fixed.     In a less simplistic description, 
%the leading coefficient of $P_d(\log T)$ is a function of $t,t',k,N, f,g$, which we will denote $c(t,t',k,N, f,g)$, that can be read off from \eqref{e:M2} or \eqref{e:sqfreeM} and has the 
%mild size restrictions:
%$$
%(k(|t|+1))^{-\epsilon} \ll c(t,t',k,N, f,g) \ll (k(|t|+1))^\epsilon.
%$$
%We note finally that for all $T \gg 1$ and  $|x| \ll T^\beta$, with $\beta <1$,  $H_0(x; h) ~\sim 2 \pi^{-\frac32}T^{1+\alpha}$. 
\end{theorem}

\begin{remark}
More specifically, for general $N$, when $f \ne g$ and $s = \frac{1}{2}\pm it$, $\scrM(s, t)$ is given in \eqref{e:scrM_fneqg_1/2-it} and  \eqref{e:scrM_fneqg_1/2+it}.  
When $f = g$ and $s = \frac{1}{2} \pm it$, $\scrM(s, t)$ is given by \eqref{e:scrM_f=g_1/2-it} and \eqref{e:scrM_f=g_1/2+it}. 
\end{remark}

%%%%%%%%%%%%%%%%%%
%P. Sarnak, Estimates for Ranhin-Selberg L-functions and quantum unique ergodicity, J. Functional Analysis 184 (2001), 419-453.
%V. Blomer, Rankin-Selberg L-functions on the critical line, manuscripta math. 117 (2005), 111-133.
%%%%%%%%%%%%%%%%%%%%
 
\subsection{Applications}
The discussion in the previous section, combined with Theorems~\ref{thm:N=1}, Theorem~\ref{thm:second_main} and Theorem~\ref{thm:second_asymp}
%This discussion, combined with Theorems~\ref{thm:N=1}, Theorem~\ref{thm:second_main} and Theorem~\ref{thm:second_asymp}
has the following two immediate corollaries:

\begin{corollary}\label{cor:nonvanishing_1}
Let $N\geq 1$, $k\geq 4$ be an even integer and $f$ and $g$ be holomorphic newforms of level $N$ and weight $k$. 
Assume that $f\neq g$.  

For any $T \gg 1$ there will exist at least one newform $u_j$, for $\Gamma_0(N)$,  
with corresponding spectral parameter $r_j$, such that $|T-|r_j|| \ll T^{\frac{1}{3} +\epsilon}$
and $L(1/2, f\times u_j), L(1/2, \bar{g}\times \overline{u_j})$ are simultaneously non-zero.  
More generally, if $|t|= T^\beta$, with $\beta < \frac{2}{3}$, 
then the same is true of $L(1/2+it, f\times u_j), L(1/2+it, \bar{g}\times \overline{u_j})$, within the interval $|T-|r_j|| \ll T^{1 -\beta +\epsilon}$.  
%For any $0 \ne |t|= T^{\beta}$, 
%the same, with shorter intervals,  can be proved by an explicit evaluation of $\scrM(1/2 -it, t)$.
 \end{corollary}

\begin{proof}
Take $\alpha = \frac{1}{3} + \epsilon$ and $h=h_{T, \alpha}$.   
Taking $t=0$, we consider the second moment $S(1/2, 0; f, g; h)$. 
Then as observed in \eqref{e:scrM_1/2_t=0_asymp}, 
the main term $\scrM(1/2, 0; f, g; h)$ is asymptotic to $c_{f, g}T^{\frac{4}{3} + \epsilon}(\log T)^2$.
The contribution of the continuous part is, as seen in \eqref{e:intpiece_asymp}, 
$\scrO(T^{\frac{2}{3} + 2\epsilon + \epsilon''})$, and the error term is  
$\scrO(T^{\frac{4}{3} + \epsilon - \delta'})$, for some $\delta' >0$.
It follows that the discrete sum over $j$ dominates and is asymptotic to $c_{f, g}T^{\frac{4}{3} + \epsilon}(\log T)^2$.
Thus the discrete sum must be non-zero, and therefore one of the products of 
$L(1/2, f\times u_j) L(1/2, \bar{g}\times \overline{u_j})$ must  not vanish in the range $|T-|r_j|| \ll T^{\frac{4}{3} +\epsilon}$.   

For non-zero $|t|=T^\beta$, with $\beta<\frac{2}{3}$,  the same can be done by choosing $\alpha = 1-\beta + \epsilon$. 
Then $\beta >1-\alpha$ and because of the exponential decay of $H_0(\pm 2it,;h)$, 
when $f\neq g$, by \eqref{e:scrM_fneqg_1/2-it}, 
\begin{equation}
\scrM(1/2-it, t)
\sim H_0(0; h) \frac{|\zeta(1-2it)|^2}{\zeta(2)} L(1, f\times \bar{g}), 
\end{equation}
when $N=1$. 
When $N>1$ positive real coefficients depending on $N$ appear. 
See \eqref{e:scrM_fneqg_1/2-it} for details. 
The crucial fact is that, the right hand side is non-zero. 
This completes the proof of the corollary.
\end{proof}

This is the first simultaneous non-vanishing result that we are aware of for two distinct Rankin-Selberg convolutions 
for $\GL(2)\times \GL(2)$.

Taking $f=g$, $\alpha = \frac{1}{3}+\epsilon$ and bounding $\scrM(1/2-it, t)$ from above by $T^{\frac{4}{3}+ \epsilon}$, as all the terms in the spectral sum are non-negative,
another immediate consequence is

\begin{corollary}\label{cor:subconvexity}
Let $N\geq 1$, $k\geq 4$ be an even integer and $f$ be a holomorphic newform of level $N$ and weight $k$. 
Let $u_j$ be a Maass newform of level $N$ with the corresponding spectrum parameter $|r_j| \ll T$.  
For $|t|\leq T^{\frac{2}{3}}$, 
\begin{equation}
\left| L\left(1/2+it, f\times u_j\right) \right|
\ll  T^{\frac23+\epsilon}. 
\end{equation}
When $|t| = T^\beta$, with $\frac{2}{3} \le \beta \leq 2\alpha<1$, 
\begin{equation}
\left| L\left(1/2+it, f\times u_j\right) \right| \ll T^{\frac12+\epsilon}|t|^{\frac{1}{4}+\epsilon}. 
\end{equation}
In both cases the implied constant is independent of $t$ and $T$.
\end{corollary}

We specialize to Maass newforms $u_j$ simply because the notation for the Rankin-Selberg convolution becomes 
more complicated when $u_j$ is not a newform, 
and oldform bounds can be reduced to newform bounds, with an extra dependence on the level. 
Also, although we restrict our results to $k\geq 4$, $k=2$  
can also be covered with more careful treatments of the lines of integration separating poles of gamma functions.  

We must emphasize that our estimate for $L\left(1/2+it, f\times u_j\right)$ is only new in the sense that it is valid for arbitrary level.  It is  the same Weyl type estimate 
as in  \cite{Ivic-Jutila} and \cite{LLY06}, which are for level $1$, although   
the result of \cite{LLY06} does not give the dependence on $t$  
The result of \cite{Ivic-Jutila} is uniform in $t$, giving the same estimate as ours in the range $0 \le \beta <1$.   
They also extend the range of  $\beta$ beyond $1$. 
As remarked before, our methods will also extend to $\beta \ge 1$ but we have chosen not to do this.

To summarize, we have obtained an expression for the spectral moment of a degree $8$ Euler product, 
which is the product of two Rankin-Selberg convolutions, as a main term plus a sharp error term.  
Except for the level $N$, and weight $k$ this error term is uniform in the parameters $T$ and $t$. 
Other estimates for degree 8 moments that we are aware of are \cite{BK17} and \cite{Ivi02}. 
In \cite{BK17}, 
\begin{equation}
\frac{1}{T^2} \sum_{j\ge 1}e^{-\frac{r_j^2}{T^2}}\frac{L(\frac12,u_j)L(\frac12, u_j\otimes\chi_d)L(\frac12, u_j\otimes f)}
{L(1,u_j,\vee^2)}
\end{equation}
is expressed as a main term plus an error, on the order of $T^{-\delta}$, for some $\delta >0$.  
Here the level is $1$,  $f$ is a Maass cuspform of eigenvalue $1/4+T^2$ and $\chi_d$ is the quadratic character of conductor $d$.  In \cite{Ivi02}, an expression is given in the level $1$ case for the spectral fourth moment of an $L$-series of a Maass form 
at $1/2$ as a main term plus an error.

In \cite{KMV} a second moment estimate for the absolute value squared of a Rankin-Selberg convolution 
is given in terms of a main term and an error estimate in the level aspect: 
\begin{equation}
\frac{1}{\left|S_k^*(q)\right|}
\sum_{f\in S_k^*(q)}\left| L\left(1/2, f\otimes g\right) \right|^2 = P(\log q)+ \scrO_{g,k,\epsilon}\left(q^{-\frac{1}{12}+\epsilon}\right),
\end{equation}
where $S_k^*(q)$ consists of cuspidal holomorphic newforms of level $q$ and weight $k$ and $P(x)$ is a cubic polynomial.
Also second moment subconvexity results in the level aspect are given in \cite{Z14}, 
and one for a moment of a degree $8$ Euler product in \cite{BKY}. 
%one of the $u_j$, with eigenvalue $1/4 + T^2$.   

%We stress again that to make notation less cumbersome, throughout this paper every implied constant 
%in a $\ll$ or $\scrO$ expression will be assumed to depend on $N$, $k$ and, if relevant, $\epsilon >0$.

\subsection{An outline of the paper}
%We state the main result for $N\geq 1$ in \S\ref{ss:second_N}.
In \S\ref{sect:secondconst} we recall the original form \cite[Theorem~1.1]{hln19} 
of our first moment paper for Rankin-Selberg convolutions, applied to a holomorphic cusp form $f$.   
This formula \eqref{e:K_first} is an expression for the first moment 
with no estimates applied to the potential error terms \eqref{e:L-holo_first} and \eqref{e:L+holo_first}.   
Each expression has a parameter $n\ge 1$.   
We remark that one of the main  innovations in this expression is the transformation and analytic continuation 
of the integral expression (inverse Mellin transforms of ratios of gamma functions) 
$F_1(s, u, \nu; x)$ and  $F_2(s, u, k; x)$ given in  \cite[Lemmas 3.3 and 3.5]{hln19}. 
To introduce the second Rankin-Selberg convolution we multiply 
by $\overline{b(n)}n^{-s-\frac{k-1}{2}}$, where $b(n)$ is a Fourier coefficient of another holomorphic cusp form $g$, 
and sum over $n \geq 1$.   
The result is a sum over $n$ of the main term of the first moment, 
plus two additional sums, corresponding to the two sources of the error term in  \cite[Theorem~1.5]{hln19}.   
We denote the sum over $n$ of the main term as $M_1(s, t)$ in  \eqref{e:def_Mst}, 
and the sum over $n$ of the two remaining terms as $\OD^\pm(s, t)$ \eqref{e:def_ODpm} 
%and $\OD^-(s, t)$ \eqref{e:def_OD-}, 
obtaining
\begin{equation}
S(s,t; f, g; h)
= M_1(s, t) + \OD^+(s, t) + \OD^-(s, t).
\end{equation}
This $M_1(s, t)$, given in \S\ref{sect:M}, turns out to be half of our final main term.  
%In \S\ref{sect:M} we give a precise description of $M_1(s,t)$. 
%using the explicit
%computations of the Fourier expansion due to Matt Young in \cite{Y19}. 
%This turns out to be half of our final main term.  

The next two sections, \S\ref{s:OD+} and \S\ref{s:OD-}, are devoted to studying $\OD^+(s, t)$ and $\OD^-(s, t)$ via analyses of the shifted double Dirichlet series.

In \S\ref{s:OD+} we begin the analysis of $\OD^+(s, t)$ and discover in \eqref{e:OD+_Z_pre}
%\eqref{e:Zdef} 
that at its heart is the shifted double Dirichlet series
\begin{equation}
Z(s, v; it) = \zeta^{(N)}(2s) 
\sum_{n=1}^\infty \sum_{m=1}^\infty \frac{\sigma_{-2it}(m; N)m^{it} a(n+m)\overline{b(n)}}{m^{v}n^{s-v-\frac{1}{2}+k}}.
\end{equation}
Another innovation in this paper is that in the remainder of \S\ref{s:OD+} 
we build on the analysis of the shifted single Dirichlet series  in \cite{HHR} 
and obtain the spectral decomposition of $\OD^+(s, t)$, 
which is given in Proposition~\ref{OD+def}.   
It is worth remarking that the analysis is made far more difficult because the spectral decomposition does not converge everywhere.  
It contains a critical strip, and the Dirichlet series representation converges to the right of this critical strip, while the spectral representation converges to the left. 
See Proposition~\ref{prop:Z} and \cite[Proposition~5.1]{HHR} for details.

When obtained, the decomposition consists of a number of pieces all but one of which contribute to the error term.   
In the following subsections  all of these terms that contribute to the error term are bounded from above  
in Propositions~\ref{prop:ub_OD+ccint}, \ref{prop:OD+res_upper} and \ref{prop:OD+Omega_upper}.  
Then, in \S\ref{sect:main2}, a second piece of the main term, $M^+_{\Omega}(s, t)$, is computed, 
using the explicit
computations of the Fourier expansion due to Young in \cite{Y19}. 
This turns out to be one half of the remaining half of our final main term. 

In \S\ref{s:OD-} the spectral analysis of the third contribution, denoted $\OD^-(s, t)$, is performed.  
This involves another shifted double Dirichlet series given in \eqref{e:cM3}:
\begin{equation}
\zeta^{(N)}(2s) \sum_{m=1}^\infty \sum_{n=m+1}^\infty \frac{\sigma_{-2it}(m; N) m^{it} a(n-m)\overline{b(n)}}
{m^{s+\frac{k-1}{2}} n^{w+k-1}}.
\end{equation}
The subsum over $n$ is a  single shifted Dirichlet series, of the type considered originally by Selberg.  
Its spectral expansion converges everywhere and it is far easier to deal with.   The double sum too breaks into
several pieces
\begin{equation}
\OD^-(s, t) =  \OD_0^-(s, t) + \OD_{\Omega, \intg}^-(s, t)+\OD_{\Omega, \res}^-(s, t),
\end{equation}
of which $\OD_{\Omega, \res}^-(s, t)=M^-_\Omega(s, t)$, given in \eqref{e:OD-_Omegares}, 
is the remaining contribution to the main term. 
The explicit description is given in \S\ref{ss:MOmega-}. 
The upper bounds for $\OD_0^-(s, t)$ and  $\OD_{\Omega, \intg}^-(s, t)$ are given in  Proposition~\ref{prop:OD-up}.

\par\vspace{2\jot}\noindent
\textbf{Acknowledgements}.\ 
The authors would like to thank Matt Young, Wenzhi Luo  and Mehmet Kiral for
some very helpful discussions and comments, Peter Humphries for some valuable comments on the existing literature, 
and KIAS and POSTECH for providing a welcoming working
environment during part of the preparation of this paper.

\section{Construction of the second moment from the first moment}\label{sect:secondconst}

\subsection{The formula for first moments}\label{s:first}
Let $h(r)$ be a function satisfying the conditions in \S\ref{ss:main_1}. 
We use the notations in \S\ref{ss:basis_L2}, \S\ref{ss:notation2} and \S\ref{ss:second_N}. 

For each positive integer $n$, we recall the first moment function from \cite{hln19}, 
over the spectral expansion for $L^2(\Gamma_0(N)\bsl \HH)$: 
Let 
\begin{multline}\label{e:K_firstmoment_def}
K(s, f; n, h)
= \sum_j \frac{h(r_j)}{\cosh(\pi r_j)} \overline{\rho_j(n)} \scrL(s, f\times u_j) 
\\ + \sum_{\cuspa} \frac{1}{4\pi} \int_{-\infty}^\infty \frac{h(r)}{\cosh(\pi r)} 
\tau_\cuspa(1/2-ir, n) \scrL_\cuspa(s, ir; f) \; dr.
\end{multline}
Here $f$ is a holomorphic newform of weight $k$ for $\Gamma_0(N)$, with the Fourier expansion \eqref{e:f_Fourier}.
For $s=\frac{1}{2}+it$, by \cite[Theorem~1.2]{hln19}, we get
\begin{equation}
K(1/2+it, f; n, h)\label{e:K_first}
= M(1/2+it, f; n) + L^-(1/2+it, f; n) + L^+(1/2+it, f; n)
\end{equation}
where
\begin{equation}\label{e:M_first}
M(1/2+it, f; n)
= \zeta^{(N)}(1+2it) \frac{A(n)}{n^{\frac{1}{2}+it}} 
H_0(0; h)
%\frac{1}{\pi^2}\int_{-\infty}^\infty r\tanh(\pi r) h(r)\; dr
+ (2\pi)^{4it}\zeta(1-2it) 
N^{-2it} \frac{\varphi(N)}{N} %\prod_{p\mid N}(1-p^{-1})
\frac{A(n)}{n^{\frac{1}{2}-it}} 
H_0(-2it; h), 
%\frac{1}{\pi^2}\int_{-\infty}^\infty 
%h(r) r \tanh(\pi r)
%\frac{\Gamma\left(\frac{k}{2}-it+ir\right)\Gamma\left(\frac{k}{2}-it-ir\right)}
%{\Gamma\left(\frac{k}{2}+it+ir\right)\Gamma\left(\frac{k}{2}+it-ir\right)}\; dr, 
\end{equation}
where $H_0(ix; h)$ is given in \eqref{e:H0}, 
and $L^-(1/2+it, f; n)$ and $L^+(1/2+it, f; n)$ are given as follows: 
\begin{multline}\label{e:L-holo_first}
L^-(1/2+it, f; n)
= -\frac{(2\pi)^{2it} \cos(\pi it)}{\pi} 
\frac{4}{2\pi i} \int_{(\sigma_u)} \frac{h(u/i) u \tan(\pi u)}{\Gamma\left(u+it+\frac{k}{2}\right) \Gamma\left(-u+it+\frac{k}{2}\right)}
\\ \times \frac{1}{2\pi i} \int_{(\sigma_0)} \Gamma\left(u-v\right) \Gamma\left(-u-v\right) 
\Gamma\left(\frac{k}{2}-it+v\right) \Gamma\left(\frac{k}{2}+it+v\right)
n^v \sum_{m=1}^{n-1} \frac{a(m)\sigma_{-2it}(n-m; N)}{(n-m)^{v-it+\frac{k}{2}}}
\; dv \; du
\end{multline}
and
\begin{multline}\label{e:L+holo_first}
L^+(1/2+it, f; n)
=i^k (2\pi)^{2it}
\frac{4}{2\pi i}\int_{(\sigma_u)} \frac{h(u/i)u}{\cos(\pi u)}
\frac{1}{\Gamma\left(-u+it+\frac{k}{2}\right)\Gamma\left(u+it+\frac{k}{2}\right)}
\\ \times \frac{1}{2\pi i} \int_{(1+\frac{k}{2}+\epsilon)}
\frac{\Gamma\left(v-it\right) \Gamma\left(v+it\right) \Gamma\left(-v+u+\frac{k}{2}\right)}
{\Gamma\left(v+u+1-\frac{k}{2}\right)}
n^{v-\frac{k}{2}} \sum_{m=1}^\infty \frac{\sigma_{-2it}(m; N)a(n+m)}{m^{v-it}} \; dv \; du.
\end{multline}
Here we assume that $1< \sigma_u < \frac{3}{2}$ and $-\frac{k}{2}< \sigma_0 < -\sigma_u$, 
so the poles of gamma functions in both $L^\pm$ are separated. 
Note that the sum over $m\geq 1$ in $L^+$ converges absolutely for $\Re(v)= 1+\frac{k}{2}+\epsilon$. 
Moreover the contour integrals in $L^\pm(1/2+it, f; n)$ converge absolutely.  
Here we also have 
\begin{equation}\label{e:sigma_1N}
\sigma_{-2it}(m; N)
= \begin{cases} 
\frac{N^{-2it}}{\prod_{p\mid N} p} P_N(1/2+it, m) \sigma_{-2it}^{(N)}(m) & \text{ for } \frac{N}{\prod_{p\mid N} p} \mid m, \\
0 & \text{ otherwise}
\end{cases}
\end{equation}
where 
\begin{equation}
\sigma_{-2it}^{(N)}(m) = \sum_{\substack{d\mid m, \\ \gcd(d, N)=1}} d^{-2it} 
\end{equation}
and 
\begin{equation}\label{e:P_M}
P_M(s, n) 
%= \prod_{p\mid N} \frac{-p^{(1-2s)(\ord_p(n)+1)}(1-p^{-2s}) + p^{(1-2s)\ord_p(N)}(1-p^{-1})}{1-p^{1-2s}}
%\\ =  \prod_{p\mid N} p^{(1-2s)\ord_p(N)-1} \frac{p^{(1-2s)(\ord_p(n)+1)}(p^{-2s-(1-2s)\ord_p(N)+1}-p^{-(1-2s)\ord_p(N)+1})
%+ p-1}{1-p^{1-2s}}
%=  \frac{N^{1-2s}}{\prod_{p\mid N} p} 
%\prod_{p\mid M} \frac{p^{(1-2s)(\ord_p(n)+1)}(p^{-2s-(1-2s)\ord_p(M)+1}-p^{-(1-2s)\ord_p(M)+1})
%+ p-1}{1-p^{1-2s}}
= 
%\frac{N^{1-2s}}{\prod_{p\mid N} p} - moved to the outside of the product. 
\prod_{p\mid M} \frac{p^{(1-2s)(\ord_p(n)+1)}p^{-(1-2s)(\ord_p(M)-1)}(1-p^{2s})+ p-1}{1-p^{1-2s}}, 
\end{equation}
for a positive integer $M$. 
Note that this $P_M(s, n)$ is different from $P_M(2s-1, n; 1)$ defined in \cite{hln19} by the factor $\frac{M^{1-2s}}{\prod_{p\mid M} p}$, 
i.e., $P_M(s, n) = \frac{M^{1-2s}}{\prod_{p\mid M} p} P_{M}(2s-1, n; 1)$.
%\begin{equation}
%\sigma_{1-2s}(n; a) = \prod_{p\mid a} \frac{p^{(\ord_p(n)+1)(1-2s)}(1-p^{2s})+ (p-1)}{1-p^{1-2s}}. 
%\end{equation}
%and $P_N(-2it, m)$ is as given in \eqref{e:P_M}. 
%and 
%\begin{equation}
%P_N(2it, m)
%= \prod_{p\mid N} \frac{-p^{-2it(\ord_p(m)+1)} (1-p^{-1-2it}) + p^{-2it\ord_p(N)} (1-p^{-1})}{1-p^{-2it}}. 
%\end{equation}

\subsection{Construction of second moments}
Recall that $g$ is a holomorphic newform of level $N$ and weight $k$, with Fourier expansion given by  \eqref{e:g_Fourier}.

The aim of this section is to perform the first step in studying the asymptotic behavior of the second moment of the Rankin-Selberg convolution $S(s, t; f, g; h)$ given in \eqref{e:secondmoment_intro}. 

Recalling \eqref{e:secondmoment_intro} and \eqref{e:K_firstmoment_def}, 
multiplying by $\overline{b(n)}n^{-s-\frac{k-1}{2}}$ and summing over $n\geq 1$, 
we have, for $s\in \C$ with sufficiently large $\Re(s)>1$, 
\begin{equation}
S(s, t; f, g, h)
= \zeta^{(N)}(2s) \sum_{n=1}^\infty \frac{\overline{b(n)}}{n^{s+\frac{k-1}{2}}} K(1/2+it, f; n, h).
\end{equation}
%\begin{multline}
%\zeta^{(N)}(2s) \sum_{n=1}^\infty \frac{\overline{b(n)}}{n^{s+\frac{k-1}{2}}} K(1/2+it, f; n, h)
%= \sum_j \frac{h(r_j)}{\cosh(\pi r_j)} \scrL(1/2+it, f\times u_j) \scrL(s, \bar{g}\times \overline{u_j})
%\\ + \sum_{\cuspa} \frac{1}{4\pi} \int_{-\infty}^\infty \frac{h(r)}{\cosh(\pi r)}  
%\tau_\cuspa(1/2+ir, n) \scrL_\cuspa(1/2+it, ir; f) \scrL_{\cuspa}(s, -ir; \bar{g}) \; dr
%\\ = S(s, t; f, g; h). 
%\end{multline}
 Then referring to \eqref{e:K_first},
%for $s\in \C$ with sufficiently large $\Re(s)>1$
\begin{equation}\label{e:secondmoment}
S(s,t; f, g; h)
%= \zeta^{(N)} (2s) \sum_{n=1}^\infty \frac{\overline{b(n)}}{n^{s+\frac{k-1}{2}}} K(1/2+it, f; n, h)
%\\ = \sum_{j} \frac{h(r_j)}{\cosh(\pi r_j)} \scrL(1/2+it, f\times u_j) \scrL(s, \bar{g}\times \overline{u_j}) 
%+ \sum_{\cuspa} \frac{1}{4\pi} \int_{-\infty}^\infty \frac{h(r)}{\cosh(\pi r)}
%\scrL_{\cuspa}(1/2+it, ir; f) \scrL_{\cuspa}(s, -ir; \overline{g}) \; dr
\\ = M_1(s, t) + \OD^+(s, t) + \OD^-(s, t), 
\end{equation}
where
\begin{equation}\label{e:def_Mst}
M_1(s, t)
= \zeta^{(N)}(2s) \sum_{n=1}^\infty \frac{\overline{b(n)}}{n^{s+\frac{k-1}{2}}} M(1/2+it, f; n)
\end{equation}
and 
\begin{equation}\label{e:def_ODpm}
\OD^{\pm}(s, t)
= \zeta^{(N)}(2s) \sum_{n=1}^\infty\frac{\overline{b(n)}}{n^{s+\frac{k-1}{2}}}  L^{\pm}(1/2+it, f, ; n).
\end{equation}

%%
%and 
%\begin{equation}\label{e:def_OD-}
%\OD^-(s, t)
%= \zeta^{(N)}(2s) \sum_{n=1}^\infty\frac{\overline{b(n)}}{n^{s+\frac{k-1}{2}}} L^-(1/2+it, f; n). 
%\end{equation}

Our goal is to establish an explicit expression for $S(s,t; f, g; h)$ on $\Re(s)=\frac{1}{2}$ as a main term plus an error term.
Ultimately, when establishing upper bounds we will only examine the two cases $t' = \pm t$.  
However most of our exposition is valid for any pair $t$, $t'$.   
For a fixed $T\gg 1$ and $h=h_{T, \alpha}$ as in \eqref{e:hdef}, 
we also set $|t| = T^\beta$ with $\beta <1$ and $\beta =0$ when $t=0$.
Our method can deal with larger $\beta$ (obtaining somewhat worse bounds, as is the case in \cite{JM05}) but restricting to $|t|<T$ simplifies arguments considerably.

\subsection{A description of $M_1(s,t)$}\label{sect:M}
Recalling the description of $M(1/2+it, f; n)$ in \eqref{e:M_first}, we have 
\begin{multline}\label{e:Mst}
M_1(s, t)
=\zeta^{(N)}(2s) \sum_{n=1}^\infty \frac{\overline{b(n)}}{n^{s+\frac{k-1}{2}}} M(1/2+it, f; n)
\\ = \frac{\zeta(2s) \zeta(1+2it)}{\zeta(2s+1+2it)}
\prod_{p\mid N} \frac{(1-p^{-2s})(1-p^{-1-2it})}{1-p^{-2s-1-2it}} 
L(s+1/2+it, f\times \bar{g}) 
H_0(0; h)
%\frac{1}{\pi^2}\int_{-\infty}^\infty r\tanh(\pi r) h(r)\; dr
\\ + \frac{(2\pi)^{4it} \zeta(2s)\zeta(1-2it)}{\zeta(2s+1-2it)} 
N^{-2it} \prod_{p\mid N} \frac{(1-p^{-1})(1-p^{-2s})}{1-p^{-2s-1+2it}}
L(s+1/2-it, f\times \bar{g})
H_0(-2it; h). 
%\\ \times \frac{1}{\pi^2}\int_{-\infty}^\infty 
%h(r) r \tanh(\pi r)
%\frac{\Gamma\left(\frac{k}{2}-it+ir\right)\Gamma\left(\frac{k}{2}-it-ir\right)}
%{\Gamma\left(\frac{k}{2}+it+ir\right)\Gamma\left(\frac{k}{2}+it-ir\right)}\; dr. 
\end{multline}
%for the Rankin-Selberg convolution 
%\begin{equation}
%L(s, f\times \bar{g}) = \zeta^{(N)}(2s) \sum_{n=1}^\infty \frac{A(n) \overline{B(n)}}{n^{s}}
%\end{equation}
%converges absolutely for $\Re(s)>1$. 
Here $H_0(*; h)$ is given in \eqref{e:H0}. 
Note that $M(s, t)$ has a meromorphic continuation to all $s \in \C$.

\section{Analysis of $\OD^+(s, t)$ and the shifted double Dirichlet series}\label{s:OD+}

Our aim in \S\ref{ss:OD+shiftedDDS}-- \S\ref{sect:4.6} is to prove the following proposition. 
\begin{proposition}\label{prop:OD+_decomp}
For $t\in \R$, $\OD^+(s, t)$ has a meromorphic continuation to $\Re(s)\geq \frac{1}{2}$. 
For $\Re(s)=\frac{1}{2}$, we have
\begin{equation}\label{e:OD+_decomp}
\OD^+(s, t) = M^+_{\Omega}(s, t) + \scrE^+(s, t), 
\end{equation}
where $M^+_{\Omega}(s, t)$ is given in \eqref{e:MOmega_explicit} and 
\begin{multline}\label{e:E_OD+}
\scrE^+(s, t) = \OD^+_{\cusp, \intg}(s, t)+\OD^+_{\cusp, \res}(s, t) 
+ \OD^+_{\cont, \intg}(s, t) + \OD^+_{\cont, \res}(s, t)
\\ + \OD^+_{\Omega, \intg}(s, t) 
+ \sum_{\ell=1}^{\frac{k}{2}} \OD^+_{\Omega, \res, 1}(s, t; \ell) + \sum_{\ell=0}^{\frac{k}{2}} \OD^+_{\Omega, \res, 2}(s, t; \ell).
\end{multline}
Here all pieces in $\scrE^+(s, t)$ are given in \eqref{e:OD+cuspint}, \eqref{e:OD+cuspres}, 
\eqref{e:OD+contint}, \eqref{e:OD+contres}, \eqref{e:OD+Omegaint}, \eqref{e:OD+Omegares1} and \eqref{e:OD+Omegares2} 
respectively. 
\end{proposition}
This will be proved by obtaining a spectral expansion of $\OD^+(s, t)$, 
via an analysis of the shifted double Dirichlet series \eqref{e:Zdef} 
and a separation of the main term and the pieces that contribute to the error term. 

\subsection{$\OD^+(s, t)$ and the shifted double Dirichlet series $Z(s, v; it)$}\label{ss:OD+shiftedDDS}
Recalling \eqref{e:def_ODpm} and \eqref{e:L+holo_first}, 
for $s\in \C$ with sufficiently large real part 
we can bring the sum over $n$ inside, obtaining
\begin{multline}\label{e:OD+_Z_pre}
\OD^+(s, t)
%= \zeta^{(N)}(2s) \sum_{n=1}^\infty\frac{\overline{b(n)}}{n^{s+\frac{k-1}{2}}} L^+(1/2+it, f; n)
= i^k (2\pi)^{2it}
\frac{4}{2\pi i}\int_{(\sigma_u)} \frac{h(u/i)u}{\cos(\pi u)}
\frac{1}{\Gamma\left(-u+it+\frac{k}{2}\right)\Gamma\left(u+it+\frac{k}{2}\right)}
\\ \times \frac{1}{2\pi i} \int_{(1+\frac{k}{2}+\epsilon)}
\frac{\Gamma\left(v-it\right) \Gamma\left(v+it\right) \Gamma\left(-v+u+\frac{k}{2}\right)}
{\Gamma\left(v+u+1-\frac{k}{2}\right)}
Z(s, v; it) \; dv \; du, 
\end{multline}
where 
\begin{equation}\label{e:Zdef}
Z(s, v; it) = \zeta^{(N)}(2s) 
\sum_{n=1}^\infty \sum_{m=1}^\infty \frac{\sigma_{-2it}(m; N)m^{it} a(n+m)\overline{b(n)}}{m^{v}n^{s-v-\frac{1}{2}+k}}
\end{equation}
is a shifted double Dirichlet series. 
Here $1+\epsilon < \sigma_u< \frac{3}{2}$ and $\Re(s)>\Re(v)+1=2+\frac{k}{2}+\epsilon$. 
We first show that in this region the series $Z(s, v; it)$ converges absolutely.

For any $n, m\geq 1$, 
\begin{equation}
\sigma_{-2it}(m; N)m^{it} a(n+m)\overline{b(n)} \ll n^{k-1} m^{\frac{k-1}{2}}, 
\end{equation}
so, as $\Re(s) > 1 + \Re(v)$, and $\Re(v)= 1+\frac{k}{2}+\epsilon$,
\begin{equation}
\sum_{n=1}^\infty \sum_{m=1}^\infty \frac{\sigma_{-2it}(m; N)m^{it} a(n+m)\overline{b(n)}}{m^{v}n^{s-v-\frac{1}{2}+k}}
\ll %\sum_{n=1}^\infty \frac{1}{n^{\Re(s)-\Re(v)+\frac{1}{2}}} \sum_{m=1}^{\infty}\frac{1}{m^{\Re(v)-\frac{k-1}{2}}} \leq
\sum_{n=1}^\infty \frac{1}{n^{\frac{3}{2}}} \sum_{m=1}^\infty \frac{1}{m^{\frac{3}{2}+\epsilon}}, 
\end{equation}
so in this range $Z(s, v; it)$ converges absolutely.

Our intention is to take the shifted double Dirichlet series $Z(s, v; it)$ 
and consider first the inner sum over $n\geq 1$: 
\begin{equation}
Z(s, v; it) 
%= \zeta^{(N)}(2s) 
%\sum_{m=1}^\infty \frac{\sigma_{-2it}(m; N)m^{it}}{m^{v}} 
%\sum_{n=1}^\infty \frac{a(n+m)\overline{b(n)}}{n^{s-v-\frac{1}{2}+k}}
%\\ 
= \zeta^{(N)}(2s) 
\sum_{m=1}^\infty \frac{\sigma_{-2it}(m; N)m^{it}}{m^{v}} D(s-v+1/2; m), 
\end{equation}
where 
\begin{equation}
D(w; m) = \sum_{n=1}^\infty \frac{a(n+m)\overline{b(n)}}{n^{w+k-1}}. 
\end{equation}
is the shifted Dirichlet series studied in \cite{HHR}. 
Its meromorphic continuation to all $w\in \C$ is given in \cite[Proposition~5.1]{HHR}. 
We then obtain the analytic properties of $Z(s, v; it)$ from the analytic properties of $D(s-v+1/2; m)$.   
The meromorphic continuation of $Z(s, v; it)$ is expressed by a spectral expansion involving  the following $L$-functions: 
for $\Re(s)>1$, let 
\begin{equation}\label{e:def_cL_uj}
\scrL(s, it; \overline{u_j})
= \zeta^{(N)}(2s) \sum_{m=1}^\infty \frac{\sigma_{-2it}(m; N) m^{it} \overline{\rho_j(m)}}{m^s}
\end{equation}
and 
\begin{equation}\label{e:def_cL_cuspa}
\scrL_{\cuspa} (s, it; ir)
= \zeta^{(N)}(2s) \sum_{m=1}^\infty\frac{\sigma_{-2it}(m; N) m^{it}\overline{\tau_{\cuspa}(1/2+ir; m)}}{m^{s}}.
\end{equation}
The details are described in the remainder of this section.

Let
\begin{equation}\label{e:M_Gamma}
M(s, z/i) = \frac{\sqrt{\pi} 2^{\frac{1}{2}-s} \Gamma\left(s-\frac{1}{2}-z\right) \Gamma\left(s-\frac{1}{2}+z\right) 
\Gamma\left(1-s\right)}
{\Gamma\left(\frac{1}{2}-z\right)\Gamma\left(\frac{1}{2}+z\right)}
\end{equation}
and note that it has poles at $s=1/2\pm z-\ell$ for non-negative integers $\ell$.

Let 
\begin{equation}
V_{f, g} (z) = f(z)\overline{g(z)} y^{k}.
\end{equation}
In the following proposition, we describe the analytic properties of $Z(s, v; it)$.
\begin{proposition}\label{prop:Z}
For $t\in \R$, the function $Z(s, v; it)$ has a meromorphic continuation to $(s, v)\in \C^2$. 
The function $Z(s, v; it)$ has poles at $s-v=-\ell\pm ir_j$ for integers $0\leq \ell <-\Re(s-v)$ with residues
\begin{equation}\label{e:Zcuspres}
\frac{(4\pi)^k 2^{-\ell\pm ir_j}}{2\sqrt{\pi}\Gamma\left(\pm ir_j +k-\ell-\frac{1}{2}\right)}
\sum_j  (-1)^{\epsilon_j} \big(\Res_{s-v=-\ell\pm ir_j}M(s-v+1/2, r_j)\big)
\scrL(s, it; \overline{u_j})  \left<u_j, V_{f, g}\right>.
\end{equation}
In addition, $Z(s, v; it)$ has poles corresponding to those of $Z_{\cont, \res}(s, v; it)$ and $Z_{\Omega}(s, v; it)$, 
which are given below.

When $\Re(s-v) <-\frac{k}{2}$, $Z(s, v; it)$ is expressed by the following absolutely convergent spectral expansion: 
\begin{equation}\label{e:Z_spec}
Z(s,v; it) = Z_{\cusp}(s, v; it) +  Z_{\cont, \intg}(s, v; it) + Z_{\cont, \res}(s, v; it) + \delta_{\Re(s)<1} Z_{\Omega}(s, v; it)
\end{equation}
where 
\begin{equation}\label{e:Zcusp}
Z_{\cusp}(s, v; it)
%= \zeta^{(N)}(2s) \sum_{m=1}^\infty\frac{\sigma_{-2it}(m; N) m^{it}}{m^{v}} D_{\cusp}(s-v+1/2;m)
= \frac{(4\pi)^k 2^{s-v}}{2\sqrt{\pi}\Gamma\left(s-v+k-\frac{1}{2}\right)}
\sum_j  M(s-v+1/2, r_j) 
(-1)^{\epsilon_j} \scrL(s, it; \overline{u_j}) \left<u_j, V_{f, g}\right>,
\end{equation}
\begin{multline}\label{e:Zcontintg}
Z_{\cont, \intg}(s, v; it)
%= \zeta^{(N)}(2s) \sum_{m=1}^\infty\frac{\sigma_{-2it}(m; N) m^{it}}{m^{v}} D_{\cont} (s-v+1/2; m)
\\ = \frac{(4\pi)^{k} 2^{s-v}}{2\sqrt{\pi} \Gamma\left(s-v+k-\frac{1}{2}\right)}
\sum_{\cuspa}\frac{1}{4\pi i} \int_{(0)} 
M(s-v+1/2, z/i) \scrL_\cuspa(s, it; z) \left<E_{\cuspa}(*, 1/2+z), V_{f, g}\right>\; dz, 
\end{multline}
\begin{multline}\label{e:Zcontres}
Z_{\cont, \res}(s, v; it)
= \sum_{\ell=0}^{\lfloor-\Re(s-v)\rfloor} 
\frac{(4\pi)^{k} 2^{s-v}}{2\sqrt{\pi} \Gamma\left(s-v+k-\frac{1}{2}\right)}
\\ \times \big(\Res_{z=s-v-1/2+\ell} M(s-v+1/2, z/i)\big)
\sum_{\cuspa} \scrL_\cuspa(s, it; s-v+\ell-1/2) 
\left<E_{\cuspa}(*, s-v+\ell), V_{f, g}\right>
%\\ \times \frac{1}{2}
%\bigg\{
%\big(\Res_{z=s-v-1/2+\ell} M(s-v+1/2, z/i)\big)
%\sum_{\cuspa} \scrL_\cuspa(s, it; s-v+\ell-1/2) 
%\left<E_{\cuspa}(*, s-v+\ell), V_{f, g}\right>
%\\ + \big(\Res_{z=-s+v+1/2-\ell} M(s-v+1/2, z/i)\big) 
%\sum_{\cuspa} \scrL_\cuspa(s, it; 1-s+v-\ell-1/2) \left<E_{\cuspa}(*, 1-s+v-\ell), V_{f, g}\right>
%\bigg\},
\end{multline}
and
\begin{multline}\label{e:ZOmega}
Z_{\Omega}(s, v; it)
= \frac{(4\pi)^{k} 2^{s-v}}{2\sqrt{\pi} \Gamma\left(s-v+k-\frac{1}{2}\right)}
\zeta(2s-1)
\\ \times \bigg\{
% z=1-s+it
\sum_{\cuspa}
\frac{\scrP_{\cuspa}(s, it; 1-s+it)}{\prod_{p\mid N}(1-p^{1-2s+2it})}
\frac{\zeta(1+2it)M(s-v+1/2, (1-s+it)/i) }{\pi^{\frac{1}{2}-s+it} \Gamma\left(-\frac{1}{2}+s-it\right)} 
\left<E_{\cuspa}(*, 3/2-s+it), V_{f, g}\right>
% z=1-s-it
\\ + \sum_{\cuspa}
\frac{\scrP_{\cuspa}(s, it; 1-s-it)}{\prod_{p\mid N}(1-p^{1-2s-2it})}
\frac{\zeta(1-2it)M(s-v+1/2, (1-s-it)/i) }{\pi^{\frac{1}{2}-s-it} \Gamma\left(-\frac{1}{2}+s+it\right)} 
\left<E_{\cuspa}(*, 3/2-s-it), V_{f, g}\right>
\bigg\}.
\end{multline}

Here, as given in \eqref{e:scrL_fac_Eis}, 
\begin{equation}
\scrL_{\cuspa}(s, it; z)
= \scrP_{\cuspa}(s, it; z)
\frac{\zeta(s+it+z)\zeta(s-it+z)\zeta(s+it-z)\zeta(s-it-z)}{\pi^{-\frac{1}{2}+z} \Gamma\left(\frac{1}{2}-z\right) \zeta^{(N)}(1-2z)} 
\end{equation}
and $\scrP_{\cuspa}(s, it; z)$ is an Euler polynomial, given in \eqref{e:scrP_Eis}, which is $O_N(1)$ when $\Re(s)= 1/2$.  
Also
\begin{equation}
\delta_{\Re(s)<1} = \begin{cases} 1 & \text{ if } \Re(s)<1\\ 0 & \text{ otherwise.}\end{cases}
\end{equation}
\end{proposition}

The factorization of $\scrL_{\cuspa}(s, it; z)$ is given in Lemma~\ref{lem:scrL_fac_Eis}, 
based on Young's explicit description of Fourier coefficients of Eisenstein series \cite{Y19}. 
In particular, $\scrP_\cuspa(s, it; 1-s\pm it)$ is given in Corollary~\ref{cor:scrP_Eis_pm} and Corollary~\ref{cor:scrP_Eis_pm}. 
Note that the factor $\prod_{p\mid N} (1-p^{1-2s\pm 2it})$ divides $\scrP_{\cuspa}(s, it; 1-s \pm it)$. 
The proof of the proposition is virtually identical to the argument given in \cite[Proposition~7.1]{HHR} and consequently is omitted.  

To simplify the formula for $Z_{\cont, \res}(s, v; it)$ and $Z_{\Omega}(s, v; it)$, 
we also use the fact that 
\begin{equation}\label{e:series_Eis_cusosum_fe}
\sum_{\cuspa} \scrL_{\cuspa}(s, it; ir) E_{\cuspa}(z, 1/2+ir)
= \sum_{\cuspa} \scrL_{\cuspa}(s, it; -ir) E_{\cuspa}(z, 1/2-ir), 
\end{equation}
%\begin{equation}\label{e:Eis_cuspsum_fe}
%\sum_{\cuspa} \overline{\tau_{\cuspa}(1/2+ir; n)} E_{\cuspa}(z, 1/2+ir)
%= \sum_{\cuspa} \overline{\tau_{\cuspa}(1/2-ir; n)} E_{\cuspa}(z, 1/2-ir)
%\end{equation}
induced from the functional equation of the sum of products of Eisenstein series over cusps \cite[(6.22')]{Iwa02}.

\begin{remark}
Inside of the integral in \eqref{e:Zcontintg}, note that $M(s-v+1/2, z/i)$ has poles at $s-v=-\ell\pm z$ 
for integers $0\leq \ell < -\Re(s-v)$. 
Moreover, $\scrL_{\cuspa}(s, it; z)$ has poles at $s=1-z\pm it$ and $s=1+z\pm it$. 
To get the meromorphic continuation of the function $Z(s, v; it)$ we need to separate those residues from the original integral, again using the same technique as in \cite[Proposition~7.1]{HHR}.   This leads to the construction of the separate terms $Z_{\cont, \intg}(s, v; it) + Z_{\cont, \res}(s, v; it) + Z_{\Omega}(s, v;it)$. 
Again, see \cite[\S7.1]{HHR} for further details. 
\end{remark}

\subsection{Factorization of Dirichlet series $\scrL(s, it; u_j)$ and $\scrL_{\cuspa}(s, it; z)$}\label{ss:factorization_scrL}
We will find it convenient to compute the factorization of the Dirichlet series $\scrL(s, it; u_j)$ 
and $\scrL_{\cuspa}(s, it; z)$. 

Let $\{u_j\}_{j\geq 1}$ be an orthonormal basis of Maass cusp forms of level $N$ as in \S\ref{s:intro}. 
Let $\scrA_{r}(N)$ be the space of Maass cuspforms of level $N$ with the Laplace eigenvalue $\frac{1}{4}+r^2$. 
Then we have the following orthogonal decomposition as in \cite{AL78}
\begin{equation}\label{e:cusp_furtherdecomp_new}
\scrA_{r}(N) = \bigoplus_{L\mid N} \bigoplus_{f\in \scrA^{\new}_{r}(L)} \Span\left\{f(dz)\;:\; d\mid \frac{N}{L}\right\}, 
\end{equation}
where $\scrA^{\new}_r(L)$ is the space of newforms of level $L$ with the Laplace eigenvalue $\frac{1}{4}+r^2$.
When a Maass cuspform $u_j$ with Laplace eigenvalue $\frac{1}{4}+r_j^2$ of level $N$, 
which is $L^2$-normalized for level $N$, is not a newform, 
then there exists $L\mid N$, a newform $u_{L, j}\in \scrA_r^{\new}(L)$, and constants $c_{L}(r_j; d)$ for $d\mid \frac{N}{L}$ 
such that 
\begin{equation}\label{e:uj_L_Fourier}
u_j(z) = \sum_{d\mid \frac{M}{L}} c_L(r_j; d) u_{L, j}(dz)
= \sum_{m\neq 0} \sum_{d\mid \frac{M}{L}} c_L(r_j; d) \rho_{L, j}(m) \sqrt{dy} K_{ir_j}(2\pi |m|dy)e^{2\pi imdx}.
\end{equation}
Here we assume that $u_{L, j}$ is $L^2$-normalized for level $L$, i.e., 
\begin{equation}
\left<u_{L, j}, u_{L, j}\right>_L=\int_{\Gamma_0(L)\bsl \HH} |u_{L, j}(z)|^2 \; d\mu(z) = 1. 
\end{equation}
 We also define the $L$-function for $u_j$ in terms of its Hecke eigenvalues, i.e., 
as the $L$-function for the newform $u_{L, j}$: 
\begin{equation}
L(s, u_j ) = L(s, u_{L, j}) = \sum_{m=1}^\infty \frac{\lambda_{u_j, L}(m)}{m^s}, 
\end{equation}
where $\lambda_{u_j, L}(m)$ is the $m$th Hecke eigenvalue of $u_{L, j}$. 

By \eqref{e:def_cL_uj}, recall that 
\begin{equation}
\scrL(s, it; u_j)
= \zeta^{(N)}(2s) \sum_{m=1}^\infty \frac{\sigma_{-2it}(m; N) m^{it} \rho_j(m)}{m^s}
\end{equation}
%We let 
%\begin{equation}\label{e:cj_s}
%c(s; u_j) = \sum_{d\mid \frac{N}{L}}c_L(r_j; d) d^{s}. 
%\end{equation}
For convenience, we set $\lambda_{L, j}(m)=0$ when $m\notin \Z$.
With this notation, we get the following lemma. 
\begin{lemma}\label{lem:scrL_fac_cusp}
We now get
\begin{equation}\label{e:scrL_fac_cusp}
\scrL(s, it; u_j)
= L(s+it, u_j) L(s-it, u_j) 
N^{-s-it}\sum_{d\mid\frac{N}{L}} c_{L}(r_j; d) \rho_{L, j}(1) d^{\frac{1}{2}-it} 
\scrP_d(s, it; u_j), 
\end{equation}
where 
\begin{multline}
\scrP_d(s, it; \overline{u_j})
= 
\prod_{p\mid L} 
{\lambda_{L, j}(p^{\ord_p(\frac{N}{d})-1})} (-p^{-1+s-it}+ {\lambda_{L, j}(p)})
\\ \times \prod_{\substack{p\nmid L\\ p\mid \frac{N}{d}}} 
\bigg\{
\bigg(p-{\lambda_{L, j}(p)} p^{-s-it}- ((p-1)(1+p^{-2it})-p^{-4it}) + p^{1-2s}\bigg) 
p^{-1} {\lambda_{L, j}(p^{\ord_p(\frac{N}{d})-2})} 
+ {\lambda_{L, j}(p^{\ord_p(\frac{N}{d})})} \bigg\}
\\ \times \prod_{\substack{p\nmid L\\ p\mid N, p\nmid \frac{N}{d}}}
p^{-1} \big( {\lambda_{L, j}(p)} p^{-s-it} -(1+p^{-2it})p^{-2s}\big)
\end{multline}
\end{lemma}

\begin{proof}
The proof consists of computing the Euler product factorization of $\scrL(s, it; u_j)$ by applying the decomposition \eqref{e:uj_L_Fourier} 
and recalling the definition of $\sigma_{-2it}(m; N)$ given in \eqref{e:sigma_1N}. 
We omit the details. 

\end{proof}
 
%\begin{lemma}
%Similarly, for each cusp $\cuspa$ for $\Gamma_0(N)\bsl \HH$, we have 
%\begin{equation}\label{e:cLcuspa_fact}
%\scrL_{\cuspa} (s, it; z)
%= \zeta(s+z+it)\zeta(s+z-it)\zeta(s-z+it)\zeta(s-z-it) \frac{P_{\cuspa}(s, z, it)}{\zeta^*(1+2z)}. 
%\end{equation}
%Here $P_{\cuspa}(s, z, it)$ is an Euler polynomial similar to $P_d(s, it; u_j; p)$ and is $O_N(1)$ when $\Re(s)= 1/2$.
%
%In particular, when $N=1$, $\sigma_{-2it}(m; N)=\sigma_{-2it}(m)$ and we have 
%\begin{equation}\label{e:cL_cusp_fact_N=1}
%\scrL(s, it; u_j)=\rho_j(1) L(s+it) L(s-it)
%\end{equation}
%and 
%\begin{equation}\label{e:cLcuspa_fact_N=1}
%\scrL(s, it; z) = \scrL_\infty(s, it; z) 
%= \frac{2\pi^{\frac{1}{2}+z}}{\Gamma\left(\frac{1}{2}+z\right)\zeta(1+2z)}\zeta(s+z+it) \zeta(s-z+it) \zeta(s+z-it)\zeta(s-z-it).
%\end{equation}
%
%When $N$ is square-free, each cusp $\cuspa$ can be parameterized as $1/a$ for $a\mid N$, and we get 
%\begin{multline}\label{e:P1/a}
%P_{1/a}(s, z, it)
%= \frac{2 N^{-\frac{1}{2}-z}a^{\frac{1}{2}+z}}{\prod_{p\mid N}(1-p^{-1-2z})}
%\\ \times \prod_{p\mid a} p^{-2it-2z-1} 
%\bigg(p^{-1}+p^{-s}p^{it+z}\big(-2+p-p^{-2z}-p^{-2it}\big)
%+ p^{-2s}\big(2-p^{-1}+p^{2z}+p^{2it}\big)
%-p^{-3s} p^{z+it+1}\bigg)
%\\ \times \prod_{p\mid \frac{N}{a}}p^{-2it}(p^{it+z-s}-p^{-1})(1-p^{-it-z-s})(1-p^{it-z-s}).
%\end{multline}
%\end{lemma}

%We will obtain the factorization of $\scrL_{\cuspa}(s, it; z)$. in the next section. 
%\subsection{Factorization of Dirichlet series $\scrL_{\cuspa}(s, it; z)$}\label{ss:factorization_Eis}
To study $\scrL_{\cuspa}(s, it; z)$ we first give 
%: Young's
explicit descriptions of the Eisenstein series $E_{\cuspa}(z, s)$ defined in \S\ref{ss:basis_L2}, following \cite{Y19}. 
A complete set of inequivalent cusps for $\Gamma_0(N)$ is given by $\frac{1}{ca}$ where $a\mid N$ 
and $c\bmod{\gcd(a, N/a)}$, $\gcd(c, \gcd(a, N/a))=1$ and we choose the representative $c$ satisfying $\gcd(c, N)=1$. 
(One can always choose such a representative.)
By \cite[Theorem~6.1]{Y19}, taking the central character as the trivial character mod $N$, for $n\neq 0$, 
we get 
\begin{multline}\label{e:tau_1/ca_n}
\tau_{\frac{1}{ca}}(s; n)
= \left(\frac{N}{\gcd(a, N/a)}\right)^{-s} \frac{1}{\varphi(\gcd(a, N/a))}
\sum_{q\mid \gcd(a, N/a)} \sum_{\substack{\chi\bmod{q}\\ \text{ primitive }}} \overline{\chi(-c)}
\frac{q^{-s} \tau(\chi)}{\pi^{-s} \Gamma\left(s\right) L^{(N)}(2s, \chi^2)} 
\\ \times \sum_{\substack{\ell\mid a, b\mid\frac{N}{a} \\ \gcd(b\ell, q)=1\\ b\frac{a}{q\ell}\mid n}} 
\frac{\mu(\ell)\mu(b) \chi(\ell b)}{(\ell b)^s} 
2\left(b\frac{a}{q\ell}\right)^{\frac{1}{2}} 
\lambda_{\chi}\left(\frac{n}{b\frac{a}{q\ell}}, s\right) 
\end{multline}
where 
%\begin{multline}\label{e:E_1/ca}
%E_{\frac{1}{ca}} (z, s)
%= \left(\frac{N}{\gcd(N/a, a)}\right)^{-s} \frac{1}{\varphi(\gcd(a, N/a))}
%\sum_{q\mid \gcd(N/a, a)} \sum_{\substack{\chi\bmod{q}\\ \text{ primitive }}} \overline{\chi(-c)}
%\frac{q^{-s} \tau(\chi)}{\pi^{-s} \Gamma\left(s\right) L^{(N)}(2s, \chi^2)} 
%\\ \times \sum_{\substack{\ell\mid a\\ \gcd(\ell, q)=1}} \sum_{\substack{b\mid \frac{N}{a}\\ \gcd(b, q)=1}}
%\frac{\mu(\ell)\mu(b) \chi(\ell b)}{(\ell b)^s} E_{\chi}^*\left(\frac{ba}{\ell q}z, s\right), 
%\end{multline}
%where $E_\chi^*(z, s)$ is the complete Eisenstein series attached to $\chi$, and its Fourier expansion is given in \cite[Propositon~4.1]{Y19}: 
%\begin{multline}\label{e:Echi_Fourier}
%E_\chi^*(z, s) = \delta_{q=1} \bigg(\pi^{-s}\Gamma\left(s\right)\zeta(2s) y^s 
%+ \pi^{-1+s} \Gamma\left(1-s\right) \zeta(2-2s) y^{1-s} \bigg)
%\\ + 2\sqrt{y} \sum_{n\neq 0} \lambda_\chi(n, s) K_{s-1/2}(2\pi |n|y) e^{2\pi inx}. 
%\end{multline}
%Here 
\begin{equation}\label{e:lambda_chi}
\lambda_\chi(n, s) 
%= \chi(\sgn(n)) |n|^{s-\frac{1}{2}} \sum_{d\mid |n|} \chi(d) \overline{\chi(|n|/d)} d^{-2s+1}
= %\chi(\sgn(n)) 
\overline{\chi(n)} |n|^{s-\frac{1}{2}} \sum_{d\mid |n|} \chi(d)^2 d^{-2s+1}. 
\end{equation}
Note that $\lambda_\chi(n, s)=0$ unless $\gcd(n, q)=1$. 
%When $\gcd(n, q)=1$, we can write 
%\begin{equation}\label{e:lambda_chi}
%\lambda_\chi(n, s) = \chi(\sgn(n)) n^{s-\frac{1}{2}} \overline{\chi(n)} 
%\sum_{d\mid n} \chi(d)^2 d^{-2s+1}. 
%\end{equation}
% Also note that in the last line in \eqref{e:E_1/ca}, for each $\ell\mid a$, since $\gcd(\ell, q)=1$, $\frac{a}{\ell q}\in \Z$.
Moreover we also observe that 
\begin{equation}\label{e:tau_1/ca_cplxcong}
\overline{\tau_{\frac{1}{ca}}(1/2+ir; n)}
= \tau_{\frac{1}{ca}}(1/2-ir; -n). 
\end{equation}

By \eqref{e:def_cL_cuspa}, recall that 
\begin{equation}
\scrL_{\cuspa} (s, it; z)
= \zeta^{(N)}(2s) \sum_{m=1}^\infty\frac{\sigma_{-2it}(m; N) m^{it}\overline{\tau_{\cuspa}(1/2+z; m)}}{m^{s}}.
\end{equation}

\begin{lemma}\label{lem:scrL_fac_Eis}
With the above parameterization of cusps for $\Gamma_0(N)$, for $a\mid N$ 
and $c\mod\gcd(a, N/a)$ with $\gcd(c,\gcd(a, N/a))=1$ (and chosen as $\gcd(c, N)=1$), 
at $\cuspa=\frac{1}{ca}$, 
we have 
\begin{equation}\label{e:scrL_fac_Eis}
\scrL_{\frac{1}{ca}}(s, it; ir)
= \scrP_{\frac{1}{ca}}(s, it; ir)
\frac{\zeta(s+it+ir)\zeta(s-it+ir)\zeta(s+it-ir)\zeta(s-it-ir)}{\pi^{-\frac{1}{2}+ir} \Gamma\left(\frac{1}{2}-ir\right) \zeta^{(N)}(1-2ir)} 
\end{equation}
where 
\begin{equation}\label{e:scrP_Eis}
\scrP_{\frac{1}{ca}}(s, it; ir) = 2\frac{N^{-2it}}{\prod_{p\mid N} p} 
\left(\frac{N}{\gcd(N/a, a)}\right)^{-\frac{1}{2}+ir} \frac{a^{-s+\frac{1}{2}+it} }{\varphi(\gcd(a, N/a))}
\prod_{p\mid N} \scrP_{p^{\ord_p(N)}}(s, it; ir; a). 
\end{equation}
Let 
\begin{equation}
\zeta_p(s, it, ir) = (1-p^{-s-it-ir}) (1-p^{-s+it-ir}) (1-p^{-s-it+ir}) (1-p^{-s+it+ir}). 
\end{equation}
When $p\mid \gcd(a, N/a)$, 
\begin{multline}\label{e:scrP_cusp_gcd}
\scrP_{p^{\ord_p(N)}}(s, it; ir; a)
%= 
%\delta_{p^2\mid \frac{N}{a}}  \zeta_p(s, it, ir) 
%p^{-s+it+ir} p^{-(\ord_p(N/a)-2)(s-it+ir)} \sigma_{2ir}(p^{\ord_p(N/a)-2})
%\\ +  \zeta_p(s, it, ir) 
%\big(-p^{-1+2ir} + \frac{p^{-2it}(1-p^{1+2it})(1-p^{-s-it+ir})+(p-1)(1-p^{-s+it+ir})}{1-p^{-2it}}\big) 
%\\ \times p^{-(\ord_p(N/a)-1)(s-it+ir)} \sigma_{2ir}(p^{\ord_p(N/a)-1})
%\\+ p^{-\ord_p(N/a)(s-it+ir)} \frac{(1-p^{-s-it+ir})(1-p^{-s+it+ir})}{1-p^{-2it}}
%\\ \times \bigg(p^{-4it} (1-p^{1+2it}) 
%(1-p^{-1+s+it+ir})(1-p^{-s+it-ir})
%(\sigma_{2ir}(p^{\ord_p(N/a)})- p^{-s-it+ir}\sigma_{2ir}(p^{\ord_p(N/a)-1}))\big)
%\\ + (p-1) (1-p^{-1+s-it+ir})(1-p^{-s-it-ir})
%(\sigma_{2ir}(p^{\ord_p(N/a)})- p^{-s+it+ir}\sigma_{2ir}(p^{\ord_p(N/a)-1}))\bigg). 
= \zeta_p(s, it, ir) 
p^{-(\ord_p(N/a)-1)(s-it+ir)} 
\\ \times \bigg\{ -1
+ \big(1-p^{-1+2ir} + \frac{p^{-2it}(1-p^{1+2it})(1-p^{-s-it+ir})+(p-1)(1-p^{-s+it+ir})}{1-p^{-2it}}\big) 
p^{-2ir} (\sigma_{2ir}(p^{\ord_p(N/a)})-1)
\bigg\}
\\+ p^{-\ord_p(N/a)(s-it+ir)} \frac{(1-p^{-s-it+ir})(1-p^{-s+it+ir})}{1-p^{-2it}}
\\ \times \bigg\{p^{-4it} (1-p^{1+2it}) 
(1-p^{-1+s+it+ir})(1-p^{-s+it-ir})
(\sigma_{2ir}(p^{\ord_p(N/a)})(1- p^{-s-it-ir}) + p^{-s-it-ir})\big)
\\ + (p-1) (1-p^{-1+s-it+ir})(1-p^{-s-it-ir})
(\sigma_{2ir}(p^{\ord_p(N/a)})(1- p^{-s+it-ir}) + p^{-s+it-ir} )\bigg\}. 
\end{multline} 
When $p\mid a$ and $p\nmid\frac{N}{a}$, 
\begin{multline}\label{e:scrP_cusp_a}
\scrP_{p^{\ord_p(N)}}(s, it; ir; a)
= (1-p^{-2it})^{-1} 
\bigg\{p^{-4it} (1-p^{1+2it})(1-p^{-1+s+it+ir}) (1-p^{-s+it-ir})(1-p^{-s+it+ir})
\\ + (p-1) (1-p^{-1+s-it+ir}) (1-p^{-s-it-ir}) (1-p^{-s-it+ir})\bigg\}. 
\end{multline} 
When $p\mid \frac{N}{a}$ and $p\nmid a$, 
\begin{multline}\label{e:scrP_cusp_Na}
\scrP_{p^{\ord_p(N)}}(s, it; ir; a)
%= \delta_{p^2 \mid \frac{N}{a}} 
%\zeta_p(s, it, ir) p^{-s+it+ir} \sigma_{2ir}(p^{\ord_p(N/a)-2})p^{-(\ord_p(N/a)-2)(s-it+ir)}
%\\ + \frac{p^{-(\ord_p(N/a)-1)(s-it+ir)} (1-p^{-s-it+ir})(1-p^{-s+it+ir})}{(1-p^{-2it})} 
%\\ \times \bigg\{ 
%p^{-2it} (1-p^{1+2it})
%(1-p^{-s+it-ir})
%\big(\sigma_{2ir}(p^{\ord_p(N/a)-1})-p^{-s-it+ir}\sigma_{2ir}(p^{\ord_p(N/a)-2})\big)
%\\ + (p-1) 
%(1-p^{-s-it-ir})
%\big(\sigma_{2ir}(p^{\ord_p(N/a)-1})-p^{-s+it+ir}\sigma_{2ir}(p^{\ord_p(N/a)-2})\big)
%\bigg\}.
\\ = p^{-(\ord_p(N/a)-1)(s-it+ir)} (1-p^{-s-it+ir})(1-p^{-s+it+ir}) 
\bigg\{-(1-p^{-s+it-ir})(1-p^{-s-it-ir})
\\ + \frac{p^{-2it} (1-p^{1+2it})(1-p^{-s+it-ir}) p^{-s-it-ir}
+ (p-1) (1-p^{-s-it-ir})p^{-s+it-ir}}{(1-p^{-2it})} 
\bigg\}. 
\end{multline} 
\end{lemma}

\begin{proof}
The proof consists of computing each Euler product factor for a prime $p$, 
and applying \eqref{e:tau_1/ca_n}. 
We omit the detailed computations. 
\end{proof}

Taking $ir=1-s+it$ and $ir=1-s-it$ in \eqref{e:scrP_cusp_gcd}, gives us the following corollary.
\begin{corollary}\label{cor:scrP_Eis_pm}
When $a\mid N$ and $a< N$, we have $\scrP_{\frac{1}{ca}}(s, it; 1-s+it)=0$ for any $c\bmod{\gcd(a, N/a)}$, $\gcd(c, \gcd(a, N/a))=1$. 
When $a=N$, taking $ir=1-s+it$, 
\begin{equation}
\frac{\scrP_{\frac{1}{N}}(s, it; 1-s+it)}{\prod_{p\mid N} (1-p^{1-2s+2it})} 
%= 2\frac{N^{-2it}}{\prod_{p\mid N} p} 
%\left(\frac{N}{\gcd(N/a, a)}\right)^{\frac{1}{2}-s+it} \frac{a^{-s+\frac{1}{2}+it} }{\varphi(\gcd(a, N/a))}
%\prod_{p\mid N} \frac{\scrP_{p^{\ord_p(N)}}(s, it; 1-s+it; a)}{1-p^{1-2s+2it}}
= 2 N^{1-2s}\prod_{p\mid N} (1-p^{-1-2it})(1-p^{-1}). 
\end{equation}
Taking $ir = 1-s-it$, we get
\begin{multline}
\frac{\scrP_{\frac{1}{ca}}(s, it; 1-s-it)}{\prod_{p\mid N} (1-p^{1-2s-2it})}
= 2N^{-\frac{1}{2}-s-it}
\left(\frac{a}{\gcd(a, N/a)}\right)^{-s+\frac{3}{2}-it}
\prod_{p\mid \frac{N}{a}} (1-p^{1-2s})(1-p^{-2it})
\prod_{\substack{p\mid a\\ p\nmid\frac{N}{a}}} (1-p^{-1})^2.
\end{multline}
\end{corollary}

\subsection{Spectral description of $\OD^+(s, t)$}

Recalling \eqref{e:OD+_Z_pre}, for $\Re(s)>2+\frac{k}{2}+\epsilon$, we have
\begin{multline}
\OD^+(s, t) = i^k (2\pi)^{2it}
\frac{4}{2\pi i}\int_{(\sigma_u)} \frac{h(u/i)u}{\cos(\pi u)}
\frac{1}{\Gamma\left(-u+it+\frac{k}{2}\right)\Gamma\left(u+it+\frac{k}{2}\right)}
\\ \times \frac{1}{2\pi i} \int_{(1+\frac{k}{2}+\epsilon)}
\frac{\Gamma\left(v-it\right) \Gamma\left(v+it\right) \Gamma\left(-v+u+\frac{k}{2}\right)}
{\Gamma\left(v+u+1-\frac{k}{2}\right)}
Z(s, v; it) \; dv \; du.
\end{multline}
Here $1+\epsilon < \sigma_u < \frac{3}{2}$. 
We now move the $v$ line of integration from $\Re(v)=1+\frac{k}{2}+\epsilon$ to $\Re(v)=\frac{1}{4}$. 
Since $\Re(s-v)>1/2$, by Proposition~\ref{prop:Z}, the function $Z(s, v; it)$ is analytic over this region. 
Then we take $\frac{1}{2} < \Re(s) < 1$ and get 
\begin{multline}\label{e:OD+_nonpolar}
\OD^+(s, t) - \OD^+_{\Omega}(s, t) 
= i^k (2\pi)^{2it}
\frac{4}{2\pi i}\int_{(\sigma_u)} \frac{h(u/i)u}{\cos(\pi u)}
\frac{1}{\Gamma\left(-u+it+\frac{k}{2}\right)\Gamma\left(u+it+\frac{k}{2}\right)}
\\ \times \frac{1}{2\pi i} \int_{(1/4)}
\frac{\Gamma\left(v-it\right) \Gamma\left(v+it\right) \Gamma\left(-v+u+\frac{k}{2}\right)}
{\Gamma\left(v+u+1-\frac{k}{2}\right)}
\big(Z(s, v; it)-Z_{\Omega}(s, v; it)\big) \; dv \; du, 
\end{multline}
where 
\begin{multline}
\OD^+_{\Omega}(s, t) 
= i^k (2\pi)^{2it}
\frac{4}{2\pi i}\int_{(\sigma_u)} \frac{h(u/i)u}{\cos(\pi u)}
\frac{1}{\Gamma\left(-u+it+\frac{k}{2}\right)\Gamma\left(u+it+\frac{k}{2}\right)}
\\ \times \frac{1}{2\pi i} \int_{(1/4)}
\frac{\Gamma\left(v-it\right) \Gamma\left(v+it\right) \Gamma\left(-v+u+\frac{k}{2}\right)}
{\Gamma\left(v+u+1-\frac{k}{2}\right)}
Z_{\Omega}(s, v; it) \; dv \; du. 
\end{multline}
Here $Z_{\Omega}(s, v; it)$ is given in \eqref{e:ZOmega}. 
With the given $s$, we now move back the $v$ line of integration to $\Re(s-v) < -\frac{k}{2}$ 
so we can write $Z(s, v; it)-Z_{\Omega}(s, v; it)$ as the absolutely convergent spectral expansion in \eqref{e:Z_spec}. 

Assume that $\Re(s)=1/2$ and move the $v$ line of integration in \eqref{e:OD+_nonpolar} to $\Re(v)=\frac 54+\frac{k}{2}$. 
We pass over the poles described in Proposition~\ref{prop:Z} 
and obtain the following proposition. 
 We use the notations in Proposition~\ref{prop:Z} as $Z_{*}(s, v; it)$ 
and let 
\begin{multline}\label{e:OD+cuspint}
\OD^+_{\cusp, \intg}(s, t) = i^k (2\pi)^{2it}
\frac{4}{2\pi i}\int_{(\sigma_u)} \frac{h(u/i)u}{\cos(\pi u)}
\frac{1}{\Gamma\left(-u+it+\frac{k}{2}\right)\Gamma\left(u+it+\frac{k}{2}\right)}
\\ \times \frac{1}{2\pi i} \int_{\left(\frac54+\frac k2 \right)}
\frac{\Gamma\left(v-it\right) \Gamma\left(v+it\right) \Gamma\left(-v+u+\frac{k}{2}\right)}
{\Gamma\left(v+u+1-\frac{k}{2}\right)}
Z_{\cusp}(s, v; it) \; dv \; du, 
\end{multline}
\begin{multline}\label{e:OD+cuspres}
\OD^+_{\cusp, \res}(s, t) = i^k (2\pi)^{2it}
\sum_{\ell=0}^{\frac{k}{2}} 
\frac{4}{2\pi i}\int_{(\sigma_u)} \frac{h(u/i)u}{\cos(\pi u)}
\frac{1}{\Gamma\left(-u+it+\frac{k}{2}\right)\Gamma\left(u+it+\frac{k}{2}\right)}
\\ \times \sum_j  (-1)^{\epsilon_j} \frac{(4\pi)^k 2^{-\ell\pm ir_j}\big(\Res_{s-v=-\ell\pm ir_j}M(s-v+1/2, r_j)\big) }
{2\sqrt{\pi}\Gamma\left(\pm ir_j +k-\ell-\frac{1}{2}\right)}
\\ \times \frac{\Gamma\left(s+\ell\mp ir_j-it\right) \Gamma\left(s+\ell\mp ir_j+it\right) \Gamma\left(-s-\ell\pm ir_j+u+\frac{k}{2}\right)}
{\Gamma\left(s+\ell\mp ir_j +u+1-\frac{k}{2}\right)} 
\scrL(s, it; \overline{u_j}) \left<u_j, V_{f, g}\right>
\; du, 
\end{multline}
\begin{multline}\label{e:OD+contint}
\OD^+_{\cont, \intg}(s, t) = i^k (2\pi)^{2it}
\frac{4}{2\pi i}\int_{(\sigma_u)} \frac{h(u/i)u}{\cos(\pi u)}
\frac{1}{\Gamma\left(-u+it+\frac{k}{2}\right)\Gamma\left(u+it+\frac{k}{2}\right)}
\\ \times \frac{1}{2\pi i} \int_{\left(\frac54+\frac k2 \right)}
\frac{\Gamma\left(v-it\right) \Gamma\left(v+it\right) \Gamma\left(-v+u+\frac{k}{2}\right)}
{\Gamma\left(v+u+1-\frac{k}{2}\right)}
Z_{\cont}(s, v; it) \; dv \; du, 
\end{multline}
\begin{multline}\label{e:OD+contres}
\OD^+_{\cont, \res}(s, t) = i^k (2\pi)^{2it}
\frac{4}{2\pi i}\int_{(\sigma_u)} \frac{h(u/i)u}{\cos(\pi u)}
\frac{1}{\Gamma\left(-u+it+\frac{k}{2}\right)\Gamma\left(u+it+\frac{k}{2}\right)}
\\ \times \frac{1}{2\pi i} \int_{\left(\frac54+\frac k2 \right)}
\frac{\Gamma\left(v-it\right) \Gamma\left(v+it\right) \Gamma\left(-v+u+\frac{k}{2}\right)}
{\Gamma\left(v+u+1-\frac{k}{2}\right)}
Z_{\cont, \res}(s, v; it) \; dv \; du
\end{multline}
and 
\begin{multline}\label{e:OD+Omega}
\OD^+_{\Omega}(s, t) = i^k (2\pi)^{2it}
\frac{4}{2\pi i}\int_{(\sigma_u)} \frac{h(u/i)u}{\cos(\pi u)}
\frac{1}{\Gamma\left(-u+it+\frac{k}{2}\right)\Gamma\left(u+it+\frac{k}{2}\right)}
\\ \times \frac{1}{2\pi i} \int_{(1/4)}
\frac{\Gamma\left(v-it\right) \Gamma\left(v+it\right) \Gamma\left(-v+u+\frac{k}{2}\right)}
{\Gamma\left(v+u+1-\frac{k}{2}\right)}
Z_{\Omega}(s, v; it) \; dv \; du.
\end{multline}
Now, in the following proposition, we decompose $\OD^+(s, t)$.

\begin{proposition}\label{OD+def}
On $\Re(s)=1/2$, we have 
\begin{equation}
\OD^+(s, t) = \OD^+_{\cusp}(s, t) + \OD^+_{\cont}(s, t)
\end{equation}
where 
\begin{equation}
\OD^+_{\cusp}(s, t) = \OD^+_{\cusp, \intg}(s, t) + \OD^+_{\cusp, \res}(s, t)
\end{equation}
and 
\begin{equation}
\OD^+_{\cont}(s, t) = \OD^+_{\cont, \intg} (s, t) + \OD^+_{\cont, \res}(s, t) + \OD^+_{\Omega}(s, t). 
%+ \OD^+_{\Omega, \intg}(s, t) + \OD^+_{\Omega, \res}(s, t).
\end{equation}
\end{proposition}

In the following two sections we obtain upper bounds for the $\OD^+_*(s, t)$.

\subsection{Upper bounds for $\OD_{\cusp, \intg}^+(s, t)$ and $\OD_{\cont, \intg}^+(s, t)$}\label{ss:OD+intg_upper}
Set $s = \frac{1}{2} -it'$, with $|t'| = |t|$ and $|t| = T^\beta$, with $0\leq \beta<1$.
The objective of this section is to prove: 
\begin{proposition}\label{prop:ub_OD+ccint}
Fix $\alpha$, $\frac{1}{3}< \alpha < \frac{2}{3}$.  
If $t = t'$, and $\alpha > \max\{\frac{1}{3}, \frac{\beta}{2}\}$, 
or if $t = -t'$ and $\alpha >\max\{2\beta-1, \frac{\beta}{2}\}$,
then  
\begin{equation}\label{e:OD+intg_upper}
\OD_{\cusp, \intg}^+(s, t),\; \OD_{\cont, \intg}^+(s, t)\ll T^{1+\epsilon}
\end{equation}
\end{proposition}

Before proving the above proposition, we clarify notations. 
In general, we will often use two complex variables $u$, $v$, and will write $u = \sigma_u + i \gamma$ and $v = \sigma_v + ir$.    Also, recall that because of the definition of $h=h_{T, \alpha}$ in \eqref{e:hdef}, which appears in the integral as $h(u/i)$,  the variable $\gamma= \Im(u)$ is effectively restricted to the range
\begin{equation}
\left|T-|\gamma|\right| \ll T^\alpha.
\end{equation}
We should remark that our work doesn't depend on the specific choice of $h$. 
This particular $h$ is relevant only in the sections where explicit error bounds are computed.

Recalling \eqref{e:OD+cuspint}, \eqref{e:Zcusp} and \eqref{e:M_Gamma}, 
\begin{multline}\label{e:OD+cuspint_explicit}
\OD^+_{\cusp, \intg}(s, t) = \frac{i^k (2\pi)^{2it} (4\pi)^k}{2}
\frac{4}{2\pi i}\int_{(\sigma_u)} \frac{h(u/i)u}{\cos(\pi u)}
\frac{1}{\Gamma\left(-u+it+\frac{k}{2}\right)\Gamma\left(u+it+\frac{k}{2}\right)}
\\ \times \frac{1}{2\pi i} \int_{\left(\frac54+\frac k2 \right)}
\frac{\Gamma\left(v-it\right) \Gamma\left(v+it\right) \Gamma\left(-v+u+\frac{k}{2}\right)}
{\Gamma\left(v+u+1-\frac{k}{2}\right)}
\\ \times \frac{1}{\Gamma\left(s-v+k-\frac{1}{2}\right)}
\sum_j  (-1)^{\epsilon_j} \scrL(s, it; \overline{u_j}) 
\frac{\Gamma\left(s-v-ir_j\right) \Gamma\left(s-v+ir_j\right) 
\Gamma\left(\frac{1}{2}-s+v\right)}
{\Gamma\left(\frac{1}{2}-ir_j\right)\Gamma\left(\frac{1}{2}+ir_j\right)}
\left<u_j, V_{f, g}\right>
\; dv \; du.
\end{multline}
Let $T>0$ and recall that $v = \sigma_v+ir$, $|t| = |t'| =  T^\beta$ (here we take $s=\frac{1}{2}-it'$), and $u = \sigma_u + i\gamma$
with $1+\epsilon < \sigma_u< \frac{3}{2}$. 
By Stirling's formula, the exponential part of 
the asymptotic behavior of the ratio of gamma functions in \eqref{e:OD+cuspint_explicit} 
%\begin{multline}
%\frac{1}{\cos(\pi u)}
%\frac{\Gamma\left(v-it\right) \Gamma\left(v+it\right) }{\Gamma\left(-u+it+\frac{k}{2}\right) \Gamma\left(u+it+\frac{k}{2}\right)}
%\frac{\Gamma\left(-v+u+\frac{k}{2}\right)}
%{\Gamma\left(v+u+1-\frac{k}{2}\right)}
%\frac{\Gamma\left(\frac{1}{2}-s+v\right)}
%{\Gamma\left(s-v+k-\frac{1}{2}\right)}
%\\ \times \frac{\Gamma\left(s-v-ir_j\right) \Gamma\left(s-v+ir_j\right)}
%{\Gamma\left(\frac{1}{2}-ir_j\right) \Gamma\left(\frac{1}{2}+ir_j\right)}
%\end{multline}
is	
\begin{equation}	
\exp\left(-\frac{\pi}{2}\big[2|\gamma| - 2\max(|\gamma|,|t|)+2\max(|r|, |t|)+ |\gamma-r| -|\gamma +r| 
+ 2\max(|t'+r|, |r_j|) -2|r_j| \big]\right).
\end{equation}
As $|t| = T^\beta$, with $\beta<1$, and $h(u/i)$ decays exponentially when $||\gamma |- T|| >T^\alpha$, 
then, as we will always choose $\sigma \leq \frac{2}{3}$, we may assume $|t | < |\gamma|$.
Then $2|\gamma| - 2\max(|\gamma|,|t|) = 0$. 
Also
 
\begin{equation}
|-r+ \gamma|-|r+\gamma| 
= \begin{cases}
-2\min(|r|, |\gamma|), & \text{ if } r, \gamma\,\text{have the same sign}, \\
2\min(|r|, |\gamma|), & \text{ if } r,\gamma\, \text{have opposite signs}.
\end{cases}
\end{equation}

We assume the worst case, that $r$ and $\gamma$ have the same sign.    
%This implies 
%\begin{equation}
%|-r+\gamma|-|r+\gamma|= -2\min(|r|, |\gamma|),
%\end{equation}
Then the exponential part becomes
\begin{equation}
\exp\left(-\frac{\pi}{2}\left(2\max(|r|, |t|)  -2\min(|r|, |\gamma|)+ 2\max(|t'+r|, |r_j|) -2|r_j|\right)\right).
\end{equation}
If $|r| >|\gamma|$,  then, as  $|t| < |\gamma|$, and $2\max(|t'+r|, |r_j|) -2|r_j| \geq 0$,  
there is exponential decay in $|r|$, while if $|r| \leq |\gamma|$, 
then the exponential part becomes
\begin{equation}
\exp\left(-\frac{\pi}{2}\left(2\max(|r|, |t|)  -2|r|+ 2\max(|t'+r|, |r_j|) -2|r_j|\right)\right).
\end{equation}
Therefore we conclude that 
the only situation in which there is no exponential decay is the case when 
$\gamma$, $r$ have the same sign, $|\gamma| \geq |r| \geq |t| = |t'|$
and $|r_j| \geq |t'+r|$. 
By Stirling's formula in bounded vertical strips
the absolute value of the non-exponential part of the integrand of \eqref{e:OD+cuspint_explicit} 
is bounded above by a constant multiple of
\begin{multline}\label{piece1}
|\gamma| |-\gamma +t|^{\sigma_u-\frac{k}{2}+\frac{1}{2}}|\gamma +t|^{-\sigma_u-\frac{k}{2}+\frac{1}{2}}
|r-t|^{\sigma_v-\frac{1}{2}}|r+t|^{\sigma_v-\frac{1}{2}}
|\gamma-r|^{\sigma_u-\sigma_v+\frac{k}{2}-\frac{1}{2}}
%\\ \times
|\gamma+r|^{-\sigma_u-\sigma_v+\frac{k}{2}-\frac{1}{2}}|t' +r|^{2\sigma_v -k}
\\ \times \sum_{|r_j|\ge |r+t'|}\left(|r_j|-|r+t'|\right)^{-\sigma_v}
\left(|r_j|+|r+t'|\right)^{-\sigma_v} \scrL(s, it; \overline{u_j}) \left<u_j, V_{f, g}\right>.
\end{multline}
Consider only the piece involving $\sigma_u$, we have two cases: $|\gamma|-|r| \gg 1$, and $|\gamma|-|r| \ll 1$.   
The second is treated the same as the first, as, by abuse of notation, 
when we write our absolute values, by $|x|$ we really mean $\max(|x|,1)$.   

Recall that the only situation in which there is no exponential decay is the case when $|\gamma| \geq |r| \geq |t| = |t'|$. 
The situations with $\gamma >0$ and $\gamma <0$ are symmetrical, 
so we will, without loss of generality, assume that $\gamma>0$.   
As $\gamma$, $r$ have the same sign,  this means $\gamma \geq r \geq |t| = |t'| \geq 0$, so $r-|t| \ge0$. 
The piece involving $\sigma_u$ can now be written as
\begin{equation}
\left(\frac{|\gamma -r|}{|\gamma +r|} \right)^{\sigma_u}
\left(\frac{|-\gamma +t|}{|\gamma +t|} \right)^{\sigma_u}
= \left(\frac{1-\frac{r}{\gamma}}{1+\frac{r}{\gamma}} \right)^{\gamma\frac{\sigma_u}{\gamma}}
\left(\frac{1-\frac{t}{\gamma}}{1+\frac{t}{\gamma}} \right)^{\gamma\frac{\sigma_u}{\gamma}}
\ll e^{-2(r+t)\frac{\sigma_u}{\gamma}}.
\end{equation}
Now move the $u$ line of integration to $\sigma_u=CT^\alpha$, for some $C\gg1$, 
which is the largest we can make it while still keeping $|h(u/i)| \ll 1$.  
No poles are passed over in this process.  

\begin{remark}\label{remark:stirling}
This is actually a little tricky.   
Stirling's formula:
\begin{equation}
|\Gamma(\sigma_u + i\gamma)| \sim\sqrt{2 \pi}|\gamma|^{\sigma_u - \frac12}e^{-\frac{\pi}{2}|\gamma|} 
\quad \text{as} \quad |\gamma|\rightarrow \infty
\end{equation}
is usually stated as being true as long as $A < \sigma_u <B$ for fixed $A,B$.   
However, a close examination of the derivation of this formula shows that this asymptotic remains true 
as long as $|A|,|B| \ll |\gamma|^{\frac 23}$.   
As in this context we are moving $\sigma_u$ to $CT^\alpha$, with $\alpha <\frac{2}{3}$, 
and the integral decays quadratically exponentially when $||\gamma| - T| \gg T^\alpha$, the asymptotic remains valid.
\end{remark}

As $\left||\gamma|-T \right|\ll T^\alpha$, (the integral has exponential decay if this is not the case), 
and $r \geq |t|$, it follows that there is arbitrarily strong polynomial decay unless 
\begin{equation}
r +t\ll T^{1-\alpha}.
\end{equation}
If $r$ and $t$ are the same sign (in this case, both positive, as $r$ is positive), 
then this implies that $r, t \ll T^{1-\alpha}$, and in particular that, as $|t| = T^\beta$, $\beta \leq 1-\alpha$.  
If $r$ and $t$ have opposite signs, then $r-T^\beta \ll T^{1-\alpha}$.  
If $\beta >1-\alpha$ this implies that $r \ll T^\beta$, while if $\beta \leq 1-\alpha$ then this implies that $r \ll T^{1-\alpha}$.  
In both cases, as $\beta <1$
\begin{equation}
T \ll |\gamma \pm r| \ll T\quad \text{and} \quad T \ll |\gamma \pm t| \ll T.
\end{equation}
Substituting into \eqref{piece1}, we find that $\OD_{\cusp, \intg}^+(s, t)$
is bounded above by a constant times the integral over $\left||\gamma|-T \right|\ll T^\alpha$ and $T^\beta \le r \ll T^\beta + T^{1-\alpha}$ of
\begin{multline}\label{piece2}
T^{1-2\sigma_v} \left| r-t\right|^{\sigma_v -\frac{1}{2}}
\left| r+t\right|^{\sigma_v -\frac{1}{2}}
\left| t'+r \right|^{2\sigma_v -k}\\ \times
\sum_{|r_j|\ge |r+t'|}\left(|r_j|-|r+t'|\right)^{-\sigma_v}
\left(|r_j|+|r+t'|\right)^{-\sigma_v}
(-1)^{\epsilon_j} \scrL(s, it; \overline{u_j})\left<u_j, V_{f, g}\right>.
\end{multline}

We now require the following lemma to estimate the sum over the $r_j$.

\begin{lemma}\label{lem:sieve}
For fixed $C>0$, and $|t|=|t'| = T^\beta$, 
\begin{equation}
\sum_{|r_j|< T^C} (-1)^{\epsilon_j} \scrL(1/2-it', it; \overline{u_j})\left<u_j, V_{f, g}\right>
\ll T^{Ck+\frac{\max\{C, \beta\}+C}{2}+\epsilon}. 
\end{equation}

\end{lemma}

\begin{proof}
By \cite[Proposition~4.1]{HHR}
\begin{equation}\label{tjsum}
\left(\sum_{|r_j|<T^C}e^{\pi|r_j|}\left| \left<u_j, V_{f, g}\right>\right|^2\right)^{\frac{1}{2}} \ll T^{Ck +\epsilon}.
\end{equation}
Applying Cauchy-Schwartz and substituting \eqref{tjsum}, 
\begin{equation}
\sum_{|r_j|< T^C} (-1)^{\epsilon_j} \scrL(s, it; \overline{u_j})\left<u_j, V_{f, g}\right>
\ll T^{Ck+\epsilon} \bigg(\sum_{|r_j|< T^C} e^{-\pi |r_j|} |\scrL(s, it; \overline{u_j})|^2\bigg)^{\frac{1}{2}}.
\end{equation}
Recalling \eqref{e:scrL_fac_cusp}, taking $s=\frac{1}{2}-it'$ with $|t|=|t'|$, we get
\begin{equation}
\scrL(1/2\pm t, it; \overline{u_j})
= L(1/2\pm 2it, \overline{u_j}) L(1/2, \overline{u_j}) 
N^{-\frac{1}{2}\pm it-it}\sum_{d\mid\frac{N}{L}} \overline{c_{L}(r_j; d) \rho_{L, j}(1)} d^{\frac{1}{2}-it} 
\scrP_d(1/2\pm it, it; \overline{u_j}).
\end{equation}
By estimating $\scrP_d(1/2\pm it, it; \overline{u_j}) \ll 1$, and applying Cauchy-Schwartz again, 
\begin{equation}
\left|\scrL(1/2\pm it, it; u_j)\right|^2
\ll \left|\rho_{L, j}(1)\right|^2 \left|L(1/2\pm 2it, u_j) L(1/2, u_j)\right|^2
N^{-1} \sum_{d\mid\frac{N}{L}} |c_{L}(r_j; d)|^2d
\end{equation}
and 
\begin{multline}
\sum_{|r_j|< T^C} (-1)^{\epsilon_j} \scrL(1/2\pm it, it; \overline{u_j})\left<u_j, V_{f, g}\right>
\\ \ll T^{Ck+\epsilon} \bigg(\sum_{|r_j|< T^C} e^{-\pi |r_j|} \left|\rho_{L, j}(1)\right|^2 
\sum_{d\mid\frac{N}{L}} |c_{L}(r_j; d)|^2
\left|L(1/2\pm 2it, u_j)\right|^2 \left|L(1/2, u_j)\right|^2 \bigg)^{\frac{1}{2}}. 
\end{multline}
By \cite[Appendix]{H-L} 
\begin{equation}
\rho_{L, j}(1)\sum_{d\mid \frac{N}{L}} c_L(r_j; d) \ll e^{\frac{\pi}{2}|r_j|} |r_j|^{\epsilon}.
\end{equation}
So we get 
\begin{equation}
\sum_{|r_j|< T^C} (-1)^{\epsilon_j} \scrL(1/2\pm it, it; \overline{u_j})\left<u_j, V_{f, g}\right>
\ll T^{Ck+\epsilon} \bigg(\sum_{|r_j|< T^C} \left|L(1/2\pm 2it, u_j)\right|^2 \left|L(1/2, u_j)\right|^2\bigg)^{\frac{1}{2}}. 
\end{equation}
Applying Cauchy-Schwartz again, this is 
\begin{equation}\label{e:reapply-CS}
\ll_N T^{Ck+\epsilon} \bigg(\sum_{|r_j| < T^C} \left|L(1/2\pm 2it, u_j)\right|^4\bigg)^{\frac{1}{4}} 
\bigg(\sum_{|r_j| < T^C} \left|L(1/2, u_j)\right|^4\bigg)^{\frac{1}{4}}. 
\end{equation}

We will use the spectral large sieve inequality 
(see \cite[(3.28)]{ky17} and \cite[Theorem~7.24]{IK04}; 
note that there is a typo in the statement of \cite[Theorem~7.24]{IK04}) 
which gives us, translating into our notation,
\begin{equation}\label{sieve}
\sum_{|r_j| \leq T^C} \Big| \sum_{m \leq M} a(m) \rho_j (m)e^{-\frac{\pi}{2}|r_j|} \Big|^2 \ll_\epsilon (T^{2C} N +M)(MNT)^\epsilon \left( \sum_{m \leq M} |a(m)|^2 \right)
\end{equation}
where $a(m)$ are arbitrary complex numbers, $u_j$ form an orthonormal basis of Maass forms of level $N$ 
and $\rho_j$ are the $L^2$-normalized Fourier coefficients of $u_j$.   

To apply this to our $L$-series, we note first that as $|r_j| \ll T^C$, the analytic conductor of $L(1/2\pm 2it, u_j)^2$ 
is on the order of $T^{4\max\{C, \beta\}} N^2$, 
and therefore, using the approximate functional equation, $L(1/2\pm 2it, u_j)^2$ can be approximated by a partial sum of length $M=T^{2\max\{C, \beta\}} N$, the square root of the analytic conductor. 
Similarly $L(1/2, u_j)^2$ can be approximated by a partial sum of length $T^{2C} N$. 
To write $L(1/2+ait, u_j)^2$, with $a=0$ or $\pm 2$ in a suitable form to apply the sieve, 
note that by the Hecke relations for $\lambda_j(m)$'s, for $\Re(w)>1$,  
\begin{equation}
L(w, u_j)^2 
= \sum_{m_1,  m_2\geq 1} \frac{\lambda_j(m_1) \lambda_j(m_2)}{(m_1m_2)^w} 
= \sum_{m\geq 1} \frac{\sigma_0(m)\sum_{d^2\mid m} \lambda_j(m/d^2)}{m^w}
= \sum_{d, m\geq 1} \frac{\sigma_0(md^2) \lambda_j(m)}{(d^2m)^w}.
\end{equation} 
Therefore, by the approximate functional equation, $L(1/2+ait, u_j)^2$ can be approximated by a Dirichlet series of 
length 
\begin{equation}
M(a) = \begin{cases} T^{2\max\{C, \beta\}}N & \text{ when }a=\pm 2, \\ 
T^{2C} N & \text{ when } a=0,  \end{cases}
\end{equation}
with numerator $\sigma_0(md^2) \lambda_j(m)$ and denominator $(d^2m)^{1/2+ait}$. 
Applying \eqref{sieve}, 
\begin{multline}
\sum_{|r_j|\leq T^C} 
\left|\sum_{m\leq M(a)} \bigg(\sum_{1\leq d\leq \sqrt{\frac{M(a)}{m}}} \frac{\sigma_0(md^2)}{(d^2m)^{\frac{1}{2}+ait}}\bigg) \lambda_j(m)\right|^2
\\ \ll_\epsilon (T^{2C}N+ M(a)) (M(a) NT)^{\epsilon} 
\bigg(\sum_{m\leq M(a)} \bigg|\sum_{1\leq d\leq \sqrt{\frac{M(a)}{m}}} \frac{\sigma_0(md^2)}{(d^2m)^{\frac{1}{2}+ait}}\bigg|^2\bigg)
\\ \ll_{\epsilon} (T^{2C}N+ M(a)) (M(a) NT)^{\epsilon}. 
\end{multline}
Here we write $\lambda_j(m)$ instead of $\rho_j(m)e^{-\frac{\pi}{2}|r_j|}$, because, again by \cite{H-L},
\begin{equation}
e^{\frac{\pi}{2}|r_j|}L^{-\epsilon} \ll_{\epsilon} \rho_{L,j}(1) \ll_{\epsilon}  e^{\frac{\pi}{2}|r_j|} L^\epsilon.
\end{equation}

Substituting the above into \eqref{e:reapply-CS} finally gives us the upper bound we need: 
\begin{equation}
\sum_{|r_j|< T^C} (-1)^{\epsilon_j} \scrL(1/2\pm it, it; \overline{u_j})\left<u_j, V_{f, g}\right>
\ll_N T^{Ck+\epsilon} T^{\frac{\max\{C, \beta\}+C}{2}}. 
\end{equation}
\end{proof}

Referring to \eqref{piece2} it is easily checked that for $\sigma_v = \frac{5}{4}+\frac{k}{2}$
 the portion of the sum with  $|r_j| >2|t'+r|$ converges absolutely 
and is dominated by the portion with $|r_j| \le 2|t'+r|$, where the bounds for the first ratio of gamma factors is weakest.  
Applying this to \eqref{piece2}, we use the bounds  $\left(|r_j|-|r+t'|\right)^{-\sigma_v}\ll1$ and 
$\left(|r_j|+|r+t'|\right)^{-\sigma_v}\ll |r+t'|^{-\sigma_v}$.
We are then finally reduced to finding an upper bound for the integral over 
$\left||\gamma|-T \right|\ll T^\alpha$ and $T^\beta \le |r| \le T^\beta + T^{1-\alpha} \log T$ of
\begin{equation}\label{orig}
T^{1-2\sigma_v} \left| r-t\right|^{\sigma_v -\frac{1}{2}}
\left| r+t\right|^{\sigma_v -\frac{1}{2}}
\left| t'+r \right|^{\sigma_v -k}T^{\frac{\max\{C,\beta\}}{2} +\frac{C}{2} + Ck + \epsilon}
\end{equation}
Here, in Lemma~\ref{lem:sieve},  $T^C = |r+t'|$.
An upper bound can be found easily by simply calculating the coefficient of $\sigma_v$.   
Interestingly, the calculation is quite different in the cases $t = t'$ and $t = -t'$.   

First, suppose $t = t'$.  
Then \eqref{orig} becomes
\begin{equation}
T^{1-2\sigma_v} \left| r-t\right|^{\sigma_v -\frac{1}{2}} \left| r+t\right|^{2\sigma_v -k-\frac{1}{2}}
T^{\frac{\max\{C,\beta\}}{2} +\frac{C}{2} + Ck + \epsilon}
\end{equation}
We then have two sub cases, $t$, $r$ have the same sign, and $t$, $r$ have opposite signs.  
As $|t + r| \ll T^{1-\alpha}$, if $t$, $r$ have the same sign then both are positive 
(recall $r$ has the same sign as $\gamma$, which is positive), 
which implies that $t,r \ll T^{1-\alpha}$, $\beta \le 1- \alpha$ and $C = 1-\alpha$ in Lemma~\ref{lem:sieve}. 
 
As it follows that
$\left| r-t\right|\ll T^{1-\alpha}$, the above is 
\begin{equation}\label{same}
\ll T^{1-2\sigma_v +(1-\alpha)(3\sigma_v -1-k) + (1-\alpha) +(1-\alpha)k + \epsilon}=T^{1-\sigma_v(3\alpha -1) + \epsilon}
\end{equation}
the coefficient of $\sigma_v$ in the exponent above is $1-3\alpha$, which is not positive when $\alpha \ge 1/3$.

If $t$, $r$ have opposite signs, then $\left| r+t\right|= \left| r- t^\beta\right|\ll T^{1-\alpha}$.   
In this case \eqref{same} is replaced by
\begin{equation}\label{opp}
T^{1-2\sigma_v+(1-\alpha)(2\sigma_v -\frac{1}{2}-k)+\beta(\sigma_v-\frac{1}{2})+\frac{\max\{1-\alpha,\beta\}}{2} 
+ \frac{1-\alpha}{2}+ (1-\alpha)k + \epsilon} 
= T^{\sigma_v(\beta-2\alpha) + 1-\frac{\beta}{2} +\frac{\max\{1-\beta, \alpha\}}{2}+1+\epsilon}.
\end{equation}
Here the coefficient of $\sigma_v$ is $\beta-2\sigma $, which is not positive when $\sigma \ge \beta/2$.
If $\beta >1-\sigma$ then the exponent in \eqref{opp} is less than or equal to $1+\epsilon$.  If $\beta \le 1-\alpha$, then, as $\sigma_v >1/2$ , the exponent in \eqref{opp} equals $ \sigma_v(1-3\alpha) + 1 + \epsilon$ which is less  $1 +\epsilon$ as $\alpha >1/3$.

Now suppose $t=-t'$.  
If $t$, $r$ have the same sign then $t'$, $r$ have opposite signs and 
\begin{equation}\label{oppsame}
\left| r+t'\right| = \left| r-t\right|  \ll T^{1-\alpha},
\end{equation}
so $C=1-\alpha$ in Lemma~\ref{lem:sieve} and \eqref{orig} becomes
\begin{equation}\label{same2}
T^{1-2\sigma_v +(1-\alpha)(2\sigma_v -1/2-k) +\beta(\sigma_v-1/2)  +\max(1-\alpha,\beta)/2 + (1-\alpha)/2+ (1-\alpha)k + \epsilon} =T^{\sigma_v (\beta-2\alpha)+1-\frac\beta 2 +\frac{\max(1-\alpha,\beta)}{2}   + \epsilon}
\end{equation}
Here, as before, the coefficient of $\sigma_v$ is $\beta-2\alpha $, which is negative when $\alpha > \frac{\beta}{2}$. 
If $\beta > 1-\alpha$ then the exponent in \eqref{same2} is $\sigma_v (\beta-2\alpha)+1+ \epsilon<1+\epsilon$. 
%which is less than $1+\epsilon$.   
If $\beta \le 1-\alpha$ then the exponent in \eqref{same2} is $\sigma_v(1-3\alpha) +1 +\epsilon<1+\epsilon$ as well. 
%which is less than $1 +\epsilon$.

If $t$, $r$ have the opposite signs then $t'$, $r$ have the same sign, 
the constant $C$ in Lemma~\ref{lem:sieve} is replaced by $\beta$ 
and instead of \eqref{same2} we have, after simplifying a bit,
\begin{equation}\label{opp2}
T^{1-2\sigma_v +\beta\sigma_v-\frac \beta 2 +(1-\alpha)\sigma_v +\sigma_v \beta +\frac{\max(1-\alpha,\beta)}{2}}=T^{\sigma_v(2\beta -\alpha -1)+1-\frac \beta 2+\frac{\max(1-\alpha,\beta)}{2}+\epsilon}.
\end{equation} 
Here the coefficient of $\sigma_v$ is $2\beta -1-\alpha$, which is again negative when $\alpha > 2\beta-1$.
If $\beta > 1-\alpha$ the exponent in \eqref{opp2} becomes $\sigma_v(2\beta -\alpha -1)+1+\epsilon< 1+\epsilon$, 
%which is less than $1 +\epsilon$, 
and if $\beta \le 1-\alpha$ as $\sigma_v > \frac{1}{2}$, 
the exponent becomes $\sigma_v(1-3\alpha) + 1 +\epsilon< 1+\epsilon$.
%, which is also less than $1 +\epsilon$.

Similarly, we get the upper bound for $\OD^+_{\cont, \intg}(s, t)$. 
The proof is identical, but instead of using  the spectral part of \eqref{e:innerbound} we use
\begin{equation}
\sum_\cuspa \int_{-T}^T |\left<E_{\cuspa}(*, 1/2 + it), V_{f, g}\right>|^2e^{\pi |t|}dt\ll_{N,k}T^{2k}\log(T).
\end{equation}
Also, instead of using Lemma~\ref{lem:sieve} 
we use the mean value theorem for integrals of Dirichlet plynomials, \cite[Theorem~9.1]{IK04}.
This completes the proof of Proposition~\ref{prop:ub_OD+ccint}.
\qed

\subsection{Upper bounds for $\OD_{\cusp, \res}^+(s, t)$ and $\OD_{\cont, \res}^+(s, t)$}\label{sect:4.5}

Recall that we have set $s = \frac{1}{2} -it'$, with $|t'| = |t|$ and $|t| = T^\beta$, with $0\leq \beta<1$.
The objective of this section is to prove the following 
\begin{proposition}\label{prop:OD+res_upper}
Fix $\alpha, \beta$, with $0\le \beta<1$, and $\frac{1}{3}< \alpha <\frac23$. 
If $2\alpha > \beta$  then for $\beta \le 1-\alpha$
\begin{equation}
\OD^+_{\cusp, \res}(s, t), \; \OD^+_{\cont, \res}(s, t)\ll_{N} T^{1+\alpha - \frac{3\alpha -1}{2}+\epsilon}.
\end{equation}
For $2\alpha > \beta >1-\alpha$
\begin{equation}
\OD^+_{\cusp, \res}(s, t), \; \OD^+_{\cont, \res}(s, t)\ll_{N} T^{1+\alpha - (2\alpha-\beta)\frac{k+1}{2}+\epsilon}.
\end{equation}

\end{proposition}
We consider $\OD_{\cusp, \res}^+(s, t)$ in \eqref{e:OD+cuspres}. 
As before, write $u=\sigma_u+i\gamma$. %and $s=1/2-it'$. 
By Stirling's formula, the exponential growth of the product of all the gamma functions in the integrand of \eqref{e:OD+cuspres}, 
is given by 
\begin{equation}
\exp\left(-\frac{\pi}{2}\left(2|\gamma|-2\max\{|\gamma|, |t|\} + 2\max\{|-t'+r_j|, |t|\}+|t'-r_j+\gamma|-|-t'+r_j+\gamma|\right)\right). 
\end{equation}
As $|t|=T^\beta$ with $0 \leq \beta< 1$ and the worst case happens when $-t'+r_j$ and $\gamma$ have the same sign, 
in that case, the exponent becomes 
\begin{equation}
\exp\left(-\frac{\pi}{2}\left(2\max\{|-t'+r_j|, |t|\}-2\min\{|\gamma|, |-t'+r_j|\}\right)\right). 
\end{equation}
If $|-t'+r_j| < |t|$ then as $|t|< |\gamma|$ this becomes $2(|t|-|-t'+r_j|)$ and there is exponential decay.
If $|-t'+r_j|\geq |t|$ then as $|t'|=|t|$, $-t'$ and $r_j$ must have the same signs and the argument becomes 
$2(|t|+|r_j|-\min\{|t|+|r_j|, |\gamma|\})$. 
This is zero when $|t|+|r_j| \leq |\gamma|$. 

Thus there is no exponential decay precisely when $-t'$ and $r_j$ have the same sign, 
$-t'+r_j$ and $\gamma$ have the same sign, and $|t|+|r_j|\leq |\gamma|$. 
Also, by \eqref{e:hdef}, there is exponential decay in the integral when $||\gamma|-T| > T^\alpha$. 

Using the polynomial piece of Stirling's estimate, we find that the integrand in \eqref{e:OD+cuspres} is bounded above
by a constant times the sum over $r_j$ restricted to the set where $-t'$ and $r_j$ have the same sign, $-t'+r_j$ 
and $\gamma$ have the same sign, and $|t|+|r_j| \leq |\gamma|$ of 
\begin{multline}\label{uupper}
|\gamma||-\gamma+t|^{\sigma_u-\frac{k}{2}+\frac{1}{2}}|\gamma+t|^{-\sigma_u-\frac{k}{2}+\frac{1}{2}} 
\\ \times 
\sum_j|-t'+r_j -t|^\ell|-t'+r_j +t|^\ell
|t'-r_j +\gamma|^{\sigma_u+\frac{k}{2}-\ell -1}|-t'+r_j +\gamma|^{-\sigma_u+\frac{k}{2}-\ell -1} |r_j|^{2\ell +1-k}|r_j|^{-\ell -\frac{1}{2}}
\\\times (-1)^{\epsilon_j} \scrL(s, it; \overline{u_j}) \left<u_j, V_{f, g}\right>.
\end{multline}
Isolating the factors raised to the $\sigma_u$, we have 
\begin{equation}
\left(\frac{|-\gamma +t|}{|\gamma+t|}\right)^{\sigma_u}\left(\frac{|t'-r_j+\gamma |}{|t'-r_j-\gamma|}\right)^{\sigma_u}
\end{equation}
Recall $|t|=|t'|=T^\beta$ with $0\leq \beta< 1$, $-t'$ and $r_j$ have the same sign, and $-t'+r_j$, $\gamma$ have the same sign. 
Also, the worst case is when $\gamma$ and $t$ have opposite signs. 
Bearing these in mind, and using the fact that in the region without exponential decay $|\gamma|=T$, 
the above becomes 
\begin{equation}
\left(\frac{|T +T^\beta|}{|T-T^\beta|}\right)^{\sigma_u}\left(\frac{|T- T^\beta-|r_j|}{|T+ T^\beta+|r_j|}\right)^{\sigma_u}
= \left(\frac{1-\frac{T^\beta+|r_j|}{T}}{1+\frac{T^\beta+|r_j|}{T}}\right)^{T\frac{\sigma_u}{T}}
\left(\frac{1+\frac{T^\beta}{T}}{1-\frac{T^\beta|}{T}}\right)^{T\frac{\sigma_u}{T}}
\ll e^{-2|r_j|\frac{\sigma_u}{T}}.
\end{equation}
Now after moving $\sigma_u$ to $KT^\alpha$ for $K\gg 1$, (bearing in mind Remark~\ref{remark:stirling}, which justifies this), the right hand side becomes 
\begin{equation}
\ll e^{-\frac{2|r_j| K}{T^{1-\alpha}}}
\end{equation}
and so there is exponential decay unless $|r_j| \ll T^{1-\alpha}$. 

To estimate the sum over $r_j$, first note that, by \eqref{e:scrL_fac_cusp}, for $|t|=|t'|$, 
%taking $s=1/2-it'$, 
\begin{equation}
\scrL(1/2\pm it, it; \overline{u_j})
= L(1/2\pm 2it, \overline{u_j}) L(1/2, \overline{u_j}) 
N^{-\frac{1}{2}\mp it-it}\sum_{d\mid\frac{N}{L}} \overline{c_{L}(r_j; d) \rho_{L, j}(1)} d^{\frac{1}{2}-it} 
\scrP_d(1/2\pm it, it; \overline{u_j}).
\end{equation}
Then, as $0\leq \beta<1$, the $|\gamma|$ dominates in most of the terms, and the expression in \eqref{uupper} is bounded above 
by a constant times 
\begin{multline}\label{uupper2}
T^{-2\ell}\sum_j \left|2|t|\pm |r_j|\right|^{\ell} |r_j|^{2\ell+\frac{1}{2}-k} 
\\ \times \rho_{L, j}(-1) L(1/2+2it, u_j) L(1/2, u_j)
N^{-\frac{1}{2}\mp it-it} \sum_{d\mid \frac{N}{L}} \overline{c_{L}(r_j; d)} d^{\frac{1}{2}-it} \scrP_d(1/2\pm it, it; \overline{u_j})
\left<u_j, V_{f, g}\right>. 
\end{multline}
Now recall Lemma~\ref{lem:sieve} and set $C= 1-\alpha$.   
This gives us
\begin{equation}
\sum_{|r_j|< T^{1-\alpha}} (-1)^{\epsilon_j} \scrL(1/2-it', it; \overline{u_j}) \left<u_j, V_{f, g}\right>
\ll T^{\frac{1}{2}\max\{1-\alpha, \beta\}+\frac{1-\alpha}{2}+(1-\alpha)k+\epsilon}. 
\end{equation}
Combining with \eqref{uupper2} and performing the $u$ integration, 
which adds a $\alpha$ to the exponent, we find that 
\begin{multline}
\OD_{\cusp, \res}^+(s, t) 
\ll T^{-2\ell + \ell \max(\beta,1-\alpha)+(2\ell +1/2-k)(1-\alpha) +\frac{\max(\beta,1-\alpha)}{2}+\frac{1-\alpha}{2} + (1-\alpha)k+\alpha}
\\ = T^{-2\ell + \ell \max(\beta,1-\alpha)+(2\ell +1/2)(1-\alpha) +\frac{\max(\beta,1-\alpha)}{2}+\frac{1-\alpha}{2} +\alpha}.
\end{multline}

If $\beta < 1-\alpha$, the exponent becomes $\ell(1-3\alpha) +\frac32 -\frac{\alpha}{2}$. 
As $\ell$ becomes large, we require its coefficient in the power, i.e., $1-3\alpha$, to be zero or negative. 
As $\alpha > \frac{1}{3}$ the exponent is less than $1 +\frac{1-\alpha}{2}+ \epsilon$.  
As the main term will be on the order of $T^{\alpha+1 +\epsilon}$ it is convenient to write the exponent 
as $1 +\frac{1-\alpha}{2}+ \epsilon = 1+\alpha - \frac{3\alpha -1}{2}+ \epsilon$.

If $\beta\geq 1-\alpha$ the $\ell$ coefficient is $\beta-2\alpha$, 
and as we have the restriction $\alpha > \frac{\beta}{2}$ this is negative.  
Consequently the exponent is less than $1+\frac{\beta}{2} + \epsilon$.  
Again, as the main term will be on the order of $T^{\alpha+1 +\epsilon}$ it is convenient to write the exponent 
as $1+\alpha - (2\alpha -\beta)\frac{k+1}{2}$.
 
The continuous contribution $\OD_{\cont, \res}^+(s, t)$ is treated by again 
using Theorem 9.1 of \cite{IK04} instead of  Lemma~\ref{lem:sieve}.
This completes the proof of the proposition.
\qed

\subsection{Analysis of $\OD^+_{\Omega}(s, t)$}\label{sect:4.6}
Recalling \eqref{e:OD+Omega}, \eqref{e:ZOmega} and \eqref{e:M_Gamma}, 
for $\Re(s)=\frac{1}{2}$, moving the $v$ line of integration to $\Re(v)=1+\frac{k}{2}+\epsilon$, 
we pass over the poles of gamma functions and get 
\begin{equation}
\OD^+_{\Omega}(s, t)
= \OD^+_{\Omega, \intg}(s, t)+ \sum_{\ell=0}^{\frac{k}{2}} \OD^+_{\Omega, \res, 1}(s, t; \ell) 
+ \sum_{\ell=0}^{\frac{k}{2}} \OD^+_{\Omega, \res, 2}(s, t; \ell)
\end{equation}
where 
\begin{multline}\label{e:OD+Omegaint}
\OD^+_{\Omega, \intg}(s, t) 
=\frac{i^k (4\pi)^k (2\pi)^{2it}}{2}\zeta(2s-1)
\frac{4}{2\pi i}\int_{(\sigma_u)} \frac{h(u/i)u}{\cos(\pi u)}
\frac{1}{\Gamma\left(-u+it+\frac{k}{2}\right)\Gamma\left(u+it+\frac{k}{2}\right)}
\\ \times \frac{1}{2\pi i} \int_{(1+\frac{k}{2}+\epsilon)}
\frac{\Gamma\left(v-it\right) \Gamma\left(v+it\right) \Gamma\left(-v+u+\frac{k}{2}\right)}
{\Gamma\left(v+u+1-\frac{k}{2}\right)}
\frac{\Gamma\left(\frac{1}{2}-s+v\right)}{\Gamma\left(s-v+k-\frac{1}{2}\right)}
\\ \times 
% z=1-s\pm it
\sum_{\pm} \frac{\pi^{s-\frac{1}{2}\mp it} \zeta(1\pm 2it)}{\Gamma\left(-\frac{1}{2}+s\mp it\right)} 
\frac{\Gamma\left(2s-1\mp it-v\right) \Gamma\left(1\pm it-v\right)}
{\Gamma\left(-\frac{1}{2}+s\mp it\right)\Gamma\left(\frac{3}{2}-s\pm it\right)}
\\ \times \sum_{\cuspa} \frac{\scrP_{\cuspa}(s, it; 1-s\pm it)}{\prod_{p\mid N}(1-p^{1-2s\pm 2it})} 
\left<E_{\cuspa}(*, 3/2-s\pm it), V_{f, g}\right>
\; dv \; du, 
\end{multline}
\begin{multline}
\OD^+_{\Omega, \res, 1}(s, t; \ell) 
=\frac{(-1)^{\ell}}{\ell!} \frac{i^k (4\pi)^k (2\pi)^{2it}}{2}\zeta(2s-1)\Gamma\left(2s-1+\ell\right)
%v=2s-1\mp it+\ell
\\ \times \sum_{\pm} 
\frac{4}{2\pi i}\int_{(\sigma_u)} \frac{h(u/i)u}{\cos(\pi u)}
\frac{1}{\Gamma\left(-u+it+\frac{k}{2}\right)\Gamma\left(u+it+\frac{k}{2}\right)}
\frac{\Gamma\left(1-2s\pm it+u+\frac{k}{2}-\ell\right)}{\Gamma\left(2s\mp it+u-\frac{k}{2}+\ell\right)} \; du
\\ \times 
\frac{\zeta(1\pm 2it)}{\pi^{\frac{1}{2}-s\pm it} \Gamma\left(-\frac{1}{2}+s\mp it\right)} 
\frac{\Gamma\left(-1+2s\mp 2it+\ell\right) \Gamma\left(2-2s\pm 2it-\ell\right)}
{\Gamma\left(-\frac{1}{2}+s\mp it\right)\Gamma\left(\frac{3}{2}-s\pm it\right)}
\frac{\Gamma\left(-\frac{1}{2}+s\mp it+\ell\right)}{\Gamma\left(\frac{1}{2}-s\pm it+k-\ell\right)}
\\ \times \sum_{\cuspa}
\frac{\scrP_{\cuspa}(s, it; 1-s\pm it)}{\prod_{p\mid N}(1-p^{1-2s\pm 2it})}
\left<E_{\cuspa}(*, 3/2-s\pm it), V_{f, g}\right>
\end{multline}
and 
\begin{multline}
\OD^+_{\Omega, \res, 2}(s, t; \ell) 
%= (-1)^\ell\frac{i^k (4\pi)^k (2\pi)^{2it}}{2}\zeta(2s-1)
%\frac{4}{2\pi i}\int_{(\sigma_u)} \frac{h(u/i)u}{\cos(\pi u)}
%\frac{1}{\Gamma\left(-u+it+\frac{k}{2}\right)\Gamma\left(u+it+\frac{k}{2}\right)}
%\\ \times \bigg\{
%% v=1+it+\ell
%\frac{\Gamma\left(u-it+\frac{k}{2}-\ell-1\right)}{\Gamma\left(1+u+it+1-\frac{k}{2}+\ell\right)}
%\frac{\Gamma\left(1+2it+\ell\right) \Gamma\left(\frac{3}{2}-s+it+\ell\right)}{\Gamma\left(s-it+k-\ell-\frac{3}{2}\right)}
%\frac{\pi^{s-\frac{1}{2}-it} \zeta(1+2it)}{\Gamma\left(-\frac{1}{2}+s-it\right)} 
%\\ \times \frac{\Gamma\left(2s-2-2it-\ell\right)}
%{\Gamma\left(-\frac{1}{2}+s-it\right)\Gamma\left(\frac{3}{2}-s+it\right)}
%\sum_{\cuspa} \frac{\scrP_{\cuspa}(s, it; 1-s+it)}{\prod_{p\mid N}(1-p^{1-2s+2it})} \left<E_{\cuspa}(*, 3/2-s+it), V_{f, g}\right>
%\\ + 
%% v= 1-it+\ell
%\frac{\Gamma\left(u+it-1+\frac{k}{2}-\ell\right)}
%{\Gamma\left(u-it+2-\frac{k}{2}-\ell\right)}
%\frac{\Gamma\left(1-2it+\ell\right) \Gamma\left(\frac{3}{2}-s-it+\ell\right)}{\Gamma\left(-\frac{3}{2}+s+it+k-\ell\right)}
%\\ \times \frac{\pi^{s-\frac{1}{2}+it} \zeta(1-2it)}{\Gamma\left(-\frac{1}{2}+s+it\right)} 
%\frac{ \Gamma\left(2s-2+2it-\ell\right)}
%{\Gamma\left(-\frac{1}{2}+s+it\right)\Gamma\left(\frac{3}{2}-s-it\right)}
%\sum_{\cuspa} \frac{\scrP_{\cuspa}(s, it; 1-s-it)}{\prod_{p\mid N}(1-p^{1-2s-2it})} \left<E_{\cuspa}(*, 3/2-s-it), V_{f, g}\right>
%\bigg\}
%%
%\; du. 
= (-1)^\ell\frac{i^k (4\pi)^k (2\pi)^{2it}}{2}\zeta(2s-1)
\\ \times \sum_{\pm}
% v=1\pm it+\ell
\frac{4}{2\pi i}\int_{(\sigma_u)} \frac{h(u/i)u}{\cos(\pi u)}
\frac{1}{\Gamma\left(-u+it+\frac{k}{2}\right)\Gamma\left(u+it+\frac{k}{2}\right)}
\frac{\Gamma\left(u\mp it+\frac{k}{2}-\ell-1\right)}{\Gamma\left(1+u\pm it+1-\frac{k}{2}+\ell\right)}
\; du
\\ \times \frac{\Gamma\left(1\pm 2it+\ell\right) \Gamma\left(\frac{3}{2}-s\pm it+\ell\right)}
{\Gamma\left(s\mp it+k-\ell-\frac{3}{2}\right)}
\frac{\pi^{s-\frac{1}{2}\mp it} \zeta(1\mp 2it)}{\Gamma\left(-\frac{1}{2}+s\mp it\right)} 
\frac{\Gamma\left(2s-2\mp 2it-\ell\right)}
{\Gamma\left(-\frac{1}{2}+s\mp it\right)\Gamma\left(\frac{3}{2}-s\pm it\right)}
\\ \times \sum_{\cuspa} \frac{\scrP_{\cuspa}(s, it; 1-s\pm it)}{\prod_{p\mid N}(1-p^{1-2s\pm 2it})} \left<E_{\cuspa}(*, 3/2-s\pm it), V_{f, g}\right>
\end{multline}
By the reflection formula for gamma functions, we get
\begin{multline}\label{e:OD+Omegares1}
\OD^+_{\Omega, \res, 1}(s, t; \ell) 
%=\frac{(-1)^{\ell}}{\ell!} \frac{(4\pi)^k (2\pi)^{2it}}{2}\zeta(2s-1)\Gamma\left(2s-1+\ell\right)
%%v=2s-1-it+\ell
%\\ \times \bigg\{
%\frac{\pi^{s-\frac{1}{2}-it} \zeta(1+2it)}{2\sin(\pi(s-it))}
%\frac{\Gamma\left(s-\frac{1}{2}-it+\ell\right)}{\Gamma\left(s-\frac{1}{2}-it\right)\Gamma\left(-s+\frac{1}{2}+it+k-\ell\right)}
%\sum_{\cuspa}\frac{\scrP_{\cuspa}(s, it; 1-s+it)}{\prod_{p\mid N}(1-p^{1-2s+2it})}
%\left<E_{\cuspa}(*, 3/2-s+it), V_{f, g}\right>
%\\ \times 
%\frac{4}{2\pi i}\int_{(\sigma_u)} \frac{h(u/i)u}{\cos(\pi u)}
%\frac{\Gamma\left(u-2s+1+it+\frac{k}{2}-\ell\right)\Gamma\left(-u-2s+1+it+\frac{k}{2}-\ell\right)}
%{\Gamma\left(-u+it+\frac{k}{2}\right)\Gamma\left(u+it+\frac{k}{2}\right)}
%\frac{\sin(\pi(u+2s-it))}{\pi}
%\; du
%% v=2s-1+it+\ell
%\\ +
%\frac{\pi^{s-\frac{1}{2}+it} \zeta(1-2it)}{2\sin(\pi(s+it))}
%\frac{\Gamma\left(s-\frac{1}{2}+it+\ell\right)}{\Gamma\left(s-\frac{1}{2}+it\right)\Gamma\left(-s+\frac{1}{2}-it+k-\ell\right)}
%\sum_{\cuspa} \frac{\scrP_{\cuspa}(s, it; 1-s-it)}{\prod_{p\mid N}(1-p^{1-2s-2it})}
%\left<E_{\cuspa}(*, 3/2-s-it), V_{f, g}\right>
%\\ \times 
%\frac{4}{2\pi i}\int_{(\sigma_u)} \frac{h(u/i)u}{\cos(\pi u)}
%\frac{\Gamma\left(u-2s+1-it+\frac{k}{2}-\ell\right)\Gamma\left(-u-2s+1-it+\frac{k}{2}-\ell\right) }
%{\Gamma\left(-u+it+\frac{k}{2}\right)\Gamma\left(u+it+\frac{k}{2}\right)}
%\frac{\sin(\pi(u+2s+it))}{\pi} \; du
%\bigg\}
=\frac{(-1)^{\ell}}{\ell!} \frac{(4\pi)^k (2\pi)^{2it}}{2}\zeta(2s-1)\Gamma\left(2s-1+\ell\right)
%v=2s-1\mp it+\ell
\\ \times \sum_{\pm} 
\frac{4}{2\pi i}\int_{(\sigma_u)} \frac{h(u/i)u}{\cos(\pi u)}
\frac{\Gamma\left(u-2s+1\pm it+\frac{k}{2}-\ell\right)\Gamma\left(-u-2s+1\pm it+\frac{k}{2}-\ell\right)}
{\Gamma\left(-u+it+\frac{k}{2}\right)\Gamma\left(u+it+\frac{k}{2}\right)}
\frac{\sin(\pi(u+2s\mp it))}{\pi}
\; du
\\ \times 
\frac{\pi^{s-\frac{1}{2}\mp it} \zeta(1\pm 2it)}{2\sin(\pi(s\mp it))}
\frac{\Gamma\left(s-\frac{1}{2}\mp it+\ell\right)}{\Gamma\left(s-\frac{1}{2}\mp it\right)\Gamma\left(-s+\frac{1}{2}\pm it+k-\ell\right)}
\\ \times \sum_{\cuspa}\frac{\scrP_{\cuspa}(s, it; 1-s\pm it)}{\prod_{p\mid N}(1-p^{1-2s\pm 2it})}
\left<E_{\cuspa}(*, 3/2-s\pm it), V_{f, g}\right>
\end{multline}
and 
\begin{multline}\label{e:OD+Omegares2}
\OD^+_{\Omega, \res, 2}(s, t; \ell) 
= (-1)^\ell\frac{i^k (4\pi)^k (2\pi)^{2it}}{2}\zeta(2s-1)
\\ \times \sum_{\pm}
% v=1\pm it+\ell
\frac{4}{2\pi i}\int_{(\sigma_u)} \frac{h(u/i)u}{\cos(\pi u)}
\frac{\Gamma\left(u\mp it+\frac{k}{2}-\ell-1\right)\Gamma\left(-u\mp it+\frac{k}{2}-\ell-1\right)}
{\Gamma\left(u\mp it+\frac{k}{2}\right)\Gamma\left(-u\mp it+\frac{k}{2}\right)}
\frac{(-1)^{\frac{k}{2}-\ell} \sin(\pi(u+it))}{\pi}\; du
\\ \times 
\frac{\Gamma\left(1\pm 2it+\ell\right)\Gamma\left(2s-2\mp 2it-\ell\right)}
{\Gamma\left(s-\frac{1}{2}\mp it\right)\Gamma\left(-s+\frac{1}{2}\pm it+1\right)}
\frac{\Gamma\left(-s+\frac{1}{2}\pm it+1+\ell\right)}{\Gamma\left(s\mp it+k-\ell-\frac{3}{2}\right)}
\\ \times \frac{\pi^{s-\frac{1}{2}\mp it} \zeta(1\pm 2it)}{\Gamma\left(s-\frac{1}{2}\mp it\right)} 
\sum_{\cuspa} \frac{\scrP_{\cuspa}(s, it; 1-s\pm it)}{\prod_{p\mid N}(1-p^{1-2s\pm 2it})} 
\left<E_{\cuspa}(*, 3/2-s\pm it), V_{f, g}\right>.
\end{multline}

Note that, in $\OD^+_{\Omega, \intg}(s, t)$ \eqref{e:OD+Omegaint}, taking $s=\frac{1}{2}\pm it$, when $f=g$, 
the pole from $\left<E_{\cuspa}(*, 3/2-s\pm it), V_{f, g}\right>$ is canceled by the zero of 
$\frac{1}{\Gamma\left(s-\frac{1}{2}\mp it\right)}$ at $s=\frac{1}{2}\pm it$. 
This pole cancelation also occurs in $\OD^+_{\Omega, \res, 1}(s, t; \ell)$ for $\ell \geq 1$ and $\OD^+_{\Omega, \res, 2}(s, t; \ell)$ when $\ell \geq 0$. 

We let 
\begin{equation}\label{e:MOmega}
M^+_{\Omega}(s, t) = \OD^+_{\Omega, \res, 1} (s, t; 0). 
\end{equation}
This also implies that, when $f\neq g$, 
\begin{equation}
\OD^+_{\Omega}(1/2\pm it, t)=\OD^+_{\Omega, \res, 1}(1/2\pm it, t; 0). 
\end{equation}

This completes the proof of \eqref{e:E_OD+} in Proposition~\ref{prop:OD+_decomp}.

\subsection{Upper bounds for $\OD^+_{\Omega, \intg}(s, t)$, $\OD^+_{\Omega, \res, 1}(s, t; \ell)$ 
and $\OD^+_{\Omega, \res, 2}(s, t; \ell)$}\label{sect:4.7}
Referring to \eqref{e:OD+Omegaint}, \eqref{e:OD+Omegares1} and \eqref{e:OD+Omegares2}, we apply the same estimation techniques as before and state the corresponding upper bounds in the following proposition.
We omit the details.
\begin{proposition}\label{prop:OD+Omega_upper}
Taking $s=\frac{1}{2}-it'$ and $|t|=|t'|=T^{\beta}$ for $0\leq \beta<1$, 
\begin{equation}
\OD^+_{\Omega, \intg}(s, t) \ll 1,
\end{equation}
for $1\leq \ell\leq \frac{k}{2}$
\begin{equation}
\OD^+_{\Omega, \res, 1} (s, t; \ell)\ll 1
\end{equation}
and for $0\leq \ell\leq \frac{k}{2}$
\begin{equation}
\OD^+_{\Omega, \res, 2}(s, t; \ell) \ll 1.
\end{equation}
\end{proposition}

Note that we haven't included $\OD^+_{\Omega, \res, 1}(s, t; 0)$. 
This is another part of the main term, which we will discuss in the next section. 

\subsection{Analysis of $M^+_{\Omega}(s, t)$}\label{sect:main2}

Recalling \eqref{e:OD+Omegares1} and \eqref{e:MOmega}, 
%\begin{equation}\label{e:MOmega}
%M^+_{\Omega}(s, t) = \OD^+_{\Omega, \res, 1} (s, t; 0). 
%\end{equation}
for further clarity we consider $M^+_{\Omega}(s, t)$ on $\Re(s)=\frac{1}{2}$.

\begin{lemma}\label{lem:M+Omega}
For $\Re(s)=\frac{1}{2}$, we get
\begin{multline}\label{e:MOmega_explicit}
M^+_{\Omega}(s, t)
= \frac{1}{4}\zeta(2-2s)\frac{(2\pi)^{2s-1+2it}}{\cos\left(\pi\left(s+\frac{1}{2}\right)\right)},
\\ \times \bigg\{
(2\pi)^{2s-1-2it}\zeta(1+2it)
\frac{\cos(\pi(2s-it))}{\sin(\pi(s-it))}
\sum_{\cuspa}\frac{\scrP_{\cuspa}(s, it; 1-s+it)}{\prod_{p\mid N}(1-p^{1-2s+2it})}
\frac{L_\cuspa(3/2-s+it, f\times \bar{g})}{\zeta^{(N)}(3-2s+2it)}
H_0(-2s+1; h)
%\frac{1}{\pi^2}\int_{-\infty}^\infty h(r) r\tanh(\pi r)
%\frac{\Gamma\left(ir-2s+1+it+\frac{k}{2}\right)\Gamma\left(-ir-2s+1+it+\frac{k}{2}\right)}
%{\Gamma\left(-ir+it+\frac{k}{2}\right)\Gamma\left(ir+it+\frac{k}{2}\right)}\; dr
% v=2s-1+it+\ell
\\ +
(2\pi)^{2s-1+2it}\zeta(1-2it)
\frac{\cos(\pi(2s+it))}{\sin(\pi(s+it))}
\sum_{\cuspa} \frac{\scrP_{\cuspa}(s, it; 1-s-it)}{\prod_{p\mid N}(1-p^{1-2s-2it})}
\frac{L_\cuspa(3/2-s-it, f\times \bar{g})}{\zeta^{(N)}(3-2s-2it)}
H_0(-2s+1-2it; h)
%\frac{1}{\pi^2} \int_{-\infty}^\infty h(r) r\tanh(\pi r)
%\frac{\Gamma\left(ir-2s+1-it+\frac{k}{2}\right)\Gamma\left(-ir-2s+1-it+\frac{k}{2}\right) }
%{\Gamma\left(-ir+it+\frac{k}{2}\right)\Gamma\left(ir+it+\frac{k}{2}\right)} \; dr
\bigg\}.
\end{multline}
Here $H_0(*; h)$ is given in \eqref{e:H0}. 
\end{lemma}

\begin{proof}
Recall \eqref{e:OD+Omegares1} for $\ell=0$. 
By opening up the inner product and applying \eqref{e:Rankin-Selberg_fg_a}, 
\begin{equation}
\left<E_\cuspa(*, 3/2-s-it), V_{f, g}\right>
= \frac{\Gamma\left(\frac{1}{2}-s-it+k\right)}{(4\pi)^{\frac{1}{2}-s-it+k} \zeta^{(N)}(3-2s-2it)}
L_\cuspa(3/2-s-it, f\times \bar{g}). 
\end{equation}
Recalling the functional equation of the Riemann zeta function and by the Legendre duplication formula 
and the reflection formula for gamma functions, 
\begin{equation}
\zeta(2s-1)\Gamma(2s-1) 
% details
%= \zeta(2-2s)\frac{\pi^{2s-\frac{3}{2}} \Gamma\left(1-s\right)\Gamma\left(2s-1\right)}{\Gamma\left(s-\frac{1}{2}\right)}
%= \zeta(2-2s)(2\pi)^{2s-2}\Gamma\left(1-s\right)\Gamma\left(s\right)
%\\
= -\zeta(2-2s)(2\pi)^{2s-2} \frac{\pi}{\cos\left(\pi\left(s+\frac{1}{2}\right)\right)}, 
\end{equation}
we get
\begin{multline}
M^+_{\Omega}(s, t)
%%\OD^+_{\Omega, \res, 1}(s, t; 0) 
%= - \frac{1}{4}\zeta(2-2s)\frac{(2\pi)^{2s-1+2it}}{\cos\left(\pi\left(s+\frac{1}{2}\right)\right)},
%\\ \times \bigg\{
%\frac{(2\pi)^{2s-1-2it}\zeta(1+2it)}{2\sin(\pi(s-it))}
%\sum_{\cuspa}\frac{\scrP_{\cuspa}(s, it; 1-s+it)}{\prod_{p\mid N}(1-p^{1-2s+2it})}
%\frac{L_\cuspa(3/2-s+it, f\times \bar{g})}{\zeta^{(N)}(3-2s+2it)}
%\\ \times 
%\frac{4}{2\pi i}\int_{(\sigma_u)} \frac{h(u/i)u}{\cos(\pi u)}
%\frac{\Gamma\left(u-2s+1+it+\frac{k}{2}\right)\Gamma\left(-u-2s+1+it+\frac{k}{2}\right)}
%{\Gamma\left(-u+it+\frac{k}{2}\right)\Gamma\left(u+it+\frac{k}{2}\right)}
%\\ \times \frac{\sin(\pi u)\cos(\pi(2s-it)) + \cos(\pi u) \sin(\pi(2s-it))}{\pi}
%\; du
%% v=2s-1+it+\ell
%\\ +
%\frac{(2\pi)^{2s-1+2it}\zeta(1-2it)}{2\sin(\pi(s+it))}
%\sum_{\cuspa} \frac{\scrP_{\cuspa}(s, it; 1-s-it)}{\prod_{p\mid N}(1-p^{1-2s-2it})}
%\frac{L_\cuspa(3/2-s-it, f\times \bar{g})}{\zeta^{(N)}(3-2s-2it)}
%\\ \times 
%\frac{4}{2\pi i}\int_{(\sigma_u)} \frac{h(u/i)u}{\cos(\pi u)}
%\frac{\Gamma\left(u-2s+1-it+\frac{k}{2}\right)\Gamma\left(-u-2s+1-it+\frac{k}{2}\right) }
%{\Gamma\left(-u+it+\frac{k}{2}\right)\Gamma\left(u+it+\frac{k}{2}\right)}
%\\ \times \frac{\sin(\pi u)\cos(\pi(2s+it)) + \cos(\pi u) \sin(\pi(2s+it))}{\pi} \; du
%\bigg\}.
= - \frac{1}{4}\zeta(2-2s)\frac{(2\pi)^{2s-1+2it}}{\cos\left(\pi\left(s+\frac{1}{2}\right)\right)}
\\ \times \sum_{\pm}
\frac{(2\pi)^{2s-1\mp 2it}\zeta(1\pm 2it)}{2\sin(\pi(s\mp it))}
\sum_{\cuspa}\frac{\scrP_{\cuspa}(s, it; 1-s\pm it)}{\prod_{p\mid N}(1-p^{1-2s\pm 2it})}
\frac{L_\cuspa(3/2-s\pm it, f\times \bar{g})}{\zeta^{(N)}(3-2s\pm 2it)}
\\ \times 
\frac{4}{2\pi i}\int_{(\sigma_u)} \frac{h(u/i)u}{\cos(\pi u)}
\frac{\Gamma\left(u-2s+1\pm it+\frac{k}{2}\right)\Gamma\left(-u-2s+1\pm it+\frac{k}{2}\right)}
{\Gamma\left(-u+it+\frac{k}{2}\right)\Gamma\left(u+it+\frac{k}{2}\right)}
\\ \times \frac{\sin(\pi u)\cos(\pi(2s\mp it)) + \cos(\pi u) \sin(\pi(2s\mp it))}{\pi}
\; du
\end{multline}

When $\Re(s)=\frac{1}{2}$, we can get 
\begin{equation}
\frac{4}{2\pi i}\int_{(\sigma_u)} h(u/i)u
\frac{\Gamma\left(u-2s+1+it+\frac{k}{2}\right)\Gamma\left(-u-2s+1+it+\frac{k}{2}\right)}
{\Gamma\left(-u+it+\frac{k}{2}\right)\Gamma\left(u+it+\frac{k}{2}\right)}
\; du = 0, 
\end{equation}
since $\Gamma\left(\pm u-2s+1+it+\frac{k}{2}\right)$ has no poles in $|\Re(u)|< \sigma_u$, $1+\epsilon < \sigma_u <\frac{3}{2}$ 
and $h(u/i)=h(-u/i)$. 
We now move the $u$ line of integration in the remaining integrals. 
After changing the variable $u=ir$, we get \eqref{e:MOmega_explicit}. 
%\begin{multline}
%M^+_{\Omega}(s, t) 
%= \frac{1}{4}\zeta(2-2s)\frac{(2\pi)^{2s-1+2it}}{\cos\left(\pi\left(s+\frac{1}{2}\right)\right)},
%\\ \times \bigg\{
%(2\pi)^{2s-1-2it}\zeta(1+2it)
%\frac{\cos(\pi(2s-it))}{\sin(\pi(s-it))}
%\sum_{\cuspa}\frac{\scrP_{\cuspa}(s, it; 1-s+it)}{\prod_{p\mid N}(1-p^{1-2s+2it})}
%\frac{L_\cuspa(3/2-s+it, f\times \bar{g})}{\zeta^{(N)}(3-2s+2it)}
%\\ \times 
%\frac{1}{\pi^2}\int_{-\infty}^\infty h(r) r\tanh(\pi r)
%\frac{\Gamma\left(ir-2s+1+it+\frac{k}{2}\right)\Gamma\left(-ir-2s+1+it+\frac{k}{2}\right)}
%{\Gamma\left(-ir+it+\frac{k}{2}\right)\Gamma\left(ir+it+\frac{k}{2}\right)}\; dr
%% v=2s-1+it+\ell
%\\ +
%(2\pi)^{2s-1+2it}\zeta(1-2it)
%\frac{\cos(\pi(2s+it))}{\sin(\pi(s+it))}
%\sum_{\cuspa} \frac{\scrP_{\cuspa}(s, it; 1-s-it)}{\prod_{p\mid N}(1-p^{1-2s-2it})}
%\frac{L_\cuspa(3/2-s-it, f\times \bar{g})}{\zeta^{(N)}(3-2s-2it)}
%\\ \times 
%\frac{1}{\pi^2} \int_{-\infty}^\infty h(r) r\tanh(\pi r)
%\frac{\Gamma\left(ir-2s+1-it+\frac{k}{2}\right)\Gamma\left(-ir-2s+1-it+\frac{k}{2}\right) }
%{\Gamma\left(-ir+it+\frac{k}{2}\right)\Gamma\left(ir+it+\frac{k}{2}\right)} \; dr
%\bigg\}.
%\end{multline}
\end{proof}

\section{Analysis of $\OD^-$}\label{s:OD-}

For $\Re(s)$, $\Re(w)>1$, let $s'=s+w+\frac{k}{2}-1$ and define 
\begin{equation}\label{e:cM3}
\scrM^{(3)}_{f, g}(s, w; it)
= \frac{\Gamma\left(k+w-1\right)}{(4\pi)^{k+w-1}} 
\zeta^{(N)}(2s') \sum_{m=1}^\infty \sum_{n=m+1}^\infty \frac{a(n-m)\overline{b(n)} \sigma_{-2it}(m; N) m^{it}}
{m^{s+\frac{k-1}{2}} n^{w+k-1}}. 
\end{equation}
Recalling \eqref{e:def_ODpm} and \eqref{e:L-holo_first}, 
and changing the indexes in the series, for $\Re(s)$ sufficiently large, we get
%\begin{multline}
%\OD^-(s, t)
%=-\frac{(2\pi)^{2it} \cos(\pi it)}{\pi} 
%\frac{4}{2\pi i} \int_{(\sigma_u)} \frac{h(u/i) u \tan(\pi u)}{\Gamma\left(u+it+\frac{k}{2}\right) \Gamma\left(-u+it+\frac{k}{2}\right)}
%\\ \times \frac{1}{2\pi i} \int_{(\sigma_0)} \Gamma\left(u-v\right) \Gamma\left(-u-v\right) 
%\Gamma\left(\frac{k}{2}-it+v\right) \Gamma\left(\frac{k}{2}+it+v\right)
%\\ \times  \zeta^{(N)}(2s) \sum_{m=1}^\infty \sum_{n=m+1}^\infty 
%\frac{a(n-m)\overline{b(n)}\sigma_{-2it}(m; N)}{m^{v-it+\frac{k}{2}} n^{s-v+\frac{k-1}{2}}}
%\; dv \; du.
%\end{multline}
\begin{multline}\label{e:OD-_cM3}
\OD^-(s, t)
=-\frac{(2\pi)^{2it} \cos(\pi it)}{\pi} 
\frac{4}{2\pi i} \int_{(\sigma_u)} \frac{h(u/i) u \tan(\pi u)}{\Gamma\left(u+it+\frac{k}{2}\right) \Gamma\left(-u+it+\frac{k}{2}\right)}
\\ \times \frac{1}{2\pi i} \int_{(\sigma_0)} \Gamma\left(u-v\right) \Gamma\left(-u-v\right) 
\Gamma\left(\frac{k}{2}-it+v\right) \Gamma\left(\frac{k}{2}+it+v\right)
\\ \times  
\frac{(4\pi)^{\frac{k}{2}+s-v-\frac{1}{2}}}{\Gamma\left(\frac{k}{2}+s-v-\frac{1}{2}\right)}
\scrM^{(3)}(v+1/2, s-v-(k-1)/2; it)
\; dv \; du.
\end{multline}
Here we assume that $1< \sigma_u < \frac{3}{2}$ and $-\frac{k}{2} < \sigma_0 < -\sigma_u$ 
so the poles of gamma functions in the $v$-integral are separated.

We now study the analytic properties of $\scrM^{(3)}_{f, g}(s, w; it)$ in the following section. 

\subsection{Shifted Dirichlet Series II}
Recalling \eqref{e:cM3}, 
\begin{equation}
\scrM^{(3)}_{f, g}(s, w; it)
= \frac{\Gamma\left(k+w-1\right)}{(4\pi)^{k+w-1}} 
\zeta^{(N)}(2s') \sum_{m=1}^\infty \frac{\sigma_{-2it}(m; N)m^{it}}{m^{s+\frac{k-1}{2}}} 
\sum_{n=m+1}^\infty \frac{a(n-m)\overline{b(n)}}{ n^{w+k-1}}, 
\end{equation}
where $s'=s+w+\frac{k}{2}-1$. 
The shifted inner sum over $n$ is essentially an inner product of $f$ and $g$ with a standard Poincar\'e series. 
In particular, for $m \geq 1$, if
\begin{equation}
P(z, w; m)
= \sum_{\gamma\in \Gamma_\infty \bsl \Gamma} (\Im(\gamma z))^w e^{2\pi i m \gamma z}, 
\end{equation}
then 
\begin{equation}
\left< P(*, w; m), V_{f, g}\right>
= \int_{\Gamma_0(N) \bsl \HH} P(z, w; m) y^k f(z) \overline{g(z)} \; \frac{dx\; dy}{y^2}
= \frac{\Gamma\left(k+w-1\right)}{(4\pi)^{k+w-1}}
\sum_{n=m+1}^\infty \frac{a(n-m) \overline{b(n)}}{n^{k+w-1}}.
\end{equation}
So we have 
\begin{equation}
\scrM^{(3)}_{f, g}(s, w; it)
= \zeta^{(N)}(2s') \sum_{m=1}^\infty \frac{\sigma_{-2it}(m; N)m^{it}}{m^{s+\frac{k-1}{2}}} 
\left<P(*, w; m), V_{f, g}\right>. 
\end{equation}

On the other hand, since $P(z, w; m)\in L^2(\Gamma_0(N)\bsl \HH)$, from its spectral expansion we obtain 
\begin{multline}
\left<P(*, w; m), V_{f, g}\right>
= \sum_{j\geq 1} \left<P(*, w; m), u_j\right> \left<u_j, V_{f, g}\right>
\\ + \sum_{\cuspa} \frac{1}{4\pi} \int_{-\infty}^\infty 
\left<P(*, w; m), E_{\cuspa}(*, 1/2+ir)\right>
\left<E_{\cuspa}(*, 1/2+ir), V_{f, g}\right>\; dr, 
\end{multline}
where 
\begin{equation}
\left<P(*, w; m), u_j\right> 
= \frac{\overline{\rho_j(m)}}{(2\pi m)^{w-\frac{1}{2}}}
\frac{\sqrt{\pi} 2^{\frac{1}{2}-w} \Gamma\left(w-\frac{1}{2}+ir_j\right) \Gamma\left(w-\frac{1}{2}-ir_j\right)}{\Gamma(w)}
\end{equation}
and
\begin{equation}
\left<P(*, w; m), E_{\cuspa}(*, 1/2+ir)\right>
= \frac{\overline{\tau_{\cuspa}\left(1/2+ir, m\right)}}{(2\pi m)^{w-\frac{1}{2}}}
\frac{\sqrt{\pi} 2^{\frac{1}{2}-w} \Gamma\left(w-\frac{1}{2}+ir\right) \Gamma\left(w-\frac{1}{2}-ir\right)}
{\Gamma(w)}.
\end{equation}

Therefore,
\begin{multline}\label{e:cM3_spec_def}
\scrM^{(3)}_{f, g}(s, w; it)
= \sum_{j\geq 1}
\frac{\sqrt{\pi} (4\pi)^{-w+\frac{1}{2}} 
\Gamma\left(w-\frac{1}{2}+ir_j\right) \Gamma\left(w-\frac{1}{2}-ir_j\right)}{\Gamma(w)}
\scrL(s', it; \overline{u_j})
%\zeta^{(N)}(2s') \sum_{m=1}^\infty \frac{\overline{\rho_j(m)} \sigma_{-2it}(m; N)m^{it}}{m^{s+w-1+\frac{k}{2}}} 
\left<u_j, V_{f, g}\right>
\\ + \frac{1}{4\pi} \int_{-\infty}^\infty 
\frac{\sqrt{\pi} (4\pi)^{-w+\frac{1}{2}} \Gamma\left(w-\frac{1}{2}+ir\right) \Gamma\left(w-\frac{1}{2}-ir\right)}
{\Gamma(w)}
\sum_{\cuspa}\scrL_{\cuspa}(s', it; ir)
%\zeta^{(N)}(2s') \sum_{m=1}^\infty \frac{\tau_{\cuspa}\left(1/2-ir, m\right) \sigma_{-2it}(m; N)m^{it}}{m^{s+w-1+\frac{k}{2}}} 
\left<E_{\cuspa}(*, 1/2+ir), V_{f, g}\right> \; dr.
\end{multline}
Here $\scrL(s, it; u_j)$ and $\scrL_{\cuspa}(s, it; ir)$ are given in \eqref{e:def_cL_uj} and \eqref{e:def_cL_cuspa} 
and their factorizations in Lemma~\ref{lem:scrL_fac_cusp} and Lemma~\ref{lem:scrL_fac_Eis} respectively. 
%\begin{equation}
%\scrL(s, it; u_j) = \zeta^{(N)}(2s) \sum_{m=1}^\infty \frac{\rho_j(m) \sigma_{-2it}(m; N) m^{it}}{m^s}
%\end{equation}
%and 
%\begin{equation}
%\scrL_{\cuspa}(s, it; ir) = \zeta^{(N)}(2s) \sum_{m=1}^\infty \frac{\tau_{\cuspa}(1/2+ir, m) \sigma_{-2it}(m; N) m^{it}}{m^s}. 
%\end{equation}
The series and the integral in \eqref{e:cM3_spec_def} are absolutely convergent.    
We begin with $\Re (s') >1$ but need to continue to $\Re (s') < 1$.   
In order to do this we set $z = ir$ and think of the integration as over the line $\Re (z) =0$.   
Setting $\Re (s')=  1+ \epsilon$, we bend the $z$ line to the right to pass over $s'$, obtaining a residue with a negative sign attached because of the direction the line moved.  
We then continue $s'$ to $\Re (s') = 1- \epsilon$, and then move the $z$ line back to $\Re (z) =0$, picking up another residue, this time with a positive sign.  
Finally we use the functional equation of the sum of the product of Eisenstein series over the cusps \eqref{e:series_Eis_cusosum_fe}. 
The end result is that when $\Re (s') = 1- \epsilon$ we have picked up an additional residual term, and we obtain
\begin{multline}\label{e:cM3_spec}
\scrM^{(3)}_{f, g}(s, w; it)
= \sum_{j\geq 1}
\frac{\sqrt{\pi} (4\pi)^{-w+\frac{1}{2}} 
\Gamma\left(w-\frac{1}{2}+ir_j\right) \Gamma\left(w-\frac{1}{2}-ir_j\right)}{\Gamma(w)}
\scrL(s', it; \overline{u_j})
%\zeta^{(N)}(2s') \sum_{m=1}^\infty \frac{\overline{\rho_j(m)} \sigma_{-2it}(m; N)m^{it}}{m^{s+w-1+\frac{k}{2}}} 
\left<u_j, V_{f, g}\right>
\\ + \frac{1}{4\pi} \int_{-\infty}^\infty 
\frac{\sqrt{\pi} (4\pi)^{-w+\frac{1}{2}} \Gamma\left(w-\frac{1}{2}+ir\right) \Gamma\left(w-\frac{1}{2}-ir\right)}
{\Gamma(w)}
\sum_{\cuspa}\scrL_{\cuspa}(s', it; ir)
%\zeta^{(N)}(2s') \sum_{m=1}^\infty \frac{\tau_{\cuspa}\left(1/2-ir, m\right) \sigma_{-2it}(m; N)m^{it}}{m^{s+w-1+\frac{k}{2}}} 
\left<E_{\cuspa}(*, 1/2+ir), V_{f, g}\right> \; dr
\\+\delta_{1/2 \leq \Re(s') < 1} \zeta(-1+2s')
\frac{\sqrt{\pi} (4\pi)^{-w+\frac{1}{2}}}{\Gamma\left(w\right)}
\bigg\{
% ir=1-s'+it
\frac{\Gamma\left(-s'+\frac{1}{2}+it+w\right) \Gamma\left(s'-\frac{3}{2}-it+w\right)}{\Gamma\left(s'-\frac{1}{2}-it\right)}
\\ \times \pi^{s'-\frac{1}{2}-it} \zeta(1+2it)
\sum_{\cuspa} \frac{\scrP_{\cuspa}(s', it; 1-s+it)}{\prod_{p\mid N}(1-p^{1-2s'+2it})} 
\left<E_{\cuspa}(*, 3/2-s'+it), V_{f, g}\right> 
% ir=1-s'-it
\\ + \frac{\Gamma\left(-s'+\frac{1}{2}-it+w\right) \Gamma\left(s'-\frac{3}{2}+it+w\right)}{\Gamma\left(s'-\frac{1}{2}+it\right)}
\\ \times \pi^{s'-\frac{1}{2}+it} \zeta(1-2it)
\sum_{\cuspa} \frac{\scrP_{\cuspa}(s', it; 1-s-it)}{\prod_{p\mid N}(1-p^{1-2s'-2it})} 
\left<E_{\cuspa}(*, 3/2-s'-it), V_{f, g}\right> 
\bigg\}.
\end{multline}
Now we have the following proposition. 
\begin{proposition}\label{prop:cM3}
For $\Re(w) > 1/2$, the function $\scrM^{(3)}_{f, g}(s, w; it)$ is analytic and it has a meromorphic continuation to all $s, w\in \C$.
\end{proposition}

\subsection{The meromorphic continuation of $\OD^-(s, t)$}
%Recalling \eqref{e:OD-_cM3}, 
%\begin{multline}\label{e:OD-_spec}
%\OD^-(s, t)
%=-\frac{(2\pi)^{2it}\cos(\pi it)}{\pi} 
%\frac{4}{2\pi i} \int_{(\sigma_u)} \frac{h(u/i) u \tan(\pi u)}{\Gamma\left(u+it+\frac{k}{2}\right) \Gamma\left(-u+it+\frac{k}{2}\right)}
%\\ \times \frac{1}{2\pi i} \int_{(\sigma_0)} \Gamma\left(u-v\right) \Gamma\left(-u-v\right) 
%\Gamma\left(\frac{k}{2}-it+v\right) \Gamma\left(\frac{k}{2}+it+v\right)
%\\ \times  
%\frac{(4\pi)^{\frac{k}{2}+s-v-\frac{1}{2}}}{\Gamma\left(\frac{k}{2}+s-v-\frac{1}{2}\right)}
%\scrM^{(3)}(v+1/2, s-v-(k-1)/2; it)
%\; dv \; du.
%\end{multline}
%Here $1< \sigma_u < \frac{3}{2}$ and $-\frac{k}{2} < \sigma_0 < \min\{-1+2\Re(s)-\frac{k}{2}, -\sigma_u\}$. 
Recall \eqref{e:OD-_cM3}. 
The choice of $-\frac{k}{2}< \sigma_0< \min\{-1+2\Re(s)-\frac{k}{2}, -\sigma_u\}$ with $1< \sigma_u < \frac{3}{2}$, 
is made as the vertical line $\Re(v)=\sigma_0$ in \eqref{e:OD-_cM3} 
passes between the poles of the gamma functions in the polar piece of 
$\scrM^{(3)}(v+1/2, s-v-(k-1)/2; it)$ (see \eqref{e:cM3_spec}).
Note that the interval is non-empty when $\Re(s)>\frac{1}{2}$, and $k \ge 4$.  
We are assuming now that  $\Re(s)=\frac{1}{2}+\epsilon$ for some $\epsilon>0$.   
We also choose $\sigma_0 = -\frac{k}{2}+\epsilon$. 
By Proposition~\ref{prop:cM3}, for $s' = s$ and $\Re(s-v-\frac{k-1}{2}) > \frac{1}{2}$, 
the function $\scrM^{(3)}(v+1/2, s-v-(k-1)/2; it)$ is analytic 
and it has a meromorphic continuation to all $s, v\in \C$. 
By \eqref{e:cM3_spec}, we have 
\begin{equation}\label{e:OD-_decomp}
\OD^-(s, t) = \OD_0^-(s, t)+\OD_{\Omega}^-(s, t),
\end{equation}
where
\begin{multline}\label{e:OD-_0}
\OD^-_0(s, t) 
= 
%-\frac{(2\pi)^{2it}\cos(\pi it)}{\pi} 
%\frac{4}{2\pi i} \int_{(\sigma_u)} \frac{h(u/i) u \tan(\pi u)}{\Gamma\left(u+it+\frac{k}{2}\right) \Gamma\left(-u+it+\frac{k}{2}\right)}
%\\ \times \frac{1}{2\pi i} \int_{(\sigma_0)} \Gamma\left(u-v\right) \Gamma\left(-u-v\right) 
%\Gamma\left(\frac{k}{2}-it+v\right) \Gamma\left(\frac{k}{2}+it+v\right)
%
%\frac{(4\pi)^{s-v+\frac{k}{2}+\frac{1}{2}}}{\Gamma\left(s-v+\frac{k}{2}+\frac{1}{2}\right)}
%\scrM^{(3)}_{f, g}(v+1/2, s-v-(k-1)/2; it)
%\; dv\; du
%
%\\ = 
-\frac{(2\pi)^{2it}\cos(\pi it)}{\pi} 
\frac{4}{2\pi i} \int_{(\sigma_u)} \frac{h(u/i) u \tan(\pi u)}{\Gamma\left(u+it+\frac{k}{2}\right) \Gamma\left(-u+it+\frac{k}{2}\right)}
\\ \times \frac{1}{2\pi i} \int_{(\sigma_0)} \Gamma\left(u-v\right) \Gamma\left(-u-v\right) 
\Gamma\left(\frac{k}{2}-it+v\right) \Gamma\left(\frac{k}{2}+it+v\right)
\\ \times \bigg\{
\sum_{j\geq 1}
\frac{\sqrt{\pi} (4\pi)^{k-\frac{1}{2}} 
\Gamma\left(s-v-\frac{k}{2}+ir_j\right) \Gamma\left(s-v-\frac{k}{2}-ir_j\right)}
{\Gamma\left(s-v+\frac{k}{2}-\frac{1}{2}\right)\Gamma\left(s-v-\frac{k}{2}+\frac{1}{2}\right)}
\scrL(s, it; \overline{u_j}) \left<u_j, V_{f, g}\right>
\\ + \frac{1}{4\pi} \int_{-\infty}^\infty 
\frac{\sqrt{\pi} (4\pi)^{k-\frac{1}{2}} \Gamma\left(s-v-\frac{k}{2}+ir\right) \Gamma\left(s-v-\frac{k}{2}-ir\right)}
{\Gamma\left(s-v+\frac{k}{2}-\frac{1}{2}\right)\Gamma\left(s-v-\frac{k}{2}+\frac{1}{2}\right)}
\sum_{\cuspa}\scrL_{\cuspa}(s, it; ir) \left<E_{\cuspa}(*, 1/2+ir), V_{f, g}\right> \; dr
\bigg\} \; dv \; du
\end{multline}
and 
\begin{multline}
\OD_{\Omega}^-(s, t) 
%=
%- \zeta(-1+2s)
%\frac{(2\pi)^{2it}\cos(\pi it)}{\pi} 
%\frac{4}{2\pi i} \int_{(\sigma_u)} \frac{h(u/i) u \tan(\pi u)}{\Gamma\left(u+it+\frac{k}{2}\right) \Gamma\left(-u+it+\frac{k}{2}\right)}
%\\ \times \frac{1}{2\pi i} \int_{(\sigma_0)} \Gamma\left(u-v\right) \Gamma\left(-u-v\right) 
%%
%\frac{\sqrt{\pi} (4\pi)^{k-\frac{1}{2}}\Gamma\left(\frac{k}{2}-it+v\right) \Gamma\left(\frac{k}{2}+it+v\right)}
%{\Gamma\left(s-v+\frac{k}{2}-\frac{1}{2}\right)\Gamma\left(s-v-\frac{k}{2}+\frac{1}{2}\right)}
%\\ \times \bigg\{
%% ir=1-s+it
%\frac{\Gamma\left(-v+it+1-\frac{k}{2}\right) \Gamma\left(-v+2s-1-it-\frac{k}{2}\right)}
%{\Gamma\left(s-\frac{1}{2}-it\right)}
%\\ \times \pi^{s-\frac{1}{2}-it} \zeta(1+2it)
%\sum_{\cuspa} \frac{\scrP_{\cuspa}(s, it; 1-s+it)}{\prod_{p\mid N}(1-p^{1-2s+2it})} 
%\left<E_{\cuspa}(*, 3/2-s+it), V_{f, g}\right> 
%% ir=1-s-it
%\\ + \frac{\Gamma\left(-v-it+1-\frac{k}{2}\right) \Gamma\left(-v+2s-1+it-\frac{k}{2}\right)}{\Gamma\left(s-\frac{1}{2}+it\right)}
%\\ \times \pi^{s-\frac{1}{2}+it} \zeta(1-2it)
%\sum_{\cuspa} \frac{\scrP_{\cuspa}(s, it; 1-s-it)}{\prod_{p\mid N}(1-p^{1-2s-2it})} 
%\left<E_{\cuspa}(*, 3/2-s-it), V_{f, g}\right> 
%\bigg\} 
%\; dv\; du.
=
- \zeta(-1+2s)
\frac{(2\pi)^{2it}\cos(\pi it)}{\pi} 
\frac{4}{2\pi i} \int_{(\sigma_u)} \frac{h(u/i) u \tan(\pi u)}{\Gamma\left(u+it+\frac{k}{2}\right) \Gamma\left(-u+it+\frac{k}{2}\right)}
\\ \times \frac{1}{2\pi i} \int_{(\sigma_0)} \Gamma\left(u-v\right) \Gamma\left(-u-v\right) 
\frac{\sqrt{\pi} (4\pi)^{k-\frac{1}{2}}\Gamma\left(\frac{k}{2}-it+v\right) \Gamma\left(\frac{k}{2}+it+v\right)}
{\Gamma\left(s-v+\frac{k}{2}-\frac{1}{2}\right)\Gamma\left(s-v-\frac{k}{2}+\frac{1}{2}\right)}
\\ \times \sum_{\pm}
% ir=1-s\pm it
\frac{\Gamma\left(-v\pm it+1-\frac{k}{2}\right) \Gamma\left(-v+2s-1\mp it-\frac{k}{2}\right)}
{\Gamma\left(s-\frac{1}{2}\mp it\right)}
\\ \times \pi^{s-\frac{1}{2}\mp it} \zeta(1\pm 2it)
\sum_{\cuspa} \frac{\scrP_{\cuspa}(s, it; 1-s\pm it)}{\prod_{p\mid N}(1-p^{1-2s\pm 2it})} 
\left<E_{\cuspa}(*, 3/2-s\pm it), V_{f, g}\right> 
\; dv\; du.
\end{multline}

In order to continue $s$ to $\Re(s)=1/2$, we first consider the $v$-integral in $\OD_{\Omega}^-(s, t)$.  
We  move the contour from $\Re(v)=-\frac{k}{2}+\epsilon$ to $\Re(v) = -\frac{k}{2}+3\epsilon$, 
passing over the pole of $\Gamma\left(-v+2s-1\pm it-\frac{k}{2}\right)$ 
at $-v+2s-1\pm it -\frac{k}{2}=0$.
This gives us
\begin{equation}\label{e:OD-Omega_decomp}
\OD_{\Omega}^-(s, t) 
= \OD_{\Omega, \res}^-(s, t) 
+ \OD_{\Omega, \intg}^-(s, t), 
\end{equation}
where 
% v=2s-1-it-k/2 and v=2s-1+it-k/2
\begin{multline}\label{e:OD-_Omegares}
\OD_{\Omega, \res}^-(s, t) 
%= -\zeta(-1+2s)
%\frac{(2\pi)^{2it}\cos(\pi it)}{\pi} 
%\\ \times \bigg\{
%% ir=1-s+it
%% v=2s-1-it-k/2
%\frac{4}{2\pi i} \int_{(\sigma_u)} h(u/i) u \tan(\pi u)
%\frac{\Gamma\left(u-2s+1+it+\frac{k}{2}\right) \Gamma\left(-u-2s+1+it+\frac{k}{2}\right)}
%{\Gamma\left(u+it+\frac{k}{2}\right) \Gamma\left(-u+it+\frac{k}{2}\right)}\; du
%\\ \times \frac{\sqrt{\pi} (4\pi)^{k-\frac{1}{2}}\Gamma\left(2s-1-2it\right) \Gamma\left(2s-1\right)}
%{\Gamma\left(-s+it+\frac{1}{2}+k\right)\Gamma\left(-s+it+\frac{3}{2}\right)}
%\frac{\Gamma\left(-2s+2+2it\right)}{\Gamma\left(s-\frac{1}{2}-it\right)}
%\\ \times \pi^{s-\frac{1}{2}-it} \zeta(1+2it)
%\sum_{\cuspa} \frac{\scrP_{\cuspa}(s, it; 1-s+it)}{\prod_{p\mid N}(1-p^{1-2s+2it})} 
%\left<E_{\cuspa}(*, 3/2-s+it), V_{f, g}\right> 
%% ir=1-s-it
%% v=2s-1+it-k/2
%\\ + 
%\frac{4}{2\pi i} \int_{(\sigma_u)} h(u/i) u \tan(\pi u)
%\frac{\Gamma\left(u-2s+1-it+\frac{k}{2}\right) \Gamma\left(-u-2s+1-it+\frac{k}{2}\right)}
%{\Gamma\left(u+it+\frac{k}{2}\right) \Gamma\left(-u+it+\frac{k}{2}\right)}\;du 
%%
%\\ \times \frac{\sqrt{\pi} (4\pi)^{k-\frac{1}{2}}\Gamma\left(2s-1\right) \Gamma\left(2s-1+2it\right)}
%{\Gamma\left(-s-it+\frac{1}{2}+k\right)\Gamma\left(-s-it+\frac{3}{2}\right)}
%\frac{\Gamma\left(-2s+2-2it\right)}{\Gamma\left(s-\frac{1}{2}+it\right)}
%\\ \times \pi^{s-\frac{1}{2}+it} \zeta(1-2it)
%\sum_{\cuspa} \frac{\scrP_{\cuspa}(s, it; 1-s-it)}{\prod_{p\mid N}(1-p^{1-2s-2it})} 
%\left<E_{\cuspa}(*, 3/2-s-it), V_{f, g}\right> 
%\bigg\}
= -\zeta(-1+2s)
\frac{(2\pi)^{2it}\cos(\pi it)}{\pi} 
\\ \times \sum_{\pm}
% ir=1-s\pm it
% v=2s-1\mp it-k/2
\frac{4}{2\pi i} \int_{(\sigma_u)} h(u/i) u \tan(\pi u)
\frac{\Gamma\left(u-2s+1\pm it+\frac{k}{2}\right) \Gamma\left(-u-2s+1\pm it+\frac{k}{2}\right)}
{\Gamma\left(u+it+\frac{k}{2}\right) \Gamma\left(-u+it+\frac{k}{2}\right)}\; du
\\ \times \frac{\sqrt{\pi} (4\pi)^{k-\frac{1}{2}}\Gamma\left(2s-1\mp 2it\right) \Gamma\left(2s-1\right)}
{\Gamma\left(-s\pm it+\frac{1}{2}+k\right)\Gamma\left(-s\pm it+\frac{3}{2}\right)}
\frac{\Gamma\left(-2s+2\pm 2it\right)}{\Gamma\left(s-\frac{1}{2}\mp it\right)}
\\ \times \pi^{s-\frac{1}{2}\mp it} \zeta(1\pm 2it)
\sum_{\cuspa} \frac{\scrP_{\cuspa}(s, it; 1-s\pm it)}{\prod_{p\mid N}(1-p^{1-2s\pm 2it})} 
\left<E_{\cuspa}(*, 3/2-s\pm it), V_{f, g}\right> 
\end{multline}
and 
\begin{multline}\label{e:OD-_Omegaint}
\OD_{\Omega, \intg}^-(s, t) 
%= -\zeta(-1+2s)
%\frac{(2\pi)^{2it}\cos(\pi it)}{\pi} 
%\frac{4}{2\pi i} \int_{(\sigma_u)} \frac{h(u/i) u \tan(\pi u)}{\Gamma\left(u+it+\frac{k}{2}\right) \Gamma\left(-u+it+\frac{k}{2}\right)}
%\\ \times \frac{1}{2\pi i} \int_{(-\frac{k}{2}+3\epsilon)} \Gamma\left(u-v\right) \Gamma\left(-u-v\right) 
%%
%\frac{\sqrt{\pi} (4\pi)^{k-\frac{1}{2}}\Gamma\left(\frac{k}{2}-it+v\right) \Gamma\left(\frac{k}{2}+it+v\right)}
%{\Gamma\left(s-v+\frac{k}{2}-\frac{1}{2}\right)\Gamma\left(s-v-\frac{k}{2}+\frac{1}{2}\right)}
%\\ \times \bigg\{
%% ir=1-s+it
%\frac{\Gamma\left(-v+it+1-\frac{k}{2}\right) \Gamma\left(-v+2s-1-it-\frac{k}{2}\right)}
%{\Gamma\left(s-\frac{1}{2}-it\right)}
%\\ \times \pi^{s-\frac{1}{2}-it} \zeta(1+2it)
%\sum_{\cuspa} \frac{\scrP_{\cuspa}(s, it; 1-s+it)}{\prod_{p\mid N}(1-p^{1-2s+2it})} 
%\left<E_{\cuspa}(*, 3/2-s+it), V_{f, g}\right> 
%% ir=1-s-it
%\\ + \frac{\Gamma\left(-v-it+1-\frac{k}{2}\right) \Gamma\left(-v+2s-1+it-\frac{k}{2}\right)}{\Gamma\left(s-\frac{1}{2}+it\right)}
%\\ \times \pi^{s-\frac{1}{2}+it} \zeta(1-2it)
%\sum_{\cuspa} \frac{\scrP_{\cuspa}(s, it; 1-s-it)}{\prod_{p\mid N}(1-p^{1-2s-2it})} 
%\left<E_{\cuspa}(*, 3/2-s-it), V_{f, g}\right> 
%\bigg\} 
%\; dv\; du.
= -\zeta(-1+2s)
\frac{(2\pi)^{2it}\cos(\pi it)(4\pi)^{k-\frac{1}{2}}}{\sqrt{\pi}}
\frac{4}{2\pi i} \int_{(\sigma_u)} \frac{h(u/i) u \tan(\pi u)}{\Gamma\left(u+it+\frac{k}{2}\right) \Gamma\left(-u+it+\frac{k}{2}\right)}
\\ \times \sum_{\pm}
% ir=1-s\pm it
\frac{1}{2\pi i} \int_{(-\frac{k}{2}+3\epsilon)} \Gamma\left(u-v\right) \Gamma\left(-u-v\right) 
\frac{\Gamma\left(\frac{k}{2}-it+v\right) \Gamma\left(\frac{k}{2}+it+v\right)}
{\Gamma\left(s-v+\frac{k}{2}-\frac{1}{2}\right)\Gamma\left(s-v-\frac{k}{2}+\frac{1}{2}\right)}
\\ \times 
\frac{\Gamma\left(-v\pm it+1-\frac{k}{2}\right) \Gamma\left(-v+2s-1\mp it-\frac{k}{2}\right)}
{\Gamma\left(s-\frac{1}{2}\mp it\right)}\; dv\; du
\\ \times \pi^{s-\frac{1}{2}\mp it} \zeta(1\pm 2it)
\sum_{\cuspa} \frac{\scrP_{\cuspa}(s, it; 1-s\pm it)}{\prod_{p\mid N}(1-p^{1-2s\pm 2it})} 
\left<E_{\cuspa}(*, 3/2-s\pm it), V_{f, g}\right>.
\end{multline}
We then continue $s$ back to $\Re (s) =\frac{1}{2}$, passing through no polar lines, 
and then move the $v$ line back to $\sigma_v = -\frac{k}{2} + \epsilon$, again passing over no poles.

Let 
\begin{equation}\label{e:M-Omega_def}
M^-_{\Omega}(s, t) = \OD^-_{\Omega, \res}(s, t)
\end{equation}
and
\begin{equation}\label{e:scrE-}
\scrE^-(s, t) = \OD^-_0(s, t) + \OD^-_{\Omega, \intg}(s, t), 
\end{equation}
where $\OD^-_0(s, t)$ and $\OD^-_{\Omega, \intg}(s, t)$ are given in \eqref{e:OD-_0} and \eqref{e:OD-_Omegaint} respectively.  
Then the above discussion has proved the following proposition:
\begin{proposition}
On $\Re(s)=\frac{1}{2}$, we have 
\begin{equation}\label{e:OD-_mainerror}
\OD^-(s, t) = M^-_{\Omega}(s, t) + \scrE^-(s, t). 
\end{equation}
\end{proposition}

\subsection{An upper bound for $\OD^-_0(s, t)$ and  $\OD_{\Omega, \intg}^-(s, t)$}
Recall we have set $s = \frac{1}{2} -it'$, with $t' = \pm t$, $|t| = T^\beta$, with $0<\beta <1$
We choose $h=h_{T, \alpha}$ as in \eqref{e:hdef} with $\frac{1}{3} < \alpha < \frac{2}{3}$.
The object of this section is to prove

\begin{proposition}\label{prop:OD-up}
If $t' = \pm t$ and $0<\beta \le 1-\alpha$, , then, 
\begin{equation}\label{e:OD-upper_0}
\OD^-_0(s, t) \ll  T^{1+\alpha - (3\alpha-1)\frac{k+1}{2}+ \epsilon}. 
\end{equation}
If $t' = t$ and $1-\alpha < \beta < 2\alpha$, then
\begin{equation}\label{e:OD-upper_1}
\OD^-_0(s, t) \ll T^{ 1+ \alpha - (2\alpha-\beta)\frac{k+1}{2}+ \epsilon}.
\end{equation}
If $t' = -t$ and $1-\alpha < \beta < \frac{\alpha+1}{2}$, and $\alpha = 2\beta-1+ \delta< \frac{2}{3}$, $\delta>0$, then
\begin{equation}\label{e:OD-upper_2}
\OD^-_0(s, t) \ll T^{1+\alpha -(1-3\frac{\beta}{2}+\delta \frac{k}{2})+\epsilon}.
\end{equation}
If $t' = \pm t$ and $\beta \le 1-\alpha$,  then
\begin{equation}
\OD_{\Omega, \intg}^-(s, t) \ll T^\epsilon.
\end{equation}
If $t' = \pm t$ and $1-\alpha < \beta < 1$, then
\begin{equation}
\OD_{\Omega, \intg}^-(s, t) \ll T^{\alpha +\epsilon}.
\end{equation}

\end{proposition}

\begin{remark}\label{remark:OD-}
The first  two expressions, \eqref{e:OD-upper_0}, when $1-\beta < \alpha$, and \eqref{e:OD-upper_1}, 
when $1-\alpha < \beta \leq 2\alpha$ will always be smaller than the main term, 
which is on the order of $T^{1+\alpha +\epsilon}$ for all $k \ge 2$.  
On the other hand \eqref{e:OD-upper_2} is only smaller than the main term for sufficiently large $k$.  
Specifically, for \eqref{e:OD-upper_2} the main term will dominate 
if $\frac{2}{3} >\alpha = 2\beta-1 +\delta$ and $k > \frac{3\beta-2}{\delta}$.
\end{remark}

\begin{proof}
Recall \eqref{e:OD-_0}
%\eqref{e:OD-_spec}
\begin{multline}\label{e:OD-_0_in}
\OD^-_0(s, t)
=-\frac{(2\pi)^{2it}\cos(\pi it)}{\pi} 
\frac{4}{2\pi i} \int_{(\sigma_u)} \frac{h(u/i) u \tan(\pi u)}{\Gamma\left(u+it+\frac{k}{2}\right) \Gamma\left(-u+it+\frac{k}{2}\right)}
\\ \times \frac{1}{2\pi i} \int_{(\sigma_0)} \Gamma\left(u-v\right) \Gamma\left(-u-v\right) 
\Gamma\left(\frac{k}{2}-it+v\right) \Gamma\left(\frac{k}{2}+it+v\right)
\\ \times \bigg\{
\sum_{j\geq 1}
\frac{\sqrt{\pi} (4\pi)^{k-\frac{1}{2}} 
\Gamma\left(s-v-\frac{k}{2}+ir_j\right) \Gamma\left(s-v-\frac{k}{2}-ir_j\right)}
{\Gamma\left(s-v+\frac{k}{2}-\frac{1}{2}\right)\Gamma\left(s-v-\frac{k}{2}+\frac{1}{2}\right)}
\scrL(s, it; \overline{u_j}) \left<u_j, V_{f, g}\right>
\\ + \frac{1}{4\pi} \int_{-\infty}^\infty 
\frac{\sqrt{\pi} (4\pi)^{k-\frac{1}{2}} \Gamma\left(s-v-\frac{k}{2}+ir\right) \Gamma\left(s-v-\frac{k}{2}-ir\right)}
{\Gamma\left(s-v+\frac{k}{2}-\frac{1}{2}\right)\Gamma\left(s-v-\frac{k}{2}+\frac{1}{2}\right)}
\sum_{\cuspa}\scrL_{\cuspa}(s, it; ir) \left<E_{\cuspa}(*, 1/2+ir), V_{f, g}\right> \; dr
\bigg\} \; dv \; du. 
\end{multline}
Here $1< \sigma_u < \frac{3}{2}$ and $-\frac{k}{2} < \sigma_0 < -\sigma_u$. 
As in the case of the residual and integral contributions to $\OD^+(s, t)$, 
we will need to move the $v$ line of integration.   In these case it can be moved as far as desired past the first pole, 
at $\sigma_v = \Re(v) =\frac12$.      
The dominant error term  comes from this residue and the integral can be made to be bounded 
by an arbitrarily large negative power of $T$ by moving $\sigma_v$ further.  
We will see shortly that potential residues encountered from the gamma functions $\Gamma(u-v)$ and $\Gamma(-u-v)$ 
by moving $\sigma_v$ will be exponentially decayed by the function $h(u/i)$ 
because of the range that the imaginary part of $v$ is effectively restricted to.

Let $v=\sigma_v+ir$ and $u=\sigma_u+i\gamma$. 
We begin by using Stirling's formula to estimate the argument of the exponential contribution from the gamma factors 
in the discrete part of \eqref{e:OD-_0_in}.
This is 
\begin{equation}
-2|t| +2\max\{|r|,|\gamma|\}+2\max\{|r|,|t|\}-2\max\{|\gamma|,|t|\}-2|t'+r|+2\max\{|t'+r|,|r_j|\}.
\end{equation}
We have $|\gamma| >|t| = T^\beta$, with $\beta <1$, so, if $|r| > |\gamma|$ then this becomes
\begin{equation}
-2|t| +2|r|+2|r|-2|\gamma|
\end{equation}
if $|t' +r| \ge |r_j|$, and
\begin{equation}
-2|t| +2|r|+2|r|-2|\gamma|+ |r_j| - |t'+r|
\end{equation}
if $|r_j| \ge |t'+r|$.
In both cases this is greater than $2(|r|-|t|)+ 2(|r|-|\gamma|) >0$, so we have exponential decay.
Thus we now take $|r| \le \gamma$.   Here we have 
\begin{equation}
-2|t| +2\max\{|r|,|t|\}-2|t'+r|+2\max\{|t'+r|,|r_j|\}.
\end{equation}
This has exponential decay if $|r| >|t|$, so we are reduced to the case $|r| \le |t|$.  The above then reduces to
\begin{equation}
2\max\{|t'+r|,|r_j|\}-2|t'+r|,
\end{equation}
making it clear that we have zero exponential decay precisely when $|r_j| < |t'+r|$ and $|r| \le |t|$.

The polynomial part of the integrand of \eqref{e:OD-_0_in} is thus bounded above by a multiple of
\begin{multline}\label{OD-3}
\frac{|\gamma| |\gamma -r|^{-\sigma_v+\sigma_u-\frac{1}{2}}
|\gamma + r|^{-\sigma_v-\sigma_u-\frac{1}{2}}|r-t|^{\sigma_v+\frac{k-1}{2}}
|r+t|^{\sigma_v+\frac{k-1}{2}}}{|\gamma +t|^{\sigma_u+\frac{k-1}{2}}|-\gamma +t|^{-\sigma_u+\frac{k-1}{2}}
|t'+r|^{-2\sigma_v}}
\\ \times \sum_{ |r_j| <|t'+r|} |t'+r -r_j|^{-\sigma_v-\frac{k}{2}}  |t'+r +r_j|^{-\sigma_v-\frac{k}{2}}
\scrL(1/2-it', it; \overline{u_j})
%L^{(N)}(1/2, \overline{u_j}) L^{(N)}(1/2 \pm2it, \overline{u_j})\rho_j(d)P(s', it; \overline{u_j})
\left<u_j, V_{f, g}\right>.
\end{multline}
Isolating the part with the exponent $\sigma_u$ gives us
\begin{equation*}
\left(\frac{|\gamma -r|}{|\gamma +r|}\right)^{\sigma_u}\left(\frac{|\gamma -t|}{|\gamma +t|}\right)^{\sigma_u}.
\end{equation*}
Again referring to Remark~\ref{remark:stirling}, we can now move $\sigma_u$ to either $\pm K T^\alpha$, 
as any poles of $\Gamma\left(-v+u\right) \Gamma\left(-v-u\right)$ that are crossed over 
will have residues exponentially decayed by the function $h(u/i)$.
Doing this, we have
\begin{equation*}
\left(\frac{|\gamma -r|}{|\gamma +r|}\right)^{\sigma_u}\left(\frac{|\gamma -t|}{|\gamma +t|}\right)^{\sigma_u}
\ll \left(\frac{|1 -\frac{r}{T}|}{|1+\frac{r}{T}|}\right)^{T\frac{\sigma_u}{T}} \left(\frac{|1 -\frac{t}{T}|}{|1 +\frac{t}{T}|}\right)^{T\frac{\sigma_u}{T}} 
\ll e^{-2|r+t|\frac{K}{T^{1-\alpha}}}.
\end{equation*}
As a consequence, there is exponential decay unless $|r+t| \ll T^{1-\alpha}$.    
If $r$ and $t$ have the same signs then  this condition implies $\beta \le 1-\alpha$, 
in which case $|r|, |t| \ll T^{1-\alpha}$.  
If $r$ and $t$ have opposite signs than $|r+t|=|T^\beta -|r|| \ll T^{1-\alpha}$.   Thus in both cases, if $t = t'$ the sum is over 
$|r_j| \ll T^{1-\alpha}$.   If $t = -t'$ and $r,t$ are the same sign then as $|r|, |t| \ll T^{1-\alpha}$, it follows that
$|r+t'| \ll T^{1-\alpha}$ and the sum is over $|r_j| \ll T^{1-\alpha}$.    If $t = -t'$ and $r,t$ have opposite signs then  $r,t'$ have the same sign and $|r+t'| \ll T^\beta$.  Thus in this case the sum is over $|r_j| \ll T^{\beta}$. 

We will now proceed to find an upper bound for the expression in \eqref{OD-3}.  
This same upper bound (after integration over $|T-|\gamma|| \ll T^\alpha$) 
will be the upper bound for residues accumulated as the $v$ line passes over poles of the gamma functions at  
$s-v-\frac{k}{2}\pm ir_j= -m$, for $m\ge 0$.  We will focus on the contribution from the discrete part.   
The corresponding poles  of the continuous part are at $s-v-\frac{k}{2}\pm ir= -m$, 
for $m\ge 0$, with the sum over the spectrum being replaced by an integral over $r$, and their contribution is of a lower order.  

As  $|r|,|t| \ll T^\beta$, with $\beta <1$, and $|T-|\gamma|| \ll T^\alpha$, with $\alpha <\frac{2}{3}$, 
(with apologies for the abuse of notation in which we also use $r$ to denote the imaginary part of $v$), 
$v = \sigma_v + ir$, it follows that
\begin{equation}
T \ll |\gamma \pm r|\ll T \quad \text{and} \quad T \ll |\gamma \pm t| \ll T.
\end{equation}
As a consequence of this, the expression in \eqref{OD-3} is bounded above by a constant times
\begin{equation}\label{OD-4}
T^{1-2\sigma_v-k} |r-t|^{\sigma_v+\frac{k-1}{2}} |r+t|^{\sigma_v+\frac{k-1}{2}} |r+t'|^{\sigma_v-\frac{k}{2}} 
\sum_{ |r_j| <|t'+r|} 
%L^{(N)}(1/2, \overline{u_j}) L^{(N)}(1/2 \pm2it, \overline{u_j})\rho_j(d)P(s', it; \overline{u_j})
\scrL(1/2-it', it; \overline{u_j})
\left<u_j, V_{f, g}\right>.
\end{equation}
Here we have bounded  one of the $|t'+r \pm r_j|^{-\sigma_v-\frac{k}{2}}$ from above by $1$, and the other by
$|t'+r|^{-\sigma_v-\frac{k}{2}}$.   

First, consider the case $t'=t$.  By the discussion above the sum over $j$ is over $|r_j| \ll T^{1-\alpha}$ .   In this case,
applying Lemma~\ref{lem:sieve} to \eqref{OD-4}, with the constant $C= 1-\alpha$,  we see that the expression in \eqref{OD-4}
is bounded above by a constant times
\begin{equation}
T^{1-2\sigma_v-k} |r-t|^{\sigma_v+\frac{k-1}{2}} |r+t|^{\sigma_v+\frac{k-1}{2}} |r+t|^{\sigma_v-\frac{k}{2}}
T^{\frac{\max\{\beta, 1-\alpha\}}{2}+\frac{(1-\alpha)}{2}+(1-\alpha)k +\epsilon }.
\end{equation}
If $r$ and $t$ have the same sign then $|r|, |t|  \ll T^{1-\alpha}$,  $|r\pm t| \ll T^{1-\alpha}$, 
and the above has an upper bound of a constant times
\begin{equation}\label{betasmall}
T^{1-2\sigma_v -k + (1-\alpha)(3\sigma_v +\frac{k}{2} -1) + 1-\alpha + (1-\alpha)k +\epsilon}=
T^{(1-3\alpha)\sigma_v +(1-3\alpha)\frac{k}{2}+1+\epsilon}.
\end{equation}
If $r$ and $t$ have opposite signs then $|r+t| \ll T^{1-\alpha}$ and $|r-t| \ll T^\beta$.
In this case the above  has an upper bound of a constant times
\begin{multline}
T^{1-2\sigma_v -k + \beta(\sigma_v +\frac{k-1}{2}) + (1-\alpha)(2\sigma_v -\frac{1}{2}) 
+\frac{\max\{\beta, 1-\alpha\}}{2} +\frac{1-\alpha}{2}+ (1-\alpha)k +\epsilon}
\\= T^{(\beta -2\alpha)\sigma_v +(\frac{\beta}{2} -\alpha)k+1-\frac{\beta}{2} + \frac{\max\{\beta, 1-\alpha\}}{2}+\epsilon}.
\end{multline}
In order to have the coefficients of $\sigma_v$, $k$ negative, we will require $\alpha > \frac{\beta}{2}$ 
when $\beta >1-\alpha$.

We now estimate the residue contribution.
Initially $\sigma_v = -\epsilon -\sigma_1$, for sufficiently small $\epsilon ,\sigma_1>0$.  
The poles occur at $\frac{1}{2} - \sigma_v -\frac{k}{2} = -m$, for $m \ge 0$, and thus 
the first pole, as we move $\sigma_v$ in a positive direction is at $\sigma_v = 1/2$.  
As the exponent is  $(1-3\alpha)\sigma_v +(1-3\alpha)\frac{k}{2} +1+ \epsilon$, for $\alpha >\frac{1}{3}$, 
the coefficients of $\sigma_v$, $k$ are negative and the dominant residue is this first one.  
If $\beta \le 1-\alpha$, the upper bound for the residue is  $T^{\frac{1 -3\alpha}{2} +(1-3\alpha)\frac{k}{2} +1 +\epsilon}$, 
and after integration over $u$ this becomes  $T^{\frac{3}{2} -\frac{\alpha}{2} +(1-3\alpha)\frac{k}{2} +\epsilon}$.   
The integral goes to zero as $\sigma_v$ becomes arbitrarily large, 
and as this is the dominant residue, this is the upper bound for the case $t' = t$, $\beta \le 1-\alpha$.
 
If $\beta > 1-\alpha$, the exponent is  
\begin{equation}
(\beta -2\alpha)\sigma_v +(\tfrac{\beta}{2} -\alpha)k+1+\epsilon.
\end{equation}
For $\alpha > \frac{\beta}{2}$ the dominant residue is again at $\sigma_v = \frac{1}{2}$ 
and the exponent becomes, after integration over $u$,
\begin{equation}
(\beta -2\alpha)\tfrac{1}{2} +(\tfrac{\beta}{2} -\alpha)k+1+ \alpha+\epsilon = 1+ \alpha - (2\alpha-\beta)\tfrac{k+1}{2}+ \epsilon.
\end{equation}
Thus the upper bound in the case $\beta > 1-\alpha$ and 
$2\alpha >\beta$ is $T^{ 1+ \alpha - (2\alpha-\beta)\frac{k+1}{2}+ \epsilon}$

Now, consider the case $t' = -t$.  
Here the sum is over $|r_j| < |t-r|$.  
If $t$ and $r$ are the same sign then just as above $\beta \le 1-\alpha$ and $|t-r|,|t+r| \ll T^{1-\alpha}$.  
We then end up with the identical bound as in \eqref{betasmall}.
If $r$ and $t$ have opposite signs then, $t'$ and $r$ have the same sign,  
$|t+r| = |T^\beta -|r| |\ll T^{1-\alpha}$, $|t-r| \ll T^\beta$ and $|t'+r| =  |T^\beta +|r||\ll T^\beta$.   
Thus the sum is over $|r_j| < T^\beta$.  If $\beta \le 1-\alpha$, 
then we obtain the same upper bound as in the case $t$ and $t'$ have the same sign, and $\beta \le 1-\alpha$.   
If $\beta >1-\alpha$ then applying Lemma~\ref{lem:sieve} to \eqref{OD-4}, with the constant $C= \beta$, 
we see that the expression in \eqref{OD-4} is bounded above by a constant times
\begin{equation}
T^{(2\beta-1-\alpha)\sigma_v+(\beta-\frac{1}{2}-\frac{\alpha}{2})k +\frac{1}{2} +\frac{\beta+\alpha}{2} + \epsilon}.
\end{equation}
The coefficients of $\sigma_v$, $k$ are negative when $\alpha > 2\beta-1$, 
and the integral will approach zero as $\sigma_v$ becomes arbitrarily large. 
The largest error contribution comes from the first pole, at $\sigma_v = \frac{1}{2}$, and after the $u$ integration, 
this is bounded above by a constant times 
\begin{equation}
T^{3\frac{\beta}{2} + \alpha +(\beta-\frac{1}{2}-\frac{\alpha}{2})k+ \epsilon}.
\end{equation}
This will be smaller than the main term, (which is on the order of $T^{1+\alpha + \epsilon}$) 
for sufficiently large $k$, depending on the size of $\alpha$.   
Specifically, we require $\alpha > 2\beta -1$.  
If we write $\alpha = 2\beta -1 + \delta$, with $\delta>0$,  and $\alpha = 2\beta -1 + \delta<1$, 
then the requirement for the main term to dominate is $k > \frac{3\beta - 2}{\delta}$.  
Finally, as usual the continuous part of the spectrum contributes a smaller error.  We omit the details.

To bound $\OD_{\Omega, \intg}^-(s, t)$ from above we refer to \eqref{e:OD-_Omegaint} and move the $v$ line of integration to $\sigma_v = -\frac{k}{2} + 1 - \epsilon$, passing over no poles in the process.     An analysis of the exponential pieces, via Stirling's formula (exactly as has been done several times previously) tells us that in this case there is exponential decay unless $|\gamma| \ge |t| \ge |r|$ where, as above, $u = \sigma_u + i\gamma$ and $v = \sigma_v + ir$.   Then an analysis of the polynomial parts, also via Stirling's formula, tells us that if $t' = \pm t$, and $s = \frac{1}{2} -it'$, then  
$\OD_{\Omega, \intg}^-(s, t) \ll T^{\epsilon}$ if $\beta \le 1-\alpha$ and  $\OD_{\Omega, \intg}^-(s, t) \ll 
T^{\beta+ \alpha-1 + \epsilon} \ll T^{\alpha+ \epsilon} $ if $\beta > 1-\alpha$.  

This completes the proof of the proposition.
\end{proof}

\subsection{Analysis of $M^-_{\Omega}(s, t)$}\label{ss:MOmega-}
%We have 
%Now we let 
%\begin{equation}
%M^-_{\Omega}(s, t) = \OD_{\Omega, \res}^-(s, t)
%\end{equation}
Recalling the formula for $M^-_{\Omega}(s, t)$ as given in \eqref{e:M-Omega_def} and \eqref{e:OD-_Omegares}, 
we will now prove the following lemma.
\begin{lemma}
For $\Re(s)=\frac{1}{2}$, we get
\begin{multline}\label{e:M-_Omega}
M^-_{\Omega}(s, t) 
%\OD_{\Omega, \res}^-(s, t) 
=  \frac{1}{2} \zeta(2-2s)
\frac{(2\pi)^{4s-2+2it}\cos(\pi it)}{\sin(\pi s)} 
\\ \times \bigg\{
% ir=1-s+it
% v=2s-1-it-k/2
\frac{(2\pi)^{-2it}}{\sin(\pi(s-it))} 
\zeta(1+2it)
\frac{1}{2} \sum_{\cuspa} \frac{\scrP_{\cuspa}(s, it; 1-s+it)}{\prod_{p\mid N}(1-p^{1-2s+2it})} 
\frac{L_\cuspa(3/2-s+it, f\times \bar{g})}{\zeta^{(N)}(3-2s+2it)}
H_0(-2s+1; h)
%\frac{1}{\pi^2} \int_{-\infty}^\infty  h(r) r \tanh(\pi r)
%\frac{\Gamma\left(ir-2s+1+it+\frac{k}{2}\right) \Gamma\left(-ir-2s+1+it+\frac{k}{2}\right)}
%{\Gamma\left(ir+it+\frac{k}{2}\right) \Gamma\left(-ir+it+\frac{k}{2}\right)}\; dr
% ir=1-s-it
% v=2s-1+it-k/2
\\ + 
\frac{(2\pi)^{2it}}{\sin(\pi(s+it))}
\zeta(1-2it)
\frac{1}{2}\sum_{\cuspa} \frac{\scrP_{\cuspa}(s, it; 1-s-it)}{\prod_{p\mid N}(1-p^{1-2s-2it})} 
\frac{L_\cuspa(3/2-s-it, f\times \bar{g})}{\zeta^{(N)}(3-2s-2it)}
H_0(-2s+1-2it; h)
%\frac{1}{\pi^2} \int_{-\infty}^\infty h(r) r \tanh(\pi r)
%\frac{\Gamma\left(ir-2s+1-it+\frac{k}{2}\right) \Gamma\left(-ir-2s+1-it+\frac{k}{2}\right)}
%{\Gamma\left(ir+it+\frac{k}{2}\right) \Gamma\left(-ir+it+\frac{k}{2}\right)}\;dr
%
\bigg\}. 
\end{multline}
Here $H_0(*; h)$ is given in \eqref{e:H0}. 
\end{lemma}

\begin{proof}
The proof is parallel to the proof of Lemma~\ref{lem:M+Omega} for $M^+_{\Omega}(s, t)$. 
and is consequently omitted. 
\end{proof} 

\section{Proof of theorems}
%\subsection{Proof of Theorem~\ref{thm:N=1}}
%Theorem~\ref{thm:N=1} is the special case of Theorem~\ref{thm:second_asymp}, with $N=1$. 

\subsection{Proof of Theorem~\ref{thm:second_main}}
The proof begins with the description of the spectral sum given in \eqref{e:secondmoment}:
\begin{equation}
S(s,t; f, g; h)  = M(s, t) + \OD^+(s, t) + \OD^-(s, t).
\end{equation}
Then in Proposition~\ref{prop:OD+_decomp}, $\OD^+(s, t)$ is decomposed as 
\begin{equation}
\OD^+(s, t) = M^+_{\Omega}(s, t) + \scrE(s, t),
\end{equation}
with $M^+_{\Omega}(s, t)$ given in \eqref{e:MOmega_explicit} and 
\begin{multline}
\scrE(s, t) = \OD^+_{\cusp, \intg}(s, t)+\OD^+_{\cusp, \res}(s, t) 
+ \OD^+_{\cont, \intg}(s, t) + \OD^+_{\cont, \res}(s, t)
\\ + \OD^+_{\Omega, \intg}(s, t) 
+ \sum_{\ell=1}^{\frac{k}{2}} \OD^+_{\Omega, \res, 1}(s, t; \ell) + \sum_{\ell=0}^{\frac{k}{2}} \OD^+_{\Omega, \res, 2}(s, t; \ell).
\end{multline}
Similarly, by \eqref{e:OD-_decomp} and \eqref{e:OD-Omega_decomp}, $\OD^-(s, t)$ is decomposed as 
\begin{equation}
\OD^-(s, t) = M^-_{\Omega}(s, t) + \OD^-_0(s, t)+ \OD^-_{\Omega, \intg}(s, t).
\end{equation}

The breakdown of the main term for $S(s,t; f, g; h)$ is given just above as
\begin{equation}
\scrM(s, t) := M(s, t) + M^+_{\Omega}(s, t) + M^-_{\Omega}(s, t). 
\end{equation}
Recalling \eqref{e:Mst}, \eqref{e:MOmega_explicit} and \eqref{e:M-_Omega},
\begin{multline}
\scrM (s, t)
= \zeta(2s) \zeta(1+2it)
H_0(0; h)
\prod_{p\mid N} \frac{(1-p^{-2s})(1-p^{-1-2it})}{1-p^{-2s-1-2it}} 
\frac{L(s+1/2+it, f\times \bar{g})}{\zeta(2s+1+2it)}
\\ + (2\pi)^{4it} \zeta(2s)\zeta(1-2it)
H_0(-2it; h)
N^{-2it} \prod_{p\mid N} \frac{(1-p^{-1})(1-p^{-2s})}{1-p^{-2s-1+2it}}
\frac{L(s+1/2-it, f\times \bar{g})}{\zeta(2s+1-2it)} 
\\ + 
(2\pi)^{4s-2}\zeta(2-2s)\zeta(1+2it)
\frac{1}{4}\frac{\cos(\pi it)-\cos(\pi(2s-it))}{\sin(\pi s)\sin(\pi(s-it))}
H_0(-2s+1; h)
\\ \times \sum_{\cuspa}\frac{\scrP_{\cuspa}(s, it; 1-s+it)}{\prod_{p\mid N}(1-p^{1-2s+2it})}
\frac{L_\cuspa(3/2-s+it, f\times \bar{g})}{\zeta^{(N)}(3-2s+2it)}
\\ + (2\pi)^{4s-2+4it}\zeta(2-2s)\zeta(1-2it)
\frac{1}{4}\frac{\cos(\pi it) - \cos(\pi(2s+it))}{\sin(\pi s)\sin(\pi(s+it))}
H_0(-2s+1-2it; h)
\\ \times 
\sum_{\cuspa} \frac{\scrP_{\cuspa}(s, it; 1-s-it)}{\prod_{p\mid N}(1-p^{1-2s-2it})}
\frac{L_\cuspa(3/2-s-it, f\times \bar{g})}{\zeta^{(N)}(3-2s-2it)}. 
\end{multline}
Since 
\begin{equation}
\cos(\pi it) -\cos(\pi (2s+it)) 
%= \cos(\pi(s+it-s)) - \cos(\pi(s+it +it))
%\\ = \cos(\pi(s+it)) \cos(\pi s) + \sin(\pi (s+it)) \sin(\pi s) 
%- \cos(\pi(s+it)) \cos(\pi it) + \sin(\pi(s+it)) \sin(\pi s)
%\\ 
= 2\sin(\pi(s+it)) \sin(\pi s), 
\end{equation}
and then by changing the ordering and by Corollary~\ref{cor:scrP_Eis_pm}
\begin{multline}
\scrM (s, t)
= \zeta(2s) \zeta(1+2it)
H_0(0; h)
\prod_{p\mid N} \frac{(1-p^{-2s})(1-p^{-1-2it})}{1-p^{-2s-1-2it}} 
\frac{L(s+1/2+it, f\times \bar{g})}{\zeta(2s+1+2it)}
\\ + (2\pi)^{4it} \zeta(2s)\zeta(1-2it)
H_0(-2it; h)
N^{-2it} \prod_{p\mid N} \frac{(1-p^{-1})(1-p^{-2s})}{1-p^{-2s-1+2it}}
\frac{L(s+1/2-it, f\times \bar{g})}{\zeta(2s+1-2it)} 
\\ + (2\pi)^{4s-2+4it}\zeta(2-2s)\zeta(1-2it)
H_0(-2s+1-2it; h)
\frac{1}{2}\sum_{\cuspa} \frac{\scrP_{\cuspa}(s, it; 1-s-it)}{\prod_{p\mid N}(1-p^{1-2s-2it})}
\frac{L_\cuspa(3/2-s-it, f\times \bar{g})}{\zeta^{(N)}(3-2s-2it)}
\\ + 
(2\pi)^{4s-2}\zeta(2-2s)\zeta(1+2it)
H_0(-2s+1; h)
N^{1-2s} \prod_{p\mid N}\frac{(1-p^{-1-2it})(1-p^{-1})}{(1-p^{-3+2s-2it})}
\frac{L(3/2-s+it, f\times \bar{g})}{\zeta(3-2s+2it)}.
\end{multline}

The polynomial $\frac{\scrP_{\cuspa}(s, it; 1-s- it)}{\prod_{p\mid N}(1-p^{1-2s- 2it})}$ 
for each cusp $\cuspa$ is given in Corollary~\ref{cor:scrP_Eis_pm}.
Following the parameterization for the cusps $\cuspa$ in \cite{Y19}, described in \S\ref{ss:factorization_scrL}, 
%$\cuspa=\frac{1}{ca}$ for $a\mid N$, $c\bmod{\gcd(a, N/a)}$, relatively prime to $\gcd(a, N/a)$ which is chosen as $\gcd(c, N)=1$. 
%Then 
% detail 
%\begin{equation}
%\frac{\scrP_{\cuspa}(s, it; 1-s+ it)}{\prod_{p\mid N}(1-p^{1-2s+ 2it})}
%= \begin{cases}
%2N^{1-2s} \prod_{p\mid N}(1-p^{-1-2it})(1-p^{-1}) & \text{ when } a=N, \\
%0 & \text{ otherwise}
%\end{cases}
%\end{equation}
%and 
\begin{multline}
\frac{\scrP_{\frac{1}{ca}}(s, it; 1-s-it)}{\prod_{p\mid N} (1-p^{1-2s-2it})}
\\ = 2N^{-\frac{1}{2}-s-it}
\left(\frac{a}{\gcd(a, N/a)}\right)^{-s+\frac{3}{2}-it}
\prod_{p\mid \frac{N}{a}} (1-p^{1-2s})(1-p^{-2it})
\prod_{\substack{p\mid a\\ p\nmid\frac{N}{a}}} (1-p^{-1})^2.
\end{multline}
With this description, we get \eqref{e:M2}.

\subsection{Proof of Theorem~\ref{thm:second_asymp}}
With the choice of the test function $h=h_{T, \alpha}$ as in \eqref{e:hdef}, 
the upper bounds for all the contributions to $\scrE^+(s, t)$ are given in
Propositions~\ref{prop:ub_OD+ccint},~\ref{prop:OD+res_upper}, and~\ref{prop:OD+Omega_upper},
and upper bounds for $\scrE^-(s, t)$ are given in Proposition~\ref {prop:OD-up}.

%
%With the choice of the test function $h=h_{T, \alpha}$ as in \eqref{e:hdef}, 
%by applying Stirling's formula, we get \eqref{e:H0_asymp} and \eqref{e:H0_j_asymp}.

Now it is remained the bounds for the main term $\scrM(s, t)$ for $s=\frac{1}{2}\pm it$. 

When $f\neq g$, taking $s=\frac{1}{2}-it$ in $\scrM(s, t)$  \eqref{e:M2}, we get
\begin{multline}\label{e:scrM_fneqg_1/2-it}
\scrM (1/2-it, t)
= \zeta(1-2it) \zeta(1+2it)
H_0(0; h)
\prod_{p\mid N} \frac{(1-p^{-1+2it})(1-p^{-1-2it})}{1-p^{-2}} 
\frac{L(1, f\times \bar{g})}{\zeta(2)}
\\ + \zeta(1+2it)\zeta(1-2it)
H_0(0; h)
\frac{1}{N}
\sum_{\cuspa=\frac{1}{ca}} 
\left(\frac{a}{\gcd(a, N/a)}\right)
\prod_{p\mid \frac{N}{a}} \frac{(1-p^{2it})(1-p^{-2it})}{1-p^{-2}}
\prod_{\substack{p\mid a\\ p\nmid\frac{N}{a}}} \frac{1-p^{-1}}{1+p^{-1}}
\frac{L_\cuspa(1, f\times \bar{g})}{\zeta(2)}
\\ + (2\pi)^{4it} \zeta(1-2it)\zeta(1-2it) H_0(-2it; h)
N^{-2it} \prod_{p\mid N} \frac{(1-p^{-1})(1-p^{-1+2it})}{1-p^{-2+4it}}
\frac{L(1-2it, f\times \bar{g})}{\zeta(2-4it)} 
\\ + 
(2\pi)^{-4it}\zeta(1+2it)\zeta(1+2it)
H_0(2it; h)
N^{2it} \prod_{p\mid N}\frac{(1-p^{-1-2it})(1-p^{-1})}{(1-p^{-2-4it})}
\frac{L(1+2it, f\times \bar{g})}{\zeta(2+4it)}.
\end{multline}
Here we use the parameterization for the cusps $\cuspa$ as in \cite{Y19}, described in \S\ref{ss:factorization_scrL}.  
%$\cuspa=\frac{1}{ca}$ for $a\mid N$, $c\bmod{\gcd(a, N/a)}$, relatively prime to $\gcd(a, N/a)$ which is chosen as $\gcd(c, N)=1$. 
Similarly, taking $s=\frac{1}{2}+it$ in $\scrM(s, t)$ \eqref{e:M2}, 
\begin{multline}\label{e:scrM_fneqg_1/2+it}
\scrM (1/2+it, t)
= 
2(2\pi)^{4it} \zeta(1+2it)\zeta(1-2it)
H_0(-2it; h)
N^{-2it} \prod_{p\mid N} \frac{(1-p^{-1})(1-p^{-1-2it})}{1-p^{-2}}
\frac{L(1, f\times \bar{g})}{\zeta(2)} 
%\\ + 
%(2\pi)^{4it}\zeta(1-2it)\zeta(1+2it)
%H_0(-2it; h)
%N^{-2it} \prod_{p\mid N}\frac{(1-p^{-1-2it})(1-p^{-1})}{(1-p^{-2})}
%\frac{L(1, f\times \bar{g})}{\zeta(2)}
%
\\ + \zeta(1+2it) \zeta(1+2it)
H_0(0; h)
\prod_{p\mid N} \frac{(1-p^{-1-2it})(1-p^{-1-2it})}{1-p^{-2-4it}} 
\frac{L(1+2it, f\times \bar{g})}{\zeta(2+4it)}
\\ + (2\pi)^{8it}\zeta(1-2it)\zeta(1-2it)
H_0(-4it; h)
N^{-1-2it}
\\ \times \sum_{\cuspa=\frac{1}{ca}} 
\left(\frac{a}{\gcd(a, N/a)}\right)^{1-2it}
\prod_{p\mid \frac{N}{a}} \frac{(1-p^{-2it})^2}{1-p^{-2+4it}}
\prod_{\substack{p\mid a\\ p\nmid\frac{N}{a}}} \frac{(1-p^{-1})^2}{1-p^{-2+4it}}
\frac{L_\cuspa(1-2it, f\times \bar{g})}{\zeta(2-4it)}.
\end{multline}

Now we assume that $f=g$. 
For each cusp $\cuspa$, since 
\begin{equation}
\Res_{s=1} E_\cuspa(z, s) = \frac{1}{\vol(\Gamma_0(N)\bsl \HH)}, 
\end{equation}
for any $\cuspa$ for $\Gamma_0(N)$, when $f=g$, we get
\begin{equation}
\Res_{s=1}L_\cuspa(s, f\times\bar{f}) = \Res_{s=1}L(s, f\times\bar{f}). 
\end{equation}
To take $s=\frac{1}{2}-it$ in $\scrM(s, t)$ \eqref{e:M2}, we first note that 
\begin{multline}
\sum_{\cuspa=\frac{1}{ca}} 
\left(\frac{a}{\gcd(a, N/a)}\right)
\prod_{p\mid \frac{N}{a}} \frac{(1-p^{2it})(1-p^{-2it})}{1-p^{-2}}
\prod_{\substack{p\mid a\\ p\nmid\frac{N}{a}}} \frac{(1-p^{-1})^2}{1-p^{-2}}
\\ = \frac{1}{\prod_{p\mid N}(1-p^{-2})}
\sum_{a\mid N} a \prod_{p\mid \gcd(a, N/a)}(1-p^{-1}) 
\prod_{p\mid \frac{N}{a}} (1-p^{2it})(1-p^{-2it})
\prod_{\substack{p\mid a\\ p\nmid\frac{N}{a}}} (1-p^{-1})^2
\\ = \frac{1}{\prod_{p\mid N}(1-p^{-2})}
\prod_{p\mid N}\bigg\{ (1-p^{2it})(1-p^{-2it}) 
+ (1-p^{2it})(1-p^{-2it}) (1-p^{-1}) \sum_{k=1}^{\ord_p(N)-1} p^k 
+ (1-p^{-1})^2 p^{\ord_p(N)}\bigg\}
\\ =N\prod_{p\mid N} \frac{(1-p^{-1+2it})(1-p^{-1-2it})}{(1-p^{-2})}.
\end{multline}
So the poles of the Rankin-Selberg convolution cancel and we get
\begin{multline}\label{e:scrM_f=g_1/2-it}
\scrM(1/2-it, t)
= 
- \frac{\zeta(1-2it) \zeta(1+2it)}{\zeta(2)}
\left.\frac{d}{ds}H_0(-2s+1-2it; h)\right|_{s=\frac{1}{2}-it}
\prod_{p\mid N} \frac{(1-p^{-1+2it})(1-p^{-1-2it})}{1-p^{-2}} 
\\ \times \Res_{s=1}L(s, f\times\bar{f})
\\ - \frac{\zeta(1-2it) \zeta(1+2it)}{\zeta(2)}
H_0(0; h)N^{-1} \prod_{p\mid N}\frac{1}{1-p^{-2}}
\\ \times \sum_{a\mid N} \bigg\{a\prod_{p\mid \frac{N}{a}} (1-p^{2it})(1-p^{-2it}) \prod_{\substack{p\mid a\\ p\nmid a}}(1-p^{-1})^2 
\prod_{p\mid \gcd(a, N/a)}(1-p^{-1}) 
\bigg(-\log \left(\frac{a}{\gcd(a, N/a)}\right) + \sum_{p\mid \frac{N}{a}} 2\log p\bigg) \bigg\} 
\\ \times  \Res_{s=1}L(s, f\times\bar{f})
\\ + \frac{\zeta(1-2it) \zeta(1+2it)}{\zeta(2)}
H_0(0; h)
\prod_{p\mid N} \frac{(1-p^{-1+2it})(1-p^{-1-2it})}{1-p^{-2}} 
\\ \times \bigg\{2\frac{\zeta'(1-2it)}{\zeta(1-2it)}+2\frac{\zeta'(1+2it)}{\zeta(1+2it)}  
- 4\frac{\zeta'(2)}{\zeta(2)} 
-4\log(2\pi)
+ 2\sum_{p\mid N} \frac{\log p}{1-p^{-1+2it}} + \log N
\bigg\} 
\Res_{s=1}L(s, f\times\bar{f})
\\ + 2\frac{\zeta(1-2it) \zeta(1+2it)}{\zeta(2)}
H_0(0; h)
\prod_{p\mid N} \frac{(1-p^{-1+2it})(1-p^{-1-2it})}{1-p^{-2}} 
\left.\frac{d}{ds}L(s, f\times \bar{f})\right|_{s=1}
\\ + (2\pi)^{4it} \zeta(1-2it)\zeta(1-2it) H_0(-2it; h)
N^{-2it} \prod_{p\mid N} \frac{(1-p^{-1})(1-p^{-1+2it})}{1-p^{-2+4it}}
\frac{L(1-2it, f\times \bar{f})}{\zeta(2-4it)} 
\\ + 
(2\pi)^{-4it}\zeta(1+2it)\zeta(1+2it)
H_0(2it; h)
N^{2it} \prod_{p\mid N}\frac{(1-p^{-1-2it})(1-p^{-1})}{(1-p^{-2-4it})}
\frac{L(1+2it, f\times \bar{f})}{\zeta(2+4it)}.
\end{multline}
Here, recalling \eqref{e:H0}, 
\begin{multline}\label{e:H0_1der-}
\left.\frac{d}{ds} H_0(-2s+1-2it; h)\right|_{s=\frac{1}{2}-it} 
\\ = -2 \frac{1}{\pi^2} \int_{-\infty}^\infty h(r) r \tanh(\pi r) 
\bigg(\psi\left(ir+it+\frac{k}{2}\right) +\psi\left(-ir+it+\frac{k}{2}\right)\bigg)\; dr, 
\end{multline}
where $\psi(x) = \frac{\Gamma'}{\Gamma}(x)$ is the digamma function. 

Taking $s=\frac{1}{2}+it$ in $\scrM(s, t)$ \eqref{e:M2}, 
\begin{multline}\label{e:scrM_f=g_1/2+it}
\scrM (1/2+it, t)
\\ = - (2\pi)^{4it} \frac{\zeta(1+2it) \zeta(1-2it)}{\zeta(2)} 
\left.\frac{d}{ds} H_0(-2s+1; h)\right|_{s=\frac{1}{2}+it} 
N^{-2it} \prod_{p\mid N} \frac{1-p^{-1-2it}}{1+p^{-1}}
\Res_{s=1}L(s, f\times\bar{f}) 
\\ + (2\pi)^{4it} \frac{\zeta(1+2it) \zeta(1-2it)}{\zeta(2)} H_0(-2it; h) N^{-2it} \prod_{p\mid N} \frac{(1-p^{-1-2it})}{1+p^{-1}}
\\ \times \bigg\{2\frac{\zeta'(1+2it)}{\zeta(1+2it)}+ 2\frac{\zeta'(1-2it)}{\zeta(1-2it)}- 4\frac{\zeta'(2)}{\zeta(2)} 
- 4\log(2\pi) + 2\log N
+ 2\sum_{p\mid N} \log p \frac{p^{-1-2it}}{(1-p^{-1-2it})}
\bigg\} 
\Res_{s=1}L(s, f\times\bar{f})
\\ + 2(2\pi)^{4it} \frac{\zeta(1+2it) \zeta(1-2it)}{\zeta(2)} H_0(-2it; h) N^{-2it} \prod_{p\mid N} \frac{(1-p^{-1-2it})}{1+p^{-1}}
\left.\frac{d}{ds}L(s, f\times\bar{f})\right|_{s=1} 
\\ + \zeta(1+2it) \zeta(1+2it)
H_0(0; h)
\prod_{p\mid N} \frac{(1-p^{-1-2it})}{1+p^{-1-2it}} 
\frac{L(1+2it, f\times \bar{f})}{\zeta(2+4it)}
\\ + (2\pi)^{8it}\zeta(1-2it)\zeta(1-2it)
H_0(-4it; h)
\\ \times N^{-1-2it}
\sum_{\cuspa=\frac{1}{ca}} 
\left(\frac{a}{\gcd(a, N/a)}\right)^{1-2it}
\prod_{p\mid \frac{N}{a}} \frac{(1-p^{-2it})^2}{1-p^{-2+4it}}
\prod_{\substack{p\mid a\\ p\nmid\frac{N}{a}}} \frac{(1-p^{-1})^2}{1-p^{-2+4it}}
\frac{L_\cuspa(1-2it, f\times \bar{f})}{\zeta(2-4it)}.
\end{multline}
Similarly, recalling \eqref{e:H0}, 
\begin{multline}\label{e:H0_1der+}
\left.\frac{d}{ds} H_0(-2s+1; h)\right|_{s=\frac{1}{2}+it} 
= -2 \frac{1}{\pi^2} \int_{-\infty}^\infty h(r) r \tanh(\pi r) 
\bigg(\psi\left(ir-it+\frac{k}{2}\right) +\psi\left(-ir-it+\frac{k}{2}\right)\bigg)
\\ \times \frac{\Gamma\left(ir-it+\frac{k}{2}\right) \Gamma\left(-ir-it+\frac{k}{2}\right)}
{\Gamma\left(ir+it+\frac{k}{2}\right) \Gamma\left(-ir+it+\frac{k}{2}\right)} \; dr. 
\end{multline}

To calculate an upper bound for $\scrM(s, t)$, note first that the main part of it is given by the asymptotic for $H_0(ix;h)$, 
(see \eqref{e:H0_asymp}),  stated more simply as an upper bound, and an upper bound for the derivatives of $H_0(ix; h)$, which are given below.
The remaining pieces are easily estimated using the trivial upper bounds 
$L_\cuspa(1\pm 2it, f\times \bar{g}) \ll (1+|t|)^\epsilon$ when $f \ne g$ or $f=g$, 
and $\zeta(1\pm 2t) \ll(1 +|t|)^\epsilon$.

We now estimate $\scrM(1/2, 0)$ by taking $t=0$ in \eqref{e:scrM_fneqg_1/2-it} and \eqref{e:scrM_f=g_1/2-it}. 
When $f\neq g$, there are double poles in $\scrM(1/2-it, t)$ \eqref{e:scrM_fneqg_1/2-it}  at $t=0$. 
The poles cancel each other and the main contribution in the remaining term comes from 
\begin{equation}
-\frac{1}{4} \prod_{p\mid N} \frac{(1-p^{-1})}{1+p^{-1}}
\frac{L(1, f\times \bar{g})}{\zeta(2)} 
\bigg(\left.\frac{d^2}{dt^2}H_0(-2it; h)\right|_{t=0}
+ \left.\frac{d^2}{dt^2} H_0(2it; h)\right|_{t=0}\bigg).
\end{equation}
Recalling \eqref{e:H0}, 
\begin{equation}
\left.\frac{d^2}{dt^2} H_0(-2it; h)\right|_{t=0}
= -4 \frac{1}{\pi^2} \int_{-\infty}^\infty h(r) r\tanh(\pi r) 
\bigg(\psi\left(ir+\frac{k}{2}\right)+ \psi\left(-ir+\frac{k}{2}\right)\bigg)^2 \; dr
\end{equation}
and 
\begin{multline}
\left.\frac{d^2}{dt^2} H_0(2it; h)\right|_{t=0} 
= -4 \frac{1}{\pi^2} \int_{-\infty}^\infty h(r) r\tanh(\pi r) 
\bigg(\psi\left(ir+\frac{k}{2}\right)+ \psi\left(-ir+\frac{k}{2}\right)\bigg)^2 \; dr
\\ -8 \frac{1}{\pi^2} \int_{-\infty}^\infty h(r) r\tanh(\pi r) 
\bigg(\psi^{(1)} \left(ir+\frac{k}{2}\right)+ \psi^{(1)} \left(-ir+\frac{k}{2}\right)\bigg) \; dr
\end{multline}
Since, for $m\geq 1$, 
\begin{equation}\label{e:digamma_der}
\psi^{(m)}(z)\sim (-1)^{m+1}\frac{1}{z^m}, 
\end{equation}
when $f\neq g$ and $t=0$, the main contribution is 
\begin{equation}\label{e:main_t=0_fneqg}
2 \prod_{p\mid N} \frac{1-p^{-1}}{1+p^{-1}}
\frac{L(1, f\times \bar{g})}{\zeta(2)} 
\frac{1}{\pi^2} \int_{-\infty}^\infty h(r) r\tanh(\pi r) 
\bigg(\psi\left(ir+\frac{k}{2}\right)+ \psi\left(-ir+\frac{k}{2}\right)\bigg)^2 \; dr. 
\end{equation}

Similarly, when $f=g$, there are triple poles (because $L(1\pm 2it, f\times\bar{f})$ has a pole at $t=0$) 
in $\scrM(1/2-it, t)$ \eqref{e:scrM_f=g_1/2-it} at $t=0$. 
Again the poles are canceled each other and the main contribution comes from 
\begin{equation}
\frac{1}{8i} 
\prod_{p\mid N} \frac{1-p^{-1}}{1+p^{-1}} 
\frac{\Res_{s=1}L(1, f\times\bar{f})}{\zeta(2)}
\bigg(\left.\frac{d^3}{dt^3}H_0(-2it; h)\right|_{t=0} - \left.\frac{d^3}{dt^3} H_0(2it; h)\right|_{t=0} \bigg)
\end{equation}
Again, by \eqref{e:digamma_der}, when $f=g$ and $t=0$, the main contribution is 
\begin{equation}\label{e:main_t=0_f=g}
2\prod_{p\mid N} \frac{1-p^{-1}}{1+p^{-1}} 
\frac{\Res_{s=1}L(1, f\times\bar{f})}{\zeta(2)}
\frac{1}{\pi^2} \int_{-\infty}^\infty h(r) r\tanh(\pi r) 
\bigg(\psi\left(ir+\frac{k}{2}\right)+ \psi\left(-ir+\frac{k}{2}\right)\bigg)^3 \; dr.
\end{equation}

Take $h=h_{T, \alpha}$, as in \eqref{e:hdef}. 
By the recurrence relation and asymptotic expansion of the digamma function, we get 
\begin{equation}
\bigg(\psi\left(ir+\frac{k}{2}\right)+ \psi\left(-ir+\frac{k}{2}\right)\bigg)^d
\sim 2^d (\log |r|)^d,
\end{equation}
from which \eqref{e:cfg_N} follows.
\qed

%: Reference
\thispagestyle{empty}
{\footnotesize
\nocite{*}
\bibliographystyle{amsalpha}
\bibliography{reference_SDDS}
}
\end{document}